\documentclass[11pt,leqno]{amsart}
\usepackage{amssymb,graphicx,color} 
\usepackage{amsmath,amsfonts}
\DeclareMathOperator{\spn}{span}
\usepackage{mathrsfs}
\usepackage{tikz}
\usepackage{subfig}

\begin{document}
\theoremstyle{plain}
\newtheorem{thm}{Theorem}[section]
\newtheorem{prop}[thm]{Proposition}
\newtheorem{lem}[thm]{Lemma}
\newtheorem{cor}[thm]{Corollary}
\newtheorem{deft}[thm]{Definition}
\newtheorem{hyp}{Assumption}
\newtheorem*{KSU}{Theorem (Kenig, Sj\"ostrand and Uhlmann)}

\theoremstyle{definition}
\newtheorem{rem}[thm]{Remark}
\numberwithin{equation}{section}
\newcommand{\eps}{\varepsilon}
\renewcommand{\d}{\partial}
\newcommand{\dd}{\mathrm{d}}
\newcommand{\e}{\mathrm{e}}
\newcommand{\re}{\mathop{\rm Re} }
\newcommand{\im}{\mathop{\rm Im}}
\newcommand{\ch}{\mathop{\rm ch}}
\newcommand{\R}{\mathbf{R}}
\newcommand{\C}{\mathbf{C}}
\renewcommand{\H}{\mathbf{H}} 
\newcommand{\N}{\mathbf{N}} 
\newcommand{\D}{\mathcal{C}^{\infty}_0} 
\newcommand{\supp}{\mathop{\rm supp}}
\hyphenation{pa-ra-met-ri-zed}
\title[]{Optimal stability estimates for a Magnetic Schr\"odinger operator with local data}
\author[]{Leyter Potenciano-Machado }
\address{Departamento de Matem\'aticas, Universidad Aut\'onoma de Madrid, Campus de Cantoblanco, 28049 Madrid, Spain}
\email{leyter.potenciano@uam.es}
%\address{Departamento de Matem\'aticas, Universidad Aut\'onoma de Madrid, Campus de Cantoblanco, 28049 Madrid, Spain}
%\email{alberto.ruiz@uam.es}
\begin{abstract}
In this paper we prove identifiability and stability estimates for a local-data inverse boundary value problem for a magnetic Schr\"odinger operator in dimension $n\geq 3$. We assume that the inaccessible part of the boundary is part of a hyperplane. We improve the identifiability result obtained by Krupchyck, Lassas and Uhlmann \cite{KLU} and also derive the corresponding stability estimates. We obtain $\log$-estimates for magnetic  and electric potentials. 

%Also that the knowledge of the Cauchy data set  is on the complement of this plane for functions supported on the complement of this plane as well. We improve the identifiability result obtained by Krupchyck, Lassas and Uhlmann \cite{KLU} and derive the corresponding stability estimates. We obtain $\log$-estimates for magnetic  and electric potentials. 

% magnetic Schr\"odinger operator with local data in dimension $n\geq 3$. We consider that the inaccessible part of the boundary is part of a plane and that the knowledge of the Cauchy data set  is on the complement of this plane for functions supported on the complement of this plane as well. We improve the identifiability result obtained by Krupchyck, Lassas and Uhlmann \cite{KLU} and derive the corresponding stability estimates. We obtain $\log$-estimates for magnetic  and electric potentials. 

\end{abstract}
\maketitle
\setcounter{tocdepth}{1} 
\tableofcontents

%\section{Tareas} 

%\begin{enumerate}
%\item Hacer la observaci—n de $\xi_0$, con su parte real e imaginaria del teorema de existencia de soluciones, mientras KU consideran que su norma real e imaginaria sea igual a $1$; nosotros consideramos que solo basta que la parte real e imaginaria sean iguales. Esta condici—n es suficiente para tener un operador tipo $\overline{\partial}$.
%\end{enumerate}

\section{Introduction}\label{introz}

Let $\Omega \subset \mathbb{R}^n $ be an open bounded set with boundary denoted by $\partial\Omega$. We consider the following magnetic Schr\"odinger operator:
\begin{equation}
\mathcal{L}_{A,q}(x, D) := \sum_{j=1}^{n}\left( D_j + A_j(x)  \right)^2 + q(x) = D^2 + A\cdot D + D\cdot A + A^2 +q,
\end{equation}
where $D=-i \nabla$, $A=\left(  A_j \right)_{j=1}^{n}\in L^\infty(\Omega; \mathbb{C}^n) $ is a magnetic potential and $q\in L^{\infty}\left( \Omega; \mathbb{C} \right) $ is an electric potential.  Notice that $D^2=-\Delta$. Here $A^2$ is defined by $A^2=\sum_{k=1}^n A_j^2$. More precisely, we are considering the operator $\mathcal{L}_{A,q}: H^1(\Omega) \rightarrow H^{-1}(\Omega)$  in a weak sense as follows:
\[
\left \langle \mathcal{L}_{A,q} u, v  \right \rangle = \int_{\Omega} Du\cdot \overline{Dv} + A\cdot (Du\overline{v}+u\overline{Dv})+ (A^2+q)u\overline{v},
\]
where $H^{-1}(\Omega)$ denotes the dual space of $H^{1}_0(\Omega)$. The inverse boundary value problem (IBVP) under consideration in this paper is to recover information  (inside $\Omega$)  about the magnetic and electric potentials from measurements on subsets of the boundary. Roughly speaking, we divide the boundary $\partial\Omega$ in two subsets $\Gamma_0$ and $\Gamma:=\partial\Omega\setminus \Gamma_0$. We shall call $\Gamma_0$ the inaccessible part of the boundary and $\Gamma$ the accessible part. Thus, when $\Gamma_0$ or $\Gamma$ are not equal to $\partial\Omega$, we say that  the IBVP has local data. In this case, in a smooth setting over the domain $\Omega$ and the magnetic potential $A$, and by assuming that zero is not a Dirichlet eigenvalue of $\mathcal{L}_{A,q}$, we can define the local Dirichlet-Neumann map (DN map for short) by
\begin{equation}\label{ldnmap1}
\begin{aligned}
\Lambda_{A,q}^\Gamma:& H^{\frac{1}{2}}_\Gamma(\partial \Omega)\rightarrow H^{-\frac{1}{2}}(\partial \Omega)\\
& f \rightarrow (\partial_\nu + iA\cdot \nu)u|_\Gamma.
\end{aligned}
\end{equation}
Here, by abuse of notation $|_\Gamma$ denotes the local trace map of functions in $H^{1}(\Omega)$ onto the accesible part $\Gamma$ of the boundary . The vector $\nu$ is the exterior unit normal of $\partial\Omega$, the set $H_\Gamma^{\frac{1}{2}}(\partial \Omega)$ consists of all $f\in H^{\frac{1}{2}}(\partial \Omega)$ such that  with $\supp f \subset \overline{\Gamma}$, which we describe as ``support constraint on $\Gamma$"; and $u\in H^{1}(\Omega)$ is the unique solution of the following Dirichlet problem:
 \begin{equation}
\label{artif3}
     \begin{cases}
            \mathcal{L}_{A,q}\, u=0  & \text{ in } \Omega \\ u|_{\d \Omega} = f.
     \end{cases}
\end{equation}
In this case, the local boundary data $B_{A,q}^\Gamma$ is given by the graph of $\Lambda^\Gamma_{A,q}$:
\begin{equation}\label{–ki}
B_{A,q}^\Gamma= \left\{  ( u|_{\partial \Omega}, \Lambda^\Gamma_{A,q} (u|_{\partial\Omega}))\in  H^{\frac{1}{2}}_\Gamma(\partial \Omega)\times H^{-\frac{1}{2}}(\partial \Omega)  :  u\in H^{1}(\Omega), \, \mathcal{L}_{A,q}u=0   \right\}.
\end{equation}
The local identifiability problem associated to the magnetic Schr\"odinger equation (\ref{artif3}) can be formulated as follows: Given two magnetic potentials $A_1, A_2$ and two electric potentials $q_1, q_2$ such that $\Lambda^\Gamma_{A_1, q_1}=\Lambda^\Gamma_{A_2, q_2}$, does it follow that $A_1=A_2$ and $q_1=q_2$ in $\Omega$? Unfortunately the answer is no, because there exists a gauge invariance of the local DN map: if $\varphi\in C^1(\overline{\Omega})$ is a real valued function with $\varphi |_{\partial\Omega}=0$, then $\Lambda_{A,q}^{\Gamma}= \Lambda_{A+\nabla \varphi, q}^{\Gamma}$. Hence, for the identifiability problem, we only  expect to prove that the magnetic fields coincide, that is $dA_1=dA_2$, and also the electric potentials coincide, that is $q_1=q_2$. Actually, it was proved by Krupchyk, Lassas and Uhlmann\cite{KLU}, assuming that the inaccessible part of the boundary $\Gamma_0$ is contained in a hyperplane, the magnetic potential $A\in W^{1,\infty}(\Omega)$ and $q\in L^\infty(\Omega)$, see Theorem $1.4$ in \cite{KLU} for more details. Here we consider the magnetic potential $A:=(A^{(1)}, A^{(2)}, \ldots, A^{(n)})$ as a 1-form as follows:
\[
A=\sum_{j=1}^{n}A^{(j)}dx_j \quad \mbox{and} \quad  dA= \sum_{1\leq j<k\leq n}\left( \partial_{x_j}A^{(k)}-\partial_{x_k}A^{(j)}  \right) d_{x_j}\wedge d_{x_k}.
\]

The man in goal of this paper is to improve the previous identifiability result \cite{KLU} and also derive the corresponding stability estimates. A priori,  we are not assuming any smoothness condition over the magnetic potentials and the domain $\Omega$, except that the inaccessible part of the boundary $\Gamma_0$ is also contained in a hyperplane. More precisely, we assume that:
\begin{equation}\label{shapee}
\Omega \subset \left\{x\in \mathbb{R}^n \, : \, x_n>0  \right\} \quad \mbox{and} \quad \Gamma_0 = \partial\Omega \cap \left\{x\in \mathbb{R}^n \, : \, x_n=0  \right\} \neq \emptyset.
\end{equation}
and throughout this paper, unless otherwise indicated, the domain $\Omega$ and the inaccessible part of the boundary $\Gamma_0$ satisfy (\ref{shapee}). Assuming that zero is not a Dirichlet eigenvalue of $\mathcal{L}_{A,q}$ might be unnatural, because to be an eigenvalue depends strongly on the coefficients of the operator $\mathcal{L}_{A,q}$, which at the same time depends on the magnetic potential $A$ and the electric potential $q$, just the unknown parameters from which we want to obtain information. In our case we leave out this eigenvalue assumption. In this lack of smoothness and eigenvalue assumptions, we could not ensure, a priori, the existence and uniqueness of solutions for (\ref{artif3}). As a consequence, the local DN map $\Lambda^\Gamma_{A,q}$ (see \ref{ldnmap1}) is not well-defined and so we need to use instead the local linear map $N^\Gamma_{A,q}$. Also, since the magnetic potentials and  $\Omega$ are not smooth, the local trace map $|_\Gamma$ of functions on the boundary has no sense as defined in (\ref{ldnmap1}) and so we extend its definition by the boundary local map $T_r^\Gamma$. Finally the boundary local data $B^\Gamma_{A,q}$ (see \ref{–ki}) will be replaced by the local Cauchy data set $C^{\Gamma}_{A,q}$, which is defined by
\begin{equation}\label{lcdsetz}
C_{A,q}^\Gamma= \left \{( T_r^\Gamma u  , N_{A,q}^\Gamma (T_r^\Gamma u ) ):  u\in H^1(\Omega, \Gamma),\,  \mathcal{L}_{A,q}u=0 \,\,  \mbox{in}\,\,  \Omega  \right \},
\end{equation}
where, roughly speaking, the space $H^{1}(\Omega, \Gamma)$ denotes all the functions in $H^1(\Omega)$  vanishing on the inaccessible part of the boundary $\Gamma_0$. For expository convenience, the precise definition of the local boundary map $T^\Gamma_r$, the local linear map $N^\Gamma_{A,q}$ and the space $H^1(\Omega, \Gamma)$ will be given in section \ref{prel}.\\

%The local IBVP under consideration in this article is to recover the information (in $\Omega$) about the magnetic and electric potential, $A$ and $q$, from the knowledge of the local Cauchy data set $C^\Gamma_{A,q}$. In particular, 

Thus, in our non regular framework, the identifiability problem associated to the magnetic Schr\"odinger equation (\ref{artif3}) can be formulated as follows: Given two magnetic potentials $A_1, A_2$ and two electric potentials $q_1, q_2$ such that $C^\Gamma_{A_1, q_1}=C^\Gamma_{A_2, q_2}$, does it follow that $A_1=A_2$ and $q_1=q_2$ in $\Omega$? Analogously to the local DN map $\Lambda^\Gamma_{A,q}$, the answer is no, because there also exists a gauge invariance of the local Cauchy data sets: if $\varphi$ is a real valued Lipschitz continuous function on $\overline{\Omega}$ with $\varphi |_{\partial\Omega}=0$, then $C_{A,q}^{\Gamma}= C_{A+\nabla \varphi, q}^{\Gamma}$ (see \cite{KU} for full data case). Hence, we only  expect to prove that $dA_1=dA_2$ and $q_1=q_2$. So our identifiability result is stated as follows.

%Following \cite{CP}, we introduce a notion of pseudo-distance between local-data Cauchy sets. Let $A_1, A_2\in L^{\infty}(\Omega; \mathbb{C}^n)$ be two magnetic potentials and let $q_1,q_2\in L^{\infty}(\Omega)$ be two electric potentials. Given $(f, g)\in C^\Gamma_{A_j,q_j}$ with $j=1,2$, we set
%\begin{align*}
%&I( (f ,g ); C^\Gamma_{A_k,q_k} )\\
%&\qquad  = \underset{(f_k,g_k)\in  C^\Gamma_{A_k,q_k} }{\inf} \left [  \left \| f-f_k \right \|_{H^1(\Omega,\Gamma)/ H^1_0(\Omega)} + \left \| g-g_k \right \|_{(H^1(\Omega, \Gamma)/ H^1_0(\Omega))^*}  \right ],
%\end{align*}
%with $k=1,2$. Then the pseudo-distance between $C_{A_1, q_1}^\Gamma$ and $C_{A_2,q_2}^\Gamma$ is defined by
%\begin{equation}\label{distanciaa}
% \mbox{dist} (C_{A_1, q_1}^\Gamma, C_{A_2,q_2}^{\Gamma}  ) = \underset{\underset{j\neq k}{j,k=1,2}}{\max}  \underset{   \underset{ \left \| f_j \right \|_{H^1(\Omega,\Gamma)/ H^1_0(\Omega)}=1 }{(f_j,g_j)\in C_{A_j, q_j}^\Gamma}    }{\sup} I( (f_j,g_j); C^\Gamma_{A_k,q_k} ).
%\end{equation}
%We first state our result concerning to the identifiability for the magnetic field and electric potential.
\begin{thm}\label{identibiafgteru}
Let $\Omega\subset \mathbb{R}^n$ be a bounded open set. Let $A_1, A_2\in L^{\infty}(\Omega; \mathbb{C}^n)$ be two magnetic potentials and $q_1, q_2\in L^{\infty}(\Omega;\mathbb{C})$ be two electric potentials. If $C_{A_1,q_1}^\Gamma=C_{A_2, q_2}^\Gamma$ then $dA_1=dA_2$ and $q_1=q_2$ in $\Omega$.
\end{thm}

If is well known that, in order to obtain stability estimates for identifiability results, one needs a priori bounds for the unknown parameters in order to control their oscillations. In this way, we introduce the notion of the admissible class of the magnetic and electric potentials, $\mathcal{A}(\Omega, M, s)$ and $\mathcal{Q}(\Omega, M, s)$, respectively. Notice that in a regular framework, the proximity of the local boundary data are naturally given by the norm operator $\left \|  \Lambda^\Gamma_{A_1, q_1}- \Lambda^\Gamma_{A_2, q_2} \right \|_{{H^{1/2}_\Gamma}(\partial\Omega)\rightarrow H^{-1/2}(\partial\Omega)}$. On the contrary, since in our case the local DN map is not well-defined, we introduce a notion of pseudo-distance between local-data Cauchy sets denoted by $\mbox{dist}(\, \cdot\, ,\, \cdot\,)$. Again, for expository convenience, the precise definition of $\mathcal{A}(\Omega, M, s)$, $\mathcal{Q}(\Omega, M, s)$ and $\mbox{dist}(\, \cdot\, ,\, \cdot\,)$ will be given in section \ref{prel}.\\

Finally, for given $E\subset \mathbb{R}^n$ a bounded open set and any function $h: E \rightarrow \mathbb{C}$ (or $\mathbb{C}^n$),  we denote by $\chi_E h$ the extension by zero of $h$ out of $E$. We can now formulate our stability results.
\begin{thm}[Stability for the magnetic field]\label{SMP}
Let $\Omega \subset \mathbb{R}^n$ be a bounded open set. Consider two constants $M>0$ and $s\in \left ( 0,1/2 \right )$. Then there exist $C>0$ (depending on $n, \Omega, M, s, \left \|  q_1\right \|_{L^\infty}, \left \|  q_2\right \|_{L^\infty}$) and an universal constant $\lambda\in \left( 0,1\right)$ such that the following estimate 
% \left \|d \left [  \chi_\Omega (A_1-A_2)  \right ]  \right \|_{L^2(\mathbb{R}^n)} 
\[
\left \|d  (A_1-A_2)   \right \|_{H^{-1}(\Omega)} \leq C \left | \log dist(C_1^\Gamma, C_2^\Gamma) \right |^{ - \frac{\lambda }{n}s^2}, 
\]
holds true for all $\chi_\Omega A_1, \chi_\Omega A_2 \in \mathscr{A}(\Omega, M,s)$ and for all $q_1,q_2\in L^{\infty}(\Omega)$, whenever
\[
dist\, (C_1^\Gamma, C_2^\Gamma)\leq e^{-C}.
\]
\end{thm}

\begin{thm}[Stability for the electric potential] \label{SEP}
Let  $\Omega \subset \mathbb{R}^n$ be a bounded open set. Consider two constants $M>0$ and $s\in \left ( 0,1/2 \right )$. Then there exist $C>0$ (depending on $n, \Omega, M, s$) and an universal constant $\lambda\in \left( 0,1\right)$ such that the following estimate 
\[
\left \| q_1-q_2   \right \|_{L^{2}(\Omega)} \leq C \left |   \log  dist \, (C_1^\Gamma, C_2^\Gamma)   \right |^{ - \frac{\lambda }{n^2}s^3}, 
\]
holds true for all $\chi_\Omega A_1$, $\chi_\Omega A_2 \in \mathscr{A}(\Omega, M,s)$ and for all $\chi_\Omega q_1$, $\chi_\Omega  q_2\in \mathscr{Q}(\Omega, M,s)$, whenever 
\[
dist \, (C_1^\Gamma, C_2^\Gamma) \leq e^{-C}. 
\]
\end{thm}

To the best of our knowledge, the theorems \ref{identibiafgteru}, \ref{SMP} and \ref{SEP} are the first results concerning to the identifiability problem and its corresponding stability estimates for this IBVP with local data (associated to (\ref{artif3})) without any smoothness condition over the boundary $\partial\Omega$, except the assumption on the inaccessible part. The proofs of these theorems will be carried out by using solutions having the vanishing condition on $\Gamma_0$. These solutions can be constructed by combining the solutions constructed in \cite{KU} and a reflection argument across $\Gamma_0$, here is where the shape of $\Omega$ plays a crucial role. The original idea of reflecting goes back to Isakov's work on identifiability for a Schr\"odinger operator in the absence of a magnetic potential \cite{Is}. To obtain our estimate for the magnetic potential we will use an integral identity (see Lemma \ref{Alid}) in order to isolate the difference of the magnetic potentials $A_1-A_2$, then by using a quantitative version of the Riemann--Lebesgue lemma and Fourier transform we prove Theorem \ref{SMP}. The stability estimate for the electric potentials will be proved by using an analogous integral identity in order to isolate now the difference of the electric potentials $q_1-q_2$ and some extra arguments. These are the Hodge decomposition derived in \cite{CP} and the gauge invariance of the Cauchy data sets in order to use the already established stability estimate for the magnetic potentials in Theorem \ref{SMP}. The Besov spaces have been introduced in section $2$ in order to control oscillations for the magnetic and electric potentials, as required in this ill-posed inverse problem. We also observe that our results imply $\log$-estimates for both magnetic and electric potentials.  This is the best stability modulus that one can expect as was proved by Mandache \cite{Ma} in the context of the DN map (\ref{ldnmap1}) in the absence of a magnetic potential and for full data case without the support constraint on $\partial\Omega$, that is when $\Gamma=\partial\Omega$ and $\Gamma_0=\emptyset$.\\

Inverse boundary value problem with global and local data have been studied by several authors. When the boundary $\partial\Omega$ is smooth enough and for full data case, Sun proved identifiability by assuming the smallness of  the magnetic potential in a suitable space \cite{Sun}. In \cite{NSU} the smallness was removed for $C^2$ and compactly supported magnetic potential and $L^\infty$ electric potential. Stability estimates for these full data results were obtained by Tzou in \cite{Tz}. The previous identifiability results were extended by Krupchyk and Uhlmann \cite{KU} for both magnetic and electric potentials in $L^{\infty}$ without any smoothness assumption over the boundary $\partial\Omega$, that is, in the context of $N^\Gamma_{A,q}$ for full data case without the support constraint on $\Gamma$, that is when $\Gamma=\partial\Omega$ and $\Gamma_0=\emptyset$. For this case, stability estimates were obtained by Caro and Pohjola \cite{CP}. In all previous cases, $\log$-estimates were obtained for both magnetic and electric potentials.\\

Now we turn to mention the results about the IBVP with local data when the boundary $\partial\Omega$ is smooth enough, that is taking into account the local DN map (\ref{ldnmap1}). In this case, the first identifiability result in the absence of a magnetic potential ($A\equiv 0$) was obtained by Isakov \cite{Is}. He proved that if $\Lambda_{0,q_1}^\Gamma=\Lambda_{0,q_2}^\Gamma$ then $q_1=q_2$ assuming  that the inaccessible part of the boundary is either part of a plane or a sphere. The corresponding stability estimate to Isakov's result was derived by Heck and Wang \cite{HW1}, obtaining a $\log$-estimate. Following similar ideas as in \cite{Is}, Krupchyk, Lassas and Uhlmann \cite{KLU} extended Isakov's identifiability result to the Schr\"odinger operator in the presence of a magnetic potential. Similar arguments were employed by Caro \cite{Ca11} to study an IBVP with local data for the Maxwell equation under the same flatness condition on $\Gamma_0$. Caro also obtained a $log$-stability estimate.\\

As far we know, there are no results about identifiability to an IBVP with local data for an arbitrary open set. Nevertheless, the analogous of Calder\'on's result for the linearization of the DN map in the case of a conductivity, has been proved by Dos Santos Ferreira, Kening, Sj\"ostrand and Uhlmann \cite{DSFKSU0}, in the case of local data for the Schr\"odinger equation in the absence of a magnetic potential.\\

We also mention that there is another kind of IBVP, the called IBVP with partial data. That is, roughly speaking, the measurements of the DN map are taken on subsets of the boundary for functions supported on the complement of these subsets. Results concerning stability estimates for these kind of IBVP are worse than our estimates (IBVP with partial data), because the first ones have $\log\log$-modulus of continuity while we have just one $\log$, see theorems \ref{SMP} and \ref{SEP}. 
In turn, we distinguish two kind of IBVP which depend of how we are seeing the domain $\Omega$. The first one is described as illuminating $\Omega$ from the infinity because we see the domain from a fixed $\xi\in S^{n-1}$. In the absence of a magnetic potential, identifiability result was obtained by Bukgheim and Uhlmann \cite{BU}. The corresponding stability estimates (for different constraints over the functions where the DN map were taken) were derived by  Heck and Wang \cite{HW} and Caro, Dos Santos and Ruiz \cite{CDSFR}. In the presence of a magnetic potential, identifiability and stability (from different constraints over the functions where the DN map were taken) were obtained by Tzou \cite{Tz} and Potenciano \cite{Po}. The second kind of IBVP is describe as illuminating $\Omega$ from a fixed point $x_0\in \mathbb{R}^n$. In this case and in the absence of a magnetic potential, identifiability was proved by Kening, Sjostrand and Uhlmann \cite{KSU} and the corresponding stability estimates were derived by Caro, Dos Santos and Ruiz  \cite{CDSFR}-\cite{CDSFR1}. In the presence of a magnetic potential, identifiability was proved by Dos Santos, Kening, Sjostrand and Uhlmann \cite{DSFKSjU}. This result was extended by Chung \cite{Ch}. For these two last cases, the corresponding stability estimates are still open.\\

Throughout this paper we denote by $C_i$ ($i\in \mathbb{Z}^+$) a positive constants which might change from formula to formula. These constants depend only on $n, \Omega$ and the priori bounds for magnetic and electric potentials.\\

This paper is organized as follows. In the section $2$ we give the precise definition and notations employed to state theorems \ref{identibiafgteru}, \ref{SMP} and \ref{SEP} . In section $3$ we prove Theorem \ref{SMP}. In the section $4$ we prove Theorem \ref{SEP}. In the section $5$ we prove Theorem \ref{identibiafgteru}.

\section{Preliminaries - Definitions and notations}\label{prel}

As we already mentioned in the section \ref{introz}, here we introduce the definitions of the local boundary map $T_r^\Gamma$, the local linear map $N^\Gamma_{A,q}$, the space $H^1(\Omega, \Gamma)$,  the admissible class of the magnetic potentials $\mathcal{A}(\Omega, M, s)$, the admissible class of the electric potentials $\mathcal{Q}(\Omega, M, s)$ and the pseudo-distance between local-data Cauchy sets $\mbox{dist}(\, \cdot \, , \, \cdot \,)$. The space $H^1(\Omega, \Gamma)$ is defined by density as 
%In particular, we improve the identifiability result obtained by Krupchyk, Lassas and Uhlmann \cite{KLU}. Our results involve a particular shape of  $\Omega$. More precisely, we assume that 
%\begin{equation}
%\Omega \subset \left\{x\in \mathbb{R}^n \, : \, x_n>0  \right\} \quad \mbox{and} \quad \Gamma = \partial\Omega \cap \left\{x\in \mathbb{R}^n \, : \, x_n=0  \right\} \neq \emptyset
%\end{equation}
\begin{equation}\label{completation}
H^1(\Omega, \Gamma):= \overline{ \left\{  u\in C^\infty(\overline{\Omega}) \, : \, u|_{\Gamma_0}=0  \right\}}^{H^1(\Omega)},
\end{equation}
where $\overline{E}^{H^1(\Omega)}$ denotes the closure of the set $E$ in the topology of $H^1(\Omega)$, whith the norm $H^1(\Omega)$ restricted to $E$.\\
Thus, we define the local boundary map as
\begin{equation}\label{localW1}
T_r^\Gamma: H^1(\Omega, \Gamma)\rightarrow H^1(\Omega, \Gamma)/ H^{1}_0(\Omega), \quad T_r^\Gamma u = \left [ u \right ],
\end{equation}
where $\left[ u \right]$ denotes the equivalence class of $u\in H^{1}(\Omega, \Gamma)$ in the quotient space $H^1(\Omega, \Gamma)/ H^{1}_0(\Omega)$.\\
The local linear map $N_{A,q}^\Gamma: H^1(\Omega, \Gamma)/ H^{1}_0(\Omega) \rightarrow \left(H^1(\Omega, \Gamma)/ H^{1}_0(\Omega)\right)^*$ is defined by
\begin{equation}\label{localW}
\left \langle N_{A,q}^\Gamma \left [ u \right ], \left [ g \right ]  \right \rangle =  \int_{\Omega} Du\cdot \overline{Dv} + A\cdot (Du\overline{v}+u\overline{Dv})+ (A^2+q)u\overline{v},
\end{equation}
for all $u\in H^1(\Omega, \Gamma)$ satisfying $\mathcal{L}_{A,q}u=0$  (in $\Omega$) and for any  $v\in \left [ g \right ]$ with $g\in H^1(\Omega, \Gamma)$. Observe that $N_{A,q}^\Gamma$ makes sense for equivalence class of functions $u\in H^1(\Omega)$ which come from the solutions of $\mathcal{L}_{A,q}u=0$. With these definitions and notations, the local Cauchy data set can also be written as 
\begin{equation}\label{lcdset}
C_{A,q}^\Gamma= \left \{(\left [ u \right ], N_{A,q}^\Gamma \left [ u \right ] ):  u\in H^1(\Omega, \Gamma),\,  \mathcal{L}_{A,q}u=0 \,\,  \mbox{in}\,\,  \Omega  \right \},
\end{equation}
see (\ref{lcdsetz}). Following \cite{CP}, we introduce the  pseudo-distance $\mbox{dist}(\, \cdot \, , \, \cdot \,)$, inspired in the Hausdorff distance, as follows. Let $A_1, A_2\in L^{\infty}(\Omega; \mathbb{C}^n)$ be two magnetic potentials and let $q_1,q_2\in L^{\infty}(\Omega)$ be two electric potentials. Given $(f, g)\in C^\Gamma_{A_j,q_j}$ with $j=1,2$, we set
\begin{align*}
&I( (f ,g ); C^\Gamma_{A_k,q_k} )\\
&\qquad  = \underset{(f_k,g_k)\in  C^\Gamma_{A_k,q_k} }{\inf} \left [  \left \| f-f_k \right \|_{H^1(\Omega,\Gamma)/ H^1_0(\Omega)} + \left \| g-g_k \right \|_{(H^1(\Omega, \Gamma)/ H^1_0(\Omega))^*}  \right ],
\end{align*}
with $k=1,2$. Then the pseudo-distance between $C_{A_1, q_1}^\Gamma$ and $C_{A_2,q_2}^\Gamma$ is defined by
\begin{equation}\label{distanciaa}
 \mbox{dist} (C_{A_1, q_1}^\Gamma, C_{A_2,q_2}^{\Gamma}  ) = \underset{\underset{j\neq k}{j,k=1,2}}{\max}  \underset{   \underset{ \left \| f_j \right \|_{H^1(\Omega,\Gamma)/ H^1_0(\Omega)}=1 }{(f_j,g_j)\in C_{A_j, q_j}^\Gamma}    }{\sup} I( (f_j,g_j); C^\Gamma_{A_k,q_k} ).
\end{equation}

We now turn to the notations and definitions concerning our stability results: Theorems \ref{SMP} and \ref{SEP}. As we have already mentioned, in order to obtain stability estimates we have to assume a priori bounds for the magnetic and electric potentials in a suitable spaces. These are in order to control the behavior of the oscillations. Thus, following \cite{CP}, for given $s>0$ we define the Besov space $B_s^{2,\infty} (\mathbb{R}^n;\mathbb{C})$ as the space consisting of all functions $f \in L^{2}(\mathbb{R}^n; \mathbb{C})$ for which the norm
\begin{equation}\label{Besovspacet}
\left \| f\right \|_{B_s^{2,\infty} (\mathbb{R}^n)} =  \left \| f \right \|_{L^2(\mathbb{R}^n)} + \underset{y\in \mathbb{R}^n\setminus \left \{ 0 \right \}   }{\sup} \dfrac{ \left \| f(\cdot +y) - f(\cdot) \right \|_{L^2(\mathbb{R}^n)}   }{\left | y \right |^s} 
\end{equation}
is finite. Note that the Besov space has been defined for scalar functions. The definition for vectorial functions are similar.\\

Thus, given $M>0$ and $s\in \left( 0,1/2 \right)$, we define the class of admissible magnetic potentials as
\[
\mathscr{A}(\Omega, M, s )= \left \{ F  \in  L^\infty\cap B^{2,\infty}_s(\mathbb{R}^n; \mathbb{C}^n)  :  \supp F \subset \overline{\Omega}, \, \left \| F \right \|_{L^\infty\cap B^{2,\infty}_s}  \leq M \right \},
\]
and the class of admissible electric potentials as
\[
\mathscr{Q}(\Omega, M, s)= \left \{ G  \in  L^\infty\cap B^{2,\infty}_s(\mathbb{R}^n; \mathbb{C})  :  \supp G \subset \overline{\Omega}, \,  \left \|  G \right \|_{L^\infty\cap B^{2,\infty}_s}   \leq M \right \},
\]
where $\left \|\, \cdot \, \right \|_{L^\infty\cap B^{2,\infty}_s}$ denotes the sum of the norms $\left \|\, \cdot \, \right \|_{L^\infty(\mathbb{R}^n)} + \left \|\, \cdot \, \right \|_{ B^{2,\infty}_s(\mathbb{R}^n)}$.
\begin{rem}
Note that in theorems \ref{SMP} and \ref{SEP}, we assume the extensions by zero of the corresponding potentials to be in the Besov space $B^{2,\infty}_s$ with $s\in \left(0, 1/2\right)$. We mention that this is true for instance for Lipschitz domains, that is when $\partial\Omega$ is locally defined by the graph of a Lipschitz function, see \cite{Tr}. In this regularity framework over $\Omega$ it was also proved by Faraco and Rogers that $\chi_\Omega$ belongs to $H^{1/2-\epsilon}(\mathbb{R}^n)$ for any $\epsilon>0$ small enough, see \cite{FRo} for more details. This is the main motivation of why  in Theorems \ref{SMP} and \ref{SEP} we have taken the exponent $s$ of the Besov spaces on $\left( 0, 1/2\right)$.
\end{rem}

\section{Stability for the magnetic field}\label{sde}

\subsection{CGO solutions for a magnetic Schr\"odinger operator}\label{mollification}
The main result in this section is Theorem \ref{Exso}. This theorem ensures the existence of solutions for the magnetic Schr\"odinger operator $\mathcal{L}_{A,q}$.
\begin{thm}\label{Exso}
Let $V\subset \mathbb{R}^n$ be a bounded open set. Consider $s\in \left(0,1/2 \right)$. Let $A\in L^\infty\cap B^{2,\infty}_s(\mathbb{R}^n;\mathbb{C}^n)$ and $q\in L^\infty(\mathbb{R}^n;\mathbb{C})$ such that $\supp A\subset \overline{V}$ and $\supp q\subset \overline{V}$. Consider $\rho\in \mathbb{C}^n$ such that $\rho\cdot \rho=0$ and $\rho=\rho_0+\rho_\tau$ with $\rho_0$ being independent of some large parameter $\tau>0$, $\left | \Re \rho_0 \right |= \left | \Im \rho_0 \right |=1$ and $\rho_\tau=\mathcal{O}(\tau^{-1})$ as $\tau \mapsto \infty$. Then there exist two positive constants $C$ and $\tau_0$ (both depending on $n, V, s, \left \| A \right \|_{L^\infty\cap B^{2,\infty}_s}, \left \| q \right \|_{L^\infty}$)  ; and a solution $u\in H^{1}(V)$ to the equation $\mathcal{L}_{A,q}\,u=0$ in $V$ of the form
\[
u(x,\rho;\tau)=e^{\tau \rho\cdot x} \left(   e^{\Phi^\sharp(x,\rho_0;\tau)} +r(x,\rho; \tau)  \right)
\]
with the following properties:
\begin{item}
\item[(i)] The function $\Phi^\sharp(\cdot, \rho_0; \tau)  \in C^{\infty}(\mathbb{R}^n)$ and satisfies for all $\alpha\in \mathbb{N}^n$
\begin{equation}\label{pruno}
\left \| \partial^{\alpha} \Phi^\sharp(\cdot, \rho_0; \tau) \right \|_{L^\infty(\mathbb{R}^n)}  \leq C \tau^{\left | \alpha \right |/(s+2)} \left \| A \right \|_{L^{\infty}(\mathbb{R}^n)}, \quad \tau\geq \tau_0.
\end{equation}
\item[(ii)] The function $r(\cdot, \rho_0; \tau)\in H^{1}(V)$ and satisfies
\begin{equation}\label{prdos}
\left \| \partial^{\alpha} r(\cdot, \rho_0; \tau) \right \|_{L^2(V)}\leq C\,  \tau^{\left | \alpha \right |-s/(s+2)}, \quad \left |\alpha  \right |\leq 1.
\end{equation} 
\item[(iii)] If we define by $\kappa:=\underset{x\in \overline{V}}{\sup }\left | x \right |$ then the solution $u$ satisfies 
\begin{equation}\label{–zpq1234}
\left \| u \right \|_{H^1(V)}\leq C\, e^{\tau \kappa   \left | \rho \right |  }.
\end{equation}
\end{item}
Moreover, if we denote by  $\Phi(\cdot; \rho_0)= (\rho_0\cdot \nabla)^{-1} (-i\rho_0\cdot A)\in L^{\infty}(\mathbb{R}^n)$ the function satisfying the following equation in $\mathbb{R}^n$
\begin{equation}\label{mmmmm}
\rho_0\cdot \nabla \Phi + i\rho_0\cdot A=0
\end{equation}
then
\begin{equation}\label{prtres}
\left \| \Phi (\cdot; \rho_0)\right \|_{L^\infty(\mathbb{R}^n)} \leq C \left \|A  \right \|_{L^{\infty}(\mathbb{R}^n)}.
\end{equation}
Finally,  for every $\chi\in C^{\infty}_0(\mathbb{R}^n)$ we have
\begin{equation}\label{prcuatro}
\left \| \chi (\Phi^\sharp(\cdot, \rho_0;\tau)- \Phi(\cdot; \rho_0)) \right \|_{L^2(\mathbb{R}^n)}\leq C \tau^{-s/(s+2)} \left \| A  \right \|_{L^\infty(\mathbb{R}^n)},
\end{equation}
where the constant $C$ also depends on $\chi$.
\end{thm}
\begin{rem}
 This theorem is a summary of two known results. On one hand, the existence of $u\in H^1(V)$ satisfying $\mathcal{L}_{A,q}u=0$ in $V$ (with $A\in L^\infty$ and $q\in L^\infty$) was proved by Krupchyk and Uhlmann in \cite{KU}. On the other hand, when $A\in L^{\infty}\cap B^{2,\infty}_s$ and $q\in L^\infty$, the corresponding estimates for $ \Phi^\sharp(\cdot, \rho_0; \tau), \Phi(\cdot; \rho_0)$ and $r(\cdot, \rho_0;\tau)$ have been taken from Proposition $2.6$ in \cite{KU} and the section $3$ in \cite{CP}. For these reasons we only give the main ideas of the proof with the repetition of the relevant material from \cite{CP} and \cite{KU}, thus making our exposition self-contained.
%Along this article we will use Theorem \ref{Exso} for different open sets $V$.
\end{rem}
\begin{proof}
From \cite{KU} it follows that  for all $A\in L^{\infty}(V)$ and for all $q\in L^\infty(V)$ there exists a function $u\in H^1(V)$ satisfying $\mathcal{L}_{A,q}u=0$ in $V$ of the form
\begin{equation}\label{ansatz}
u(x)= e^{\tau \rho\cdot x}(a+r),
\end{equation}
where $\rho\in \mathbb{C}^n$ with $\rho\cdot \rho=0$, $\tau$ is a large positive parameter, $a$ is a smooth function solving a transport equation (see (\ref{transportequation})) and $r$ satisfies a remainder equation (see (\ref{remainderterm})). Their construction involves basically two arguments. The first argument concerns a mollification procedure for the magnetic potential $A$. More precisely: consider $\varphi \in C^{\infty}_0(\mathbb{R}^n)$ such that $0\leq \varphi \leq 1$ and $\supp \varphi \subset \overline{B_1(0)}$, where $\overline{B_1(0)}$ denotes the closure of the  ball in $\mathbb{R}^n$ of radius 1 centered at the origin. For each $\delta>0$ we define $\varphi_\delta(x)=\delta^{-n}\varphi(x/\delta)$ and set $A^\sharp_{\delta}= A*\varphi_\delta$ which belongs to $C^{\infty}_0(\mathbb{R}^n; \mathbb{C}^n)$. Then there exists a positive constant $C_1>0$ (depending on $V$ and $n$) such that 
\begin{equation}\label{Asharp}
\left \| A-A^{\sharp}_\delta  \right \|_{L^2(\mathbb{R}^n)} \leq C_1\, \delta^{s} \left \| A \right \|_{B^{2,\infty}_s(\mathbb{R}^n)}
\end{equation}
and for each $\alpha \in \mathbb{N}^n$ we have
\begin{equation}\label{Asharp1}
\left \| \partial^{\alpha} A^\sharp_\delta \right \|_{L^\infty(\mathbb{R}^n)} \leq C_2\, \delta^{-\left | \alpha \right |} \left \| A \right \|_{L^{\infty}(\mathbb{R}^n)},
\end{equation}
see the section $3$ in \cite{CP} for more details. The second argument involves the use of a Carleman estimate for the Laplacian operator derived by Salo and Tzou \cite{ST}. This Carleman estimate are between $H^1(V)$ and its dual space $H^{-1}(V)$, that means the gain of two derivatives. This last fact  is the main tool to the construction of solutions of the form (\ref{ansatz}). More precisely, an easy computation gives us that a function $u\in H^{1}(V)$ of the form (\ref{ansatz}) is a solution for $\mathcal{L}_{A,q}u=0$ if we have the following identity in $H^{-1}(V)$
\begin{align*}
&0= \tau^{-2}\mathcal{L}_{A,q}a -\tau^{-1} \left( 2i\rho_\tau\cdot Da + 2i\rho_0\cdot (A-A^\sharp_\delta)a +2i\rho_\tau\cdot A a  \right)\\
&\qquad  -\tau^{-1}\left( 2i\rho_0\cdot Da+2i \rho_0\cdot A^\sharp_\delta a \right) + e^{-\tau\rho\cdot x}\tau^{-2}\mathcal{L}_{A,q}(e^{\tau\rho\cdot x} r).
\end{align*}
Then $u\in H^{1}(V)$ of the form (\ref{ansatz}) is a solution of $\mathcal{L}_{A,q}u=0$ provided that  $a$ satisfies the following equation in $\mathbb{R}^n$
\begin{equation}\label{transportequation}
\rho_0\cdot \nabla a + i\rho_0\cdot A^\sharp_\delta a=0 
\end{equation}
and the term $r$ satisfies the following equation in $H^{-1}(V)$ 
\begin{equation}\label{remainderterm}
\begin{aligned}
&e^{-\tau\rho\cdot x}\tau^{-2}\mathcal{L}_{A,q} (e^{\tau\rho\cdot x}r)\\
&= -\tau^{-2}\mathcal{L}_{A,q}a +2i\tau^{-1}\left(\rho_1\cdot Da + \rho_0\cdot (A-A^\sharp_\delta)a+\rho_1\cdot A a  \right).
\end{aligned}
\end{equation}
The equation (\ref{transportequation}) and (\ref{remainderterm}) are called transport and remainder equations, respectively. The transport equation (\ref{transportequation}) can be solved as follows. If we make the ansatz $a=e^{\Phi^{\sharp}}$ then $\Phi^\sharp$ satisfies the equation
\begin{equation}\label{xiequation}
\rho_0\cdot \nabla \Phi^{\sharp} + i\rho_0\cdot A^\sharp_\delta =0.
\end{equation}
This equation is easy to solve because the condition $\rho_0\cdot \rho_0=0$ imply that $\Re \rho_0 \cdot \Im \rho_0=0$ and $\left | \Re \rho_0 \right |= \left | \Im \rho_0 \right |$ and then the operator $\rho_0\cdot \nabla$ becomes a $\partial_{\overline{z}}$ operator, where for each $x\in \mathbb{R}^n$ we have considered the complex variable $z(x)=\Re\rho_0\cdot x + i\Im\rho_0\cdot x$. Thus, the function $\Phi^{\sharp}=(\rho_0\cdot \nabla)^{-1}(-i\rho_0\cdot A^{\sharp}_\delta)$ belongs to $C^\infty(\mathbb{R}^n)$ and satisfies (\ref{xiequation}). Moreover, from (\ref{Asharp1}) we have that for all $\alpha\in \mathbb{N}^n$ there exists a constant $C_3>0$ such that
\begin{equation}\label{phi}
\left \| \partial^{\alpha} \Phi^\sharp \right \|_{L^\infty(\mathbb{R}^n)}  \leq C_3\, \delta^{-\left | \alpha \right |} \left \| A \right \|_{L^{\infty}(\mathbb{R}^n)}.
\end{equation}
For more details about the solvability of (\ref{xiequation}), see for example Lema $4.6$ in \cite{SaM}.
For similar reasons as above, the function $\Phi(\cdot; \rho_0)= (\rho_0\cdot \nabla)^{-1} (-i\rho_0\cdot A)\in L^{\infty}(\mathbb{R}^n)$ solves the following equation
\[
\rho_0\cdot \nabla \Phi + i\rho_0\cdot A=0.
\]
Moreover we have the estimate
\[
\left \| \Phi(\cdot; \rho_0) \right \|_{L^\infty(\mathbb{R}^n)} \leq C_4 \left \|A  \right \|_{L^{\infty}(\mathbb{R}^n)},
\]
where the constant $C_4>0$ only depend on $V$ and $n$. Also from (\ref{Asharp}), for every $\chi\in C^{\infty}_0(\mathbb{R}^n)$ there exist a constant $C_5>0$ (depending on $\Omega, n$ and $\chi$) such that
\[
\left \| \chi (\Phi^\sharp(\cdot, \rho_0)- \Phi(\cdot; \rho_0)) \right \|_{L^2(\mathbb{R}^n)}\leq C_5\, \delta^{s} \left \| A  \right \|_{B^{2,\infty}_s(\mathbb{R}^n)},
\]
see the section $3$ in \cite{CP} for more details. Now we explain the solvability of the remainder equation (\ref{remainderterm}). We start by setting 
\[
w= -\tau^{-2}\mathcal{L}_{A,q}a +2i\tau^{-1}\left(\rho_\tau\cdot Da + \rho_0\cdot (A-A^\sharp_\delta)a+\rho_\tau\cdot A a  \right).
\]
Then by Proposition $2.3$ in \cite{KU}, there exists $r\in H^1(V)$ a solution of (\ref{remainderterm}) and two positive constants $C_6$ and $\tau_0$ such that  
\begin{equation}\label{erre}
\left \| r \right \|_{H^1_{scl}(V)}\leq C_6\, \tau \left \| w \right \|_{H^{-1}_{scl}(V)},
\end{equation}
for all $\tau\geq \tau_0$. Here the semi-classical norms are defined by
\[
\left \| r \right \|^2_{H^1_{scl}(V)}= \left \| r \right \|^2_{L^2(V)}+  \left \|\tau^{-1} \nabla r \right \|^2_{L^2(V)},
\]
\[
\left \| w \right \|_{H^{-1}_{scl}(V)}= \underset{0\neq \phi \in C^\infty_0(V)}{\sup} \dfrac{\left \langle  w,\phi \right \rangle_{L^2(V)}}{\left \| w \right \|_{H^1_{scl}(V)}}.
\]
If we define $\kappa:=\underset{x\in \overline{V}}{\sup }\left | x \right |$ then from (\ref{phi}) and by  taking $\delta=\tau^{-{1}/{(s+2)}}$ into (\ref{Asharp1}), we get
\begin{align*}
&\left \| w \right \|_{H^{-1}_{scl}(V)} \leq C_7\,  e^{\kappa\left \| A \right \|_{L^\infty}} \tau^{-(2s+2)/(s+2)}\\
&\qquad \qquad \qquad   \times  \left(  1+ \left \| A \right \|_{L^\infty}  + \left \| A \right \|^2_{L^\infty} + \left \| q \right \|_{L^\infty} + \left \| A \right \|_{B_s^{2,\infty}}  \right).
\end{align*}
Combining the above inequality with (\ref{erre}) we obtain
\begin{equation}\label{restww}
\begin{aligned}
& \left \| r \right \|_{H^1_{scl}(V)} \leq C_8\, e^{\kappa\left \| A \right \|_{L^\infty}} \tau^{-s/(s+2)}\\
&\qquad  \qquad   \times  \left(  1+ \left \| A \right \|_{L^\infty}  + \left \| A \right \|^2_{L^\infty} + \left \| q \right \|_{L^\infty} + \left \| A \right \|_{B_s^{2,\infty}}  \right)
\end{aligned}
\end{equation}
and by similar kind of computations as above we obtain
\begin{equation}\label{uac}
\begin{aligned}
& \left \| u \right \|_{H^1(V)}\leq C_9\, e^{\tau \kappa   \left | \rho \right |  }e^{C\left \| A \right \|_{L^\infty(V)}}\\
& \qquad  \qquad   \times    \left(  1+ \left \| A \right \|_{L^\infty}  + \left \| A \right \|^2_{L^\infty} + \left \| q \right \|_{L^\infty} + \left \| A \right \|_{B_s^{2,\infty}}  \right).
\end{aligned}
\end{equation}
So these are the main ideas of the proof.
\end{proof}
\begin{rem}[Estimates for the identifiability result]\label{remesbaxzq}
For our identifiability result, see Theorem \ref{identibiafgteru}, we are assuming that $A\in L^{\infty}(\Omega; \mathbb{C}^n)$. In this case we can obtain a similar estimates as (\ref{pruno})-(\ref{prcuatro}). These estimates are enough to prove Theorem \ref{identibiafgteru} (see Proposition $2.6$ in \cite{KU}) and can be stated as follows. There exist two positive constants $C$ and $\tau_0$ such that: the function $\Phi^\sharp(\cdot, \rho_0; \tau)  \in C^{\infty}(\mathbb{R}^n)$ and satisfies for all $\alpha\in \mathbb{N}^n$ and for all $ \lambda\in \left(0,1/2\right)$
\begin{equation}\label{prunozxa}
\left \| \partial^{\alpha} \Phi^\sharp(\cdot, \rho_0; \tau) \right \|_{L^\infty(\mathbb{R}^n)}  \leq C \tau^{\lambda \left | \alpha \right |}, \quad \tau\geq \tau_0.
\end{equation}
The function $r(\cdot, \rho_0; \tau)\in H^{1}(V)$ and satisfies for all $\left |\alpha  \right |\leq 1$
\begin{equation}\label{prdoszxa}
\left \| \partial^{\alpha} r(\cdot, \rho_0; \tau) \right \|_{L^2(V)}\leq C\,  \tau^{\left | \alpha \right |},\quad \tau\geq \tau_0.
\end{equation} 
The estimate (\ref{prtres}) is the same. Finally,  for every $\chi\in C^{\infty}_0(\mathbb{R}^n)$ we have
\begin{equation}\label{prcuatrozxa}
\underset{\tau\rightarrow \infty}{\lim}\left \| \chi (\Phi^\sharp(\cdot, \rho_0;\tau)- \Phi(\cdot; \rho_0)) \right \|_{L^2(\mathbb{R}^n)}=0.
\end{equation}
where the constant $C$ also depends on $\chi$. We will not use these estimates until section $5$ to prove Theorem \ref{identibiafgteru}.
\end{rem}
\subsection{From the boundary to the interior} In this section we state an integral estimate which involves a relation between the potentials (magnetic and electric) and the distance between their corresponding Cauchy data sets. This integral identity was proved implicitly in \cite{KU}, see Proposition $3.2$ therein.
\begin{lem} \label{Alid}
Let $\Omega$ be an open bounded set. Let $A_1,A_2\in L^{\infty}(\Omega; \mathbb{C}^n)$ and $q_1,q_2\in L^{\infty}(\Omega; \mathbb{C})$. Let $U_1, U_2\in H^{1}(\Omega)$ be functions satisfying in $\Omega$: $\mathcal{L}_{A_1, q_1}U_1=0$ and $\mathcal{L}_{\overline{A}_2, \overline{q_2}}U_2=0$. Then the following identity holds true:
\begin{align*}
& \left \langle N_{A_1,q_1} \left [ U_1 \right ], \overline{T_rU_2}  \right \rangle - \overline{\left \langle N_{\overline{A}_2,\overline{q_2}}\left [ U_2 \right ], \overline{T_rU_1}  \right \rangle}\\
& = \int_{\Omega} \left [ (A_1-A_2)\cdot \left( DU_1\overline{U_2} +  U_1 \overline{DU_2} \right) + \left( A_1^2-A_2^2+q_1-q_2 \right) U_1\overline{U_2} \right ] dx.
\end{align*}
\end{lem}
\begin{cor}\label{coAlid}
Consider all conditions from the previous lemma. Consider also $s\in \left(0,1/2\right)$. Let $U_1, U_2\in H^{1}(\Omega)$ be functions satisfying in $\Omega$: $\mathcal{L}_{A_1, q_1}U_1=0$ with  $U_1|_{\Gamma_0}=0$ and $\mathcal{L}_{\overline{A}_2, \overline{q_2}}U_2=0$ with $U_2|_{\Gamma_0}=0$. Then there exists a positive constant $C$ (depending on $n, \Omega, \left \| A_j \right \|_{L^\infty\cap B^{2,\infty}_s}$, $ \left \| q_j \right \|_{L^\infty}; j=1,2$) such that
\begin{equation}\label{uau}
\begin{aligned}
& \left |  \int_{\Omega}  (A_1-A_2)\cdot \left( DU_1\overline{U_2} +  U_1 \overline{DU_2} \right) + \left( A_1^2-A_2^2+q_1-q_2 \right) U_1\overline{U_2}  \right |\\
&  \qquad  \leq C dist \, (C_1^\Gamma, C_2^\Gamma) \left \| U_1 \right \|_{H^1(\Omega)} \left \| U_2 \right \|_{H^1(\Omega)},
\end{aligned}
\end{equation}
where $C_j^\Gamma$ denotes the local Cauchy data set $C_{A_j, q_j}^\Gamma$, $j=1,2$.
\end{cor}

\begin{rem}
Corollary \ref{coAlid} was proved for full data case in Proposition $2.1$ in \cite{CP}. The proof for local data follows from the full data case and taking into account that $U_1|_{\Gamma_0}= U_2|_{\Gamma_0}=0$ with (\ref{distanciaa}).
\end{rem}

\subsection{Solutions vanishing on $\Gamma_0$} In this section we will use Theorem \ref{Exso} to construct solutions $U\in H^{1}(\Omega)$ for the magnetic Schr\"odinger operator $\mathcal{L}_{A,q}U=0$ (in $\Omega$) with the required condition $U|_{\Gamma_0}=0$. To achieve this condition we will use a reflection argument as in \cite{KLU}. The main results of this section is Proposition \ref{construcsolutions}.\\

We set $x^*=(x^\prime, -x_n)$ for any $x=(x^\prime, x_n)\in \mathbb{R}^n$, $f^*(x)=f(x^*)$ for any function $f$ and $E^*=\left \{ x^*: x\in E \right \} $. Also for any $\rho\in \mathbb{C}^{n}$ we define $\rho^*=(\Re \rho)^*+ i(\Im \rho)^*$.
%\begin{figure}[h]
%\caption{Example}
%\includegraphics[width=0.35\textwidth]{leytera}
%\hspace{0.1\linewidth}
%\includegraphics[width=0.35\textwidth]{leyterb}
%\includegraphics[width=0.49\textwidth]{leyterc}
%\hspace{0.1\linewidth}
%\includegraphics[width=0.49\textwidth]{leyterd}
%\end{figure}
Then, similarly to \cite{KLU}, we extend the magnetic an electric potentials from $\Omega$ to $\Omega^*$ by reflection with respect to the plane $\left\{ x\in \mathbb{R}^n \, : \, x_n=0\right\}$ as follows. Recall that we have denoted by $A=(A^{(1)}, A^{(2)}, \ldots, A^{(n-1)}, A^{(n)})$ a magnetic potential. Then for $A^{(k)}$ with $k=1,2,\ldots, n-1$ we make an even extension and for $A^{(n)}$ we make an odd extension. We denote this extension by $\widetilde{A}$. More precisely,  for all $k=1,2, \ldots, n-1$ we have:
 \[
     \widetilde{A^{(k)}} (x)=  \begin{cases}
         A^{(k)} (x^\prime, x_n), & x\in \Omega \\ 
          A^{(k)} (x^\prime, -x_n) ,& x\in \Omega^{*},
     \end{cases}
\]
and
 \[
     \widetilde{A^{(n)}} (x)=  \begin{cases}
         A^{(n)} (x^\prime, x_n), & x\in \Omega \\ 
         - A^{(n)} (x^\prime, -x_n) ,& x\in \Omega^{*}.
     \end{cases}
\]
In the same way, for a electric potential $q$ we make an even extension. We denote these extensions by $\widetilde{q}$. More precisely, we have:
 \[
     \widetilde{q_j} (x)=  \begin{cases}
         q (x^\prime, x_n), & x\in \Omega \\ 
          q (x^\prime, -x_n) ,& x\in \Omega^{*}.
     \end{cases}
\]

The following lemma gives us the smoothness properties for the magnetic potentials of the above extensions.
\begin{lem}\label{heyw}
Let $\Omega\subset \mathbb{R}^n$ be a bounded set. Let $M>0$ and $s\in \left( 0,1/2\right)$. Consider $A\in L^\infty(\Omega; \mathbb{C}^n)$ and $q\in L^\infty({\Omega; \mathbb{C}})$. If $\chi_\Omega A\in \mathcal{A}(\Omega, M,s)$ and $\chi_\Omega q \in \mathcal{Q}(\Omega, M, s)$ then $\chi_{\Omega \cup \Omega^*}\widetilde{A}\in \mathcal{A}(\Omega\cup\Omega^*, 2M,s)$ and $\chi_{\Omega\cup\Omega^*}\widetilde{q}\in \mathcal{Q}(\Omega\cup\Omega^*, 2M,s)$. 
\end{lem}
\begin{proof}The proof is based on the following observation. If $f\in B^{2, \infty}_s(\mathbb{R}^n; \mathbb{C})$ then for all $y\in \mathbb{R}^n$ and  by Plancherel's theorem we get 
\[
\left \| f(\cdot+y)- f(\cdot ) \right \|_{L^2(\mathbb{R}^n)}^2 = \int_{\mathbb{R}^n}\left | \widehat{f} (\xi)\right |^2 \left | e^{-2\pi i\xi\cdot y}-1 \right |^2d\xi,
\]
which imply the following equivalent norm to (\ref{Besovspacet}):
\begin{equation}\label{equiB}
\left \| f \right \|_{ B^{2, \infty}_s(\mathbb{R}^n)}^2=\int_{\mathbb{R}^n}\left | \widehat{f} (\xi)\right |^2d\xi    + \underset{y\in \mathbb{R}^n\setminus \left\{ 0 \right\}}{\sup}  \dfrac{ \int_{\mathbb{R}^n}\left | \widehat{f} (\xi)\right |^2 \left | e^{-2\pi i\xi\cdot y}-1 \right |^2d\xi}{\left | y \right |^{2s}}.
\end{equation}
According to the above identity, the next step will be to obtain a relation between $\widehat{\chi_\Omega A}$ and $\widehat{\chi_{\Omega\cup \Omega^*}\widetilde{A}}$. In this way, for any $j=1,2, \ldots, n-1$ we have
\begin{align*}
 \widehat{\chi_{\Omega\cup \Omega^*}\widetilde{A^{(j)}}} (\xi)&  = \int_{\mathbb{R}^n}e^{i\xi\cdot x} \chi_{\Omega\cup \Omega^*}\widetilde{A^{(j)}}dx = \int_{\Omega\cup \Omega^*}e^{i\xi\cdot x} \widetilde{A^{(j)}}dx\\
& =  \int_{\Omega}e^{i\xi\cdot x} \widetilde{A^{(j)}}dx +  \int_{ \Omega^*}e^{i\xi\cdot x} \widetilde{A^{(j)}}dx\\
& =  \int_{\Omega}e^{i\xi\cdot x} A^{(j)}(x)dx + \int_{\Omega^*}e^{i\xi\cdot x} A^{(j)}(x^*)dx \\
& =  \int_{\Omega}e^{i\xi\cdot x} A^{(j)}(x)dx + \int_{\Omega}e^{i\xi^*\cdot x} A^{(j)}(x)dx \\
& = \widehat{\chi_\Omega A^{(j)}} (\xi) + \widehat{\chi_\Omega A^{(j)}} (\xi^*). 
\end{align*}
Hence, by (\ref{equiB}) and the above identity, we get
\[
\left \| \chi_{\Omega\cup \Omega^*}\widetilde{A^{(j)}} \right \|_{B^{2,\infty}_s(\mathbb{R}^n)}\leq 2 \left \| \chi_{\Omega}A^{(j)}\right \|_{B^{2,\infty}_s(\mathbb{R}^n)}.
\]
Analogously, we obtain
\[
\left \| \chi_{\Omega\cup \Omega^*}\widetilde{A^{(n)}} \right \|_{B^{2,\infty}_s(\mathbb{R}^n)}\leq 2 \left \| \chi_{\Omega}A^{(n)}\right \|_{B^{2,\infty}_s(\mathbb{R}^n)}.
\]
Moreover, since $A\in L^{\infty}(\Omega;\mathbb{C}^n)$ it follows that $\chi_{\Omega\cup\Omega^*}\widetilde{A}\in L^{\infty}(\mathbb{R}^n;\mathbb{C}^n)$. Thus, by combining this fact with the two above inequalities we obtain the desired result for $A$. The proof for $q$ is analogous. So our proof is completed.
\end{proof}
Now roughly we explain the main ideas to obtain $U\in H^{1}(\Omega)$ satisfying $\mathcal{L}_{A, q}\, U=0$ in $\Omega$ with the required condition on the boundary $U|_{\Gamma_0}=0$. At the beginning one can apply Theorem \ref{Exso}  with $V=\Omega\cup \Omega^*$ in order to obtain $u\in H^1(\Omega\cup\Omega^*)$ satisfying $\mathcal{L}_{\chi_{\Omega\cup\Omega^*} \widetilde{A},\, \chi_{\Omega\cup\Omega^*} \widetilde{q}}\, u=0$ in $\Omega\cup\Omega^*$. Then it is easy to see that $U(x):=u(x)-u(x^*)$ is a solution of $\mathcal{L}_{A,q}U=0$ in $\Omega$. Thus, it only remains to prove the vanishing  condition on $\Gamma_0$. As we will see, in order to prove such condition we have to use an integration by parts. A priori this is not possible in $\Omega$ because we are not assuming any smoothness over $\partial\Omega$ (in particular over $\partial\Omega\setminus \Gamma_0$). To remedy this technical obstruction we will consider a ball $B$ such that $\Omega\cup\Omega^*\subset \subset B$ and then we construct a solutions of the operator $\mathcal{L}_{A,q}U=0$ in $B^+$, the upper half part of $B$. It is now possible to  apply an integration by parts in $B^+$ in order to obtain $U|_{\partial B^+ \cap \left\{ x_n=0\right\}}=0$. Finally, $U|_\Omega$ is a solution in $\Omega$ of the initial equation $\mathcal{L}_{A, q}\, U|_\Omega=0$ and the vanishing condition on $\Gamma_0$ follows from $\Gamma_0\subset \partial B^+ \cap \left\{ x_n=0\right\}$. The following proposition will be devoted to state and prove these ideas. 
\begin{prop} \label{construcsolutions}
Let $\Omega\subset \mathbb{R}^n$ be a bounded open set.  Let $A\in L^{\infty}(\Omega, \mathbb{C}^n)$ and $q\in L^{\infty}(\Omega, \mathbb{C})$. Given $M>0$ and $s\in \left( 0,1/2 \right)$, suposse that $\chi_\Omega A$ belongs to $\mathcal{A}(\Omega, M,s)$. Consider $\rho\in \mathbb{C}^n$ such that $\rho\cdot \rho=0$ and $\rho=\rho_0+\rho_\tau$ with $\rho_0$ being independent of some large parameter $\tau>0$, $\left | \Re \rho_0 \right |= \left | \Im \rho_0 \right |=1$ and $\rho_\tau=\mathcal{O}(\tau^{-1})$ as $\tau \mapsto \infty$. Then there exist two positive constants $C$ and $\tau_0$ (both depending on $n, \Omega, M, s, \left \| q \right \|_{L^\infty(\Omega)}$); and a solution $U\in H^{1}(\Omega)$ to the equation $\mathcal{L}_{A,q}\,U=0$ in $\Omega$ with $U|_{\Gamma_0}=0$ and of the form
\begin{equation}\label{az}
U(x,\rho;\tau)= u(x,\rho;\tau)- u(x^*,\rho;\tau),\quad x\in \Omega,
\end{equation}
where $u\in H^1({\Omega\cup\Omega^*})$ is a function satisfying $\mathcal{L}_{\widetilde{A}, \widetilde{q}}\,u=0$ in $\Omega\cup\Omega^*$ and has the form:
\begin{equation}\label{bz}
u(x,\rho;\tau)= e^{\tau \rho\cdot x} \left(   e^{\Phi^\sharp(x,\rho_0;\tau)} +r(x,\rho; \tau)  \right), \quad x\in \Omega\cup\Omega^*.
\end{equation}
Moreover we have the following properties:
\begin{item}
\item[(i)] The function $\Phi^\sharp(\cdot, \rho_0; \tau)  \in C^{\infty}(\mathbb{R}^n)$ and satisfies for all $\alpha\in \mathbb{N}^n$
\begin{equation}\label{prunoz}
\left \| \partial^{\alpha} \Phi^\sharp(\cdot, \rho_0; \tau) \right \|_{L^\infty(\mathbb{R}^n)}  \leq C \tau^{\left | \alpha \right |/(s+2)}, \quad \tau\geq \tau_0.
\end{equation}
\item[(ii)] The function $r(\cdot, \rho_0; \tau)\in H^{1}(\Omega\cup\Omega^*)$ and satisfies
\begin{equation}\label{prdosz}
\left \| \partial^{\alpha} r(\cdot, \rho_0; \tau) \right \|_{L^2(\Omega\cup\Omega^*)}\leq C \tau^{\left | \alpha \right |-s/(s+2)}, \quad \left |\alpha  \right |\leq 1.
\end{equation}
\item[(iii)] If we define by $\kappa:=\underset{x\in \overline{\Omega\cup\Omega^*}}{\sup }\left | x \right |$ then the solution $u$ satisfies 
\begin{equation}\label{jaci1}
\left \| u \right \|_{H^1(\Omega\cup\Omega^*)}\leq C\, e^{\tau \kappa   \left | \rho \right |  }.
\end{equation}
\end{item}
If we denote by  $\Phi(\cdot; \rho_0)= (\rho_0\cdot \nabla)^{-1} (-i\rho_0\cdot (\chi_{\Omega\cup\Omega^*} \widetilde{A}))\in L^{\infty}(\mathbb{R}^n)$ the function satisfying the equation in $\mathbb{R}^n$
\begin{equation}\label{mmmmmz}
\rho_0\cdot \nabla \Phi + i\rho_0\cdot (\chi_{\Omega\cup\Omega^*} \widetilde{A})=0
\end{equation}
then
\begin{equation}\label{prtresz}
\left \| \Phi (\cdot; \rho_0)\right \|_{L^\infty(\mathbb{R}^n)} \leq C.
\end{equation}
Finally,  for every $\chi\in C^{\infty}_0(\mathbb{R}^n)$ we have
\begin{equation}\label{prcuatroz}
\left \| \chi (\Phi^\sharp(\cdot, \rho_0;\tau)- \Phi(\cdot; \rho_0)) \right \|_{L^2(\mathbb{R}^n)}\leq C \tau^{-s/(s+2)},
\end{equation}
where the constant $C$ also depends on $\chi$.
\end{prop}
\begin{proof}
The proof is an immediate consequence of Theorem \ref{Exso}. Let $B$ be a ball centered at some fixed point on $\Gamma_0$ and such that $\overline{\Omega\cup\Omega^*}\subset B$. By hypothesis $\chi_\Omega A$ belongs to $\mathcal{A}(\Omega, M,s)$, thus by Lemma \ref{heyw}, it follows that $\chi_{\Omega \cup \Omega^*}\widetilde{A}\in \mathcal{A}(\Omega\cup\Omega^*, 2M,s)$. Since the function $\chi_{\Omega \cup \Omega^*}\widetilde{A}$ is zero out of $B$, we deduce that $\chi_{\Omega \cup \Omega^*}\widetilde{A}$ also belongs to $\mathcal{A}(B, 2M,s)$, which imply that $\chi_{\Omega \cup \Omega^*}\widetilde{A}\in L^\infty\cap B^{2,\infty}_s(\mathbb{R}^n, \mathbb{C}^n)$ and $\supp (\chi_{\Omega \cup \Omega^*}\widetilde{A})\subset B$. Notice also that $\chi_{\Omega\cup\Omega^*}\widetilde{q}\in L^{\infty}(\mathbb{R}^n, \mathbb{C})$ and $\supp (\chi_{\Omega\cup\Omega^*}\widetilde{q})\subset B$. Then, by Theorem \ref{Exso} applied to the functions $\chi_{\Omega\cup\Omega^*}\widetilde{A}$ and $\chi_{\Omega\cup\Omega^*}\widetilde{q}$ and $V=B$; there exist two positive constants $C$ and $\tau_0$ (both depending on $n, \Omega, M,  \left \| q \right \|_{L^\infty(\Omega)} $); and a function $u\in H^{1}(B)$ of the form
\[
u(x,\rho;\tau)=e^{\tau \rho\cdot x} \left(   e^{\Phi^\sharp(x,\rho_{0};\tau)} +r(x,\rho; \tau)  \right), \quad \tau\geq \tau_0,
\]
 satisfying in $B$:
 \begin{equation}\label{restrictl}
 \mathcal{L}_{\chi_{\Omega\cup \Omega^*\widetilde{A}}, \chi_{\Omega\cup\Omega^*}\widetilde{q} }u=0
 \end{equation}
with the corresponding estimates (\ref{pruno})-(\ref{prcuatro}). These estimates implies the estimates (\ref{prunoz})-(\ref{prcuatroz}). Now by a straightforward computation we have
\[
\mathcal{L}_{\chi_\Omega A, \chi_\Omega q} u (x) =0, \quad x\in B^+
\]
and
\[
\mathcal{L}_{\chi_\Omega A, \chi_\Omega q} u (x^*) =0, \quad x\in B^+,
\]
where $B^+= \left\{ x\in B : x_n>0  \right\}$, denoted the upper half part of $B$. Thus, from the two above equations we immediately deduce that the function defined by $U(x):=u(x)-u(x^*)$ satisfies $\mathcal{L}_{\chi_\Omega A, \chi_\Omega q}\, U=0$ in $B^+$. Also it is easy to see that $U(x)=0$ on $\partial B^+\cap \left\{ x_n=0\right\}$. Finally it is clear that $U$ restricted to $\Omega$, still denoted by $U$, satisfies the assertion of the Proposition and then the proof is complete.
\end{proof}
The next step will be to use Proposition \ref{construcsolutions} to construct functions $U_1, U_2\in H^1(\Omega)$, for some suitables $\rho_1$ and $\rho_2$, satisfying $\mathcal{L}_{A_1, q_1}U_1=0$ with $U_1|_{\Gamma_0}=0$ and $\mathcal{L}_{\overline{A_2}, \overline{q_2}}U_2=0$ with $U_2|_{\Gamma_0}=0$. Then plugging these solutions into (\ref{uau}) we shall obtain information about $A_1-A_2$. Firstly, we shall give the motivation behind the choice of $\rho_1$ and $\rho_2$. Given $\xi\in \mathbb{R}^n$, let $\mu_1$ and $\mu_2$ be unit vectors in $\mathbb{R}^n$ such that 
\begin{equation}\label{–n–z}
\xi\cdot \mu_{1}=\xi\cdot \mu_{2}=\mu_{1}\cdot \mu_{2}=0.
\end{equation}
Then, as in \cite{KU}, for a large parameter $\tau>0$ we set
\begin{equation}\label{jkw}
\begin{aligned}
&\rho_1= \frac{i}{2}\tau^{-1}\xi + i \sqrt{1-\tau^{-2} \frac{\left | \xi \right |^2}{4}}\mu_{1} +\mu_{2},\\
&\rho_2= -\frac{i}{2}\tau^{-1}\xi + i \sqrt{1-\tau^{-2} \frac{\left | \xi \right |^2}{4}}\mu_{1} -\mu_{2}.
\end{aligned}
\end{equation}
Observe that $\rho_1$ and $\rho_2$ can be written as follows:
\begin{equation}\label{jkww}
\rho_1= \rho_{1,0}+\mathcal{O}(\tau^{-1}), \, \rho_2=  \rho_{2,0}+\mathcal{O}(\tau^{-1}),\quad \left | \rho_1 \right | = \left | \rho_2 \right |= \left | \rho_1^* \right | = \left | \rho_2^* \right |=\sqrt{2},
\end{equation}
where
\begin{equation}\label{jkwww}
\rho_{1,0}=  i\mu_{1} + \mu_{2}, \quad \rho_{2,0}=  i\mu_{1} - \mu_{2}.
\end{equation}

Now, by Proposition \ref{construcsolutions}, for such $\rho_1$ there exists $U_1\in H^{1}(\Omega)$ of the form
\[
U_1= e^{\tau\rho_1\cdot x} \left( e^{\Phi_1^\sharp} +r_1  \right)- e^{\tau\rho_1^*\cdot x} \left( e^{\Phi_1^{\sharp^*}} +r_1^*  \right)
\]
satisfying $\mathcal{L}_{A_1,q_1}U_1=0$ and $U_1|_{\Gamma_0}=0$. Analogously, for $\rho_2$ there exists $U_2\in H^{1}(\Omega)$ of the form
\[
U_2= e^{\tau\rho_2\cdot x} \left( e^{\Phi_2^\sharp} +r_2  \right)- e^{\tau\rho_2^*\cdot x} \left( e^{\Phi_2^{\sharp^*}} +r_2^*  \right)
\]
satisfying $\mathcal{L}_{\overline{A_2}, \overline{q_2}}U_2=0$ and $U_2|_{\Gamma_0}=0$. In order to exploit the information, about $A_1-A_2$, encoded in the integral estimate (\ref{uau}), we have to compute $DU_1\overline{U_2}+ U_1\overline{DU_2}$ and $U_1\overline{U_2}$ and so we have to compute expressions of the form: 
\begin{align*}
e^{\tau(\rho_1+\overline{\rho_2})\cdot x} f_1(x)&=e^{i\tau\xi\cdot x}f_1(x),\\
e^{\tau(\rho_1^*+\overline{\rho_2^*})\cdot x} f_2(x)&=e^{i\tau\xi^*\cdot x}f_2(x),\\
e^{\tau(\rho_1+ \overline{\rho_2^*})\cdot x}f_3(x)&=e^{\left( \frac{i}{2} (\xi+\xi^*)+ i \sqrt{\tau^2- \frac{\left | \xi \right |^2}{4}}(\mu_{1}-\mu_{1}^*) +  \tau(\mu_{2}-\mu_{2}^*)\right)\cdot x}f_3(x),\\
e^{\tau(\rho_1^*+ \overline{\rho_2})\cdot x}f_4(x)&= e^{\left( \frac{i}{2} (\xi^*+\xi)- i  \sqrt{\tau^2- \frac{\left | \xi \right |^2}{4}}(\mu_{1}-\mu_{1}^*) - \tau (\mu_{2}-\mu_{2}^*)\right)\cdot x}f_4(x),
\end{align*}
for some suitable functions $f_1, f_2, f_3$ and $f_4$. The expressions involving $f_1$ and $f_2$ will give us information about the difference of the magnetic potentials $A_1-A_2$. To estimate the expressions involving $f_3$ and $f_4$ we will use a  quantitive version of the Riemann--Lebesgue lemma. Thus, we have to fix  a priori $\mu_{1}$ and $\mu_{2}$ satisfying:
\begin{equation}\label{z––}
\mu_2=\mu_2^*,
\end{equation}
\begin{equation}\label{–z}
\underset{\tau\to \infty}{\lim}\left |  \frac{1}{2} (\xi+\xi^*)+  \sqrt{\tau^2- \frac{\left | \xi \right |^2}{4}}(\mu_{1}-\mu_{1}^*)  \right | = +\infty,
\end{equation}
and
\begin{equation}\label{––z}
\underset{\tau\to \infty}{\lim}\left | \frac{1}{2} (\xi+\xi^*)-  \sqrt{\tau^2- \frac{\left | \xi \right |^2}{4}}(\mu_{1}-\mu_{1}^*)\right | = +\infty.
\end{equation}
To fix such $\mu_{1}$ and $\mu_{2}$ satisfying (\ref{z––})-(\ref{––z}) we proceed as in \cite{Is}.\\

Given $\xi=(\xi_1, \xi_2, \ldots, \xi_{n-1}, \xi_n)\in \mathbb{R}^n$, we denote $\xi^\prime=(\xi_1, \xi_2, \ldots, \xi_{n-1})$. Thus, we write $\xi=(\xi^\prime, \xi_n)$. Given $l=1,2, \ldots, n-1$; we set
\begin{equation}\label{Er}
E_{l}:= \left \{\xi=(\xi_1, \xi_2, \ldots, \xi_n)\in \mathbb{R}^n : 0<\underset{ \underset{k\neq l}{k=1} }{\sum^{n-1}}  \xi_k^2  \right \},
\end{equation}
and for each $\xi\in \bigcap_{l=1}^{n-1}E_{l}$, we consider the following unit vectors in $\mathbb{R}^n$: 
\begin{equation}\label{baor}
e(1):= \frac{1}{ \left | \xi^\prime \right |}(\xi^\prime, 0), \quad e(2) ,\quad e_n
\end{equation}
with
\begin{equation}\label{baort}
e(2)\in \left( \mbox{span}\left\{e(1), e_n  \right\}   \right)^{\perp} \quad \mbox{and}\quad e(2)=e(2)^*,
\end{equation}
where $e_n$ denotes the $n$-th canonical unit vector in $\mathbb{R}^n$. Notice that every $\xi\in \bigcap_{l=1}^{n-1}E_{l}$ can be written as $\xi= \left | \xi^\prime \right | e(1)+ \xi_n e_n$.
%Now for each $\xi\in E_R$, we consider three orthogonal unit vectors  $e(1), e(2)$ and $e_n$ as in (\ref{baor})-(\ref{baort}). Hence, it is clear that $\mu_1$ and $\mu_2$ defined as 
\begin{lem}\label{muum}
For each $\xi\in \bigcap_{l=1}^{n-1}E_{l}$ and given $j,k=1, \ldots, n$; there exist constants $\alpha, \beta$ and unit vectors $\mu_1, \mu_2$ satisfying (\ref{–n–z}), (\ref{z––})-(\ref{––z}), such that
\begin{equation}\label{abz1}
\xi_j e_k-\xi_k e_j= \alpha \mu_1 + \beta \mu_2, \quad j,k=1,2, \ldots,n,
\end{equation}
where $e_l$ denotes the  $l$-th canonical unit vector in $\mathbb{R}^n$. Moreover, $\mu_1$ can be chosen of the following form:
\begin{equation}\label{muuuuqs1}
\mu_{1}= -\frac{\xi_n}{\left | \xi \right |}e(1) + \frac{\left | \xi^\prime \right |}{\left | \xi \right |}e_n,
\end{equation}
independent of $j$ and $k$.
\end{lem}
\begin{proof}
Notice that for every $\xi\in \bigcap_{l=1}^{n-1}E_{l}$ we deduce that $\left | \xi^\prime \right |>0$ and $\left | \xi \right |>0$. So, $\mu_1$ in (\ref{muuuuqs1}) is well-defined. It is immediate that for $j,k=1,2, \ldots, n-1$; the unit vectors $\mu_1$ and $\mu_2$ defined by
\begin{equation}\label{zlk}
\mu_{1}:= -\frac{\xi_n}{\left | \xi \right |}e(1) + \frac{\left | \xi^\prime \right |}{\left | \xi \right |}e_n, \quad \mu_2=(\mu_2)_{j,k}:= {(\xi_j e_k-\xi_k e_j)}/{\xi_j^2+\xi_k^2}
\end{equation}
satisfy  (\ref{–n–z}) and (\ref{z––})-(\ref{––z}). Moreover, we have the following identity
\begin{equation}\label{xcbty1}
\xi_j e_k-\xi_k e_j = 0\mu_1+ ({\xi_j^2+\xi_k^2})\mu_2, \quad j,k=1,2, \ldots, n-1.
\end{equation}
Here $\alpha=0$ and $\beta= \xi_j^2+ \xi_k^2$.\\

It remains to prove (\ref{abz1}) for vectors of the form $\xi_j e_n-\xi_n e_j$ with $j=1,2, \ldots, n-1$. To prove that we consider $\mu_1$ as in (\ref{zlk}), from which we deduce that
\[
e_n=\dfrac{\left | \xi \right |}{\left | \xi^\prime \right |}\mu_1 + \dfrac{\xi_n}{\left | \xi^\prime \right |}e(1),
\]
and then we would like to find two constants, $\alpha$ and $\beta$, and one unit vector $\mu_2$ satisfying, together with $\mu_1$, the conditions (\ref{–n–z}), (\ref{z––})-(\ref{––z}); such that the following equality
\[
\xi_j e_n-\xi_n e_j= \dfrac{\xi_j\left | \xi \right |}{\left | \xi^\prime \right |}\mu_1 + \dfrac{\xi_j\xi_n}{\left | \xi^\prime \right |}e(1)-\xi_ne_j= \alpha \mu_1 + \beta \mu_2, \quad j=1,2, \ldots, n-1,
\]
holds true. Since $\mu_1$ and $\mu_2$ have to be orthogonal, from the above identity and by a standard computations we deduce that
\[
\alpha= \dfrac{\xi_j\left | \xi \right |}{\left | \xi^\prime \right |}, \quad \beta= \dfrac{\xi_n}{\left | \xi^\prime \right |}\left(  \left | \xi^\prime \right |^2-\xi_j^2  \right)^{1/2}
\]
and
\[
\mu_2=(\mu_2)_{j,n}:= \left( \left | \xi^\prime \right |^2 - \xi_j^2  \right)^{-1/2} \left( \xi_j e(1)-  \left | \xi^\prime \right | e_j \right), \quad j=1,2, \ldots, n-1.
\]
It is easy to check that such vectors satisfy, together with $\mu_1$, the required conditions  (\ref{–n–z}) and (\ref{z––})-(\ref{––z}). Thus, we conclude the proof.
\end{proof}

From now on, unless otherwise stated, we consider $\rho_1$ and $\rho_2$ as in (\ref{jkw}) with $\mu_1$ and $\mu_2$ given by Lemma \ref{muum}. Thus, we have the following equalities:
\begin{equation}\label{prs}
\begin{aligned}
&\tau(\rho_1+\overline{\rho_2})\cdot x= i\xi\cdot x, \quad  \tau(  \rho_1^*+\overline{\rho_2^*} )\cdot x= i\xi^*\cdot x,    \\
&\tau( \rho_1 + \overline{\rho_2^*})\cdot x = i\left( \xi^\prime, 2 \sqrt{\tau^{2}- \frac{\left | \xi \right |^2}{4}} \frac{ \left |  \xi^\prime \right |}{ \left |  \xi \right |} \right)\cdot x, \\
&\tau( \rho_1^* + \overline{\rho_2} )\cdot x=  i\left( \xi^\prime, -2 \sqrt{\tau^{2}- \frac{\left | \xi \right |^2}{4}} \frac{ \left |  \xi^\prime \right |}{ \left |  \xi \right |} \right)\cdot x.
\end{aligned}
\end{equation}

\subsection{A Fourier estimate for the magnetic potentials} This section will be devoted to prove the following proposition.
\begin{prop}\label{ggggg5}
Let $\Omega\subset \mathbb{R}^n$ be a bounded open set.  Let $A_1, A_2\in L^{\infty}(\Omega, \mathbb{C}^n)$ and $q_1, q_2\in L^{\infty}(\Omega, \mathbb{C})$. Given $M>0$ and $s\in \left( 0,1/2 \right)$, assume that $\chi_\Omega A_1$ and $\chi_\Omega A_2$ belong to $\mathcal{A}(\Omega, M,s)$. Then there exist three positive constants $C$, $\tau_0$ and $\epsilon_0$ (all depending on $\Omega, n,M, s, \left \| q_1 \right \|_{L^\infty}, \left \| q_2 \right \|_{L^\infty}$) such that the following estimate:
\begin{equation}\label{final2} 
\begin{aligned}
& \left | \mathcal{F}\left [ d(\chi_{\Omega\cup\Omega^*}\widetilde{A_1})  \right ](\xi) - \mathcal{F}\left [ d(\chi_{\Omega\cup\Omega^*}\widetilde{A_2})  \right ](\xi)   \right | \\
&\quad  \leq C  \left | \xi \right |\left[  \tau^{-s/(s+2)}  + e^{2\tau \kappa}  dist \, (C_1^\Gamma, C_2^\Gamma) +\tau^{s/(s+2)}  \left( e^{-4 \pi \epsilon^2 \tau^2 \frac{\left |\xi^\prime  \right |^2}{\left |  \xi \right |^2} }  + \epsilon^s \right) \right],
\end{aligned}
\end{equation}
holds true for all $\xi\in \bigcap_{l=1}^{n-1}E_{l}$, $\tau\geq \tau_0$ and for all $0<\epsilon< \epsilon_0$.
\end{prop}

\begin{rem}
We emphasize that the constant $C$ is independent of $\xi$. To prove this proposition we will use two known results. The first one was obtained by Heck and Wang (see Lemma $2.1$ in \cite{HW1}). It is a quantitative version of the Riemann--Lebesgue lemma. 
\end{rem}
\begin{lem}\label{RLe}
Assume that $f\in L^1(\mathbb{R}^n)$ and there exist $\sigma >0, C_0>0$, and $s\in \left( 0, 1\right)$ such that
\begin{equation}\label{xzse1}
\left \| f(\cdot +y) -f(\cdot) \right \|_{L^1(\mathbb{R}^n)} \leq C_0\left | y \right |^s
\end{equation}
whenever $\left | y \right |<\sigma$. Then there exist two positive constants $K$ and $\epsilon_0$ such that for any $0<\epsilon<\epsilon_0$, the inequality
\[
\left | \widehat{f}(\xi) \right | \leq C_0 K ( e^{-\pi \epsilon^2 \left | \xi \right |^2} +\epsilon^s),
\]
holds true with $K=K(\left \|f  \right \|_{L^1},n, \sigma, s)$.
\end{lem}
The second result is a well known result on nonlinear Fourier transform. For a proof see Proposition $3.3$ in \cite{KU} and  also Lemma $2.6$ in \cite{Tz}. 
\begin{lem}\label{KyUh}
Let $ \xi, \mu_1, \mu_2\in \mathbb{R}^n$ ($n\geq 3$) be orthogonal vectors such that $\left | \mu_2\right |=\left |\mu_1\right |=1$. If $W\in (L^{\infty}\cap \mathcal{E}^\prime  )(\mathbb{R}^n; \mathbb{C}^n)$ and  $\Phi$ satisfies 
\[
(  i\mu_1+ \mu_2)\cdot \nabla \Phi +  ( i\mu_1+  \mu_2)\cdot W=0
\]
in $\mathbb{R}^n$ then 
\[
(  i\mu_1+ \mu_2)\cdot \int_{\mathbb{R}^n} W(x)e^{i\xi \cdot x} e^{\Phi(x)}dx= ( i\mu_1+  \mu_2)\cdot \int_{\mathbb{R}^n} W(x)e^{i\xi \cdot x}dx.
\]
\end{lem}
We are now in position to prove Proposition \ref{ggggg5}.

\begin{proof}
We shall start by computing the right-hand side of  (\ref{uau}) multiplied by $\tau^{-1}$, i.e. the task is now to estimate the expression
\[
\tau^{-1} \int_{\Omega} \left [   (A_1-A_2)\cdot(DU_1 \overline{U_2} + U_1 \overline{DU_2}) + (A_1^2-A_2^2+q_1-q_2)U_1\overline{U_2} \right ]dx,
\]
using the solutions $U_1, U_2\in H^{1}(\Omega)$ given by Proposition \ref{construcsolutions}. More precisely, for $A_1, q_1$ and $\rho_1$ given by (\ref{jkw}), Proposition \ref{construcsolutions} ensures the existence of a function $U_1\in H^1(\Omega)$ satisfying $\mathcal{L}_{A_1, q_1} U_1=0$ in $\Omega$ with $U_1|_{\Gamma_0}=0$ having the form:
\begin{equation}\label{unos}
U_1(x)= e^{\tau \rho_1\cdot x} \left(   e^{\Phi_1^\sharp} +r_1  \right)- e^{\tau \rho_1^*\cdot x} \left(   e^{\Phi_1^{\sharp^*}} +r_1^* \right).
\end{equation}
Analogously, by Proposition \ref{construcsolutions} applied to $\overline{A_2}, \overline{q_2}$ and $\rho_2$ defined by (\ref{jkw}), there exists $U_2\in H^{1}(\Omega)$ satisfying $\mathcal{L}_{\overline{A_2}, \overline{q_2}}U_2=0$ in $\Omega$ with $U_2|_{\Gamma_0}=0$ having the form:
\begin{equation}\label{doss}
U_2(x)= e^{\tau \rho_2\cdot x} \left(   e^{\Phi_2^\sharp} +r_2  \right)- e^{\tau \rho_2^*\cdot x} \left(   e^{\Phi_2^{\sharp^*}} +r_2^* \right).
\end{equation}
Both solutions have the following properties. The functions $\Phi_1^\sharp(\cdot, \rho_{0,1}; \tau)$ and $\Phi_2^\sharp(\cdot, \rho_{0,2}; \tau)$ belong to  $C^{\infty}(\mathbb{R}^n)$ and satisfy for all $\alpha\in \mathbb{N}^n$
\begin{equation}\label{prunozss}
\left \| \partial^{\alpha} \Phi^\sharp_i\right \|_{L^\infty(\mathbb{R}^n)} +   \left \| \partial^{\alpha} {\Phi^\sharp_i}^*\right \|_{L^\infty(\mathbb{R}^n)}  \leq C \tau^{\left | \alpha \right |/(s+2)}, \quad  \tau\geq \tau_0, \, i=1,2.
\end{equation}
For $i=1,2$, the functions $r_i$ and  $r_i^*$ belong to $H^{1}(\Omega\cup\Omega^*)$ and satisfy
\begin{equation}\label{prdoszss}
\left \| \partial^{\alpha} r_i\right \|_{L^2(\Omega\cup\Omega^*)}+ \left \| \partial^{\alpha} r_i^*\right \|_{L^2(\Omega\cup\Omega^*)}\leq C \tau^{\left | \alpha \right |-s/(s+2)}, \quad \left |\alpha  \right |\leq 1.
\end{equation}
Moreover, from (\ref{jkww}), we get
\begin{equation}\label{jaci11}
\left \| U_i \right \|_{H^1(\Omega)}\leq C\, e^{\tau \kappa   \left | \rho \right |  }\leq C e^{\tau\kappa}, \quad i=1,2.
\end{equation}
Also, from (\ref{mmmmmz}), we denote by  $\Phi_1(\cdot; \rho_{1,0})= (\rho_{1,0}\cdot \nabla)^{-1} (-i\rho_{1,0}\cdot (\chi_{\Omega\cup\Omega^*} \widetilde{A_1}))\in L^{\infty}(\mathbb{R}^n)$ the function satisfying the equation in $\mathbb{R}^n$
\begin{equation}\label{mmmmmz12}
\rho_{1,0}\cdot \nabla \Phi_1 + i\rho_{1,0}\cdot (\chi_{\Omega\cup\Omega^*} \widetilde{A_1})=0
\end{equation}
and 
$\Phi_2(\cdot; \rho_{2,0})= (\rho_{2,0}\cdot \nabla)^{-1} (-i\rho_{2,0}\cdot (\overline{\chi_{\Omega\cup\Omega^*} \widetilde{A_2}}))\in L^{\infty}(\mathbb{R}^n)$ the function satisfying the equation in $\mathbb{R}^n$
\begin{equation}\label{mmmmmz123}
\rho_{2,0}\cdot \nabla \Phi_2 + i\rho_{2,0}\cdot (\overline{\chi_{\Omega\cup\Omega} \widetilde{A_2}})=0.
\end{equation}
From (\ref{prtresz}), both functions satisfy the estimate
\begin{equation}\label{prtresz15}
\left \| \Phi_i (\cdot; \rho_{i, 0})\right \|_{L^\infty(\mathbb{R}^n)} \leq C, \quad i=1,2.
\end{equation}
Finally, from (\ref{prcuatroz}), for every $\chi\in C^{\infty}_0(\mathbb{R}^n)$ we have
\begin{equation}\label{zuno1}
\left \| \chi (\Phi^\sharp_i(\cdot, \rho_0;\tau)- \Phi_i(\cdot; \rho_0)) \right \|_{L^2(\mathbb{R}^n)}\leq C_1 \tau^{-s/(s+2)}, \quad i=1,2.
\end{equation}

With these solutions at hand and by a straightforward  computation, we get
\begin{equation}\label{lilax1}
\begin{aligned}
& \tau^{-1} \int_{\Omega}  (A_1-A_2)\cdot(DU_1 \overline{U_2} + U_1 \overline{DU_2})  \\
%&=  + i\int_{\Omega}  (\overline{\rho_2}-\rho_1)\cdot(A_1-A_2) e^{\tau(\rho_1+\overline{\rho_2})\cdot x} e^{\Phi_1^\sharp+ \overline{\Phi_2^\sharp}} \\
%& + i\int_{\Omega} (\overline{\rho_2^*}-\rho_1^*)\cdot (A_1-A_2) e^{\tau(\rho_1^*+\overline{\rho_2^*})\cdot x} e^{\Phi_1^{\sharp^*}+ \overline{\Phi_2^{\sharp^*}}}\\
%&+ i\int_{\Omega} (\rho_1- \overline{\rho_2^*})\cdot (A_1-A_2) e^{\tau(\rho_1+\overline{\rho_2^*})\cdot x}  e^{\Phi_1^{\sharp}+ \overline{\Phi_2^{\sharp^*}}} \\
%&   +i\int_{\Omega}(\rho_1^*- \overline{\rho_2})\cdot (A_1-A_2) e^{\tau(\rho_1^*+\overline{\rho_2})\cdot x}  e^{\Phi_1^{\sharp^*}+ \overline{\Phi_2^{\sharp}}}  + \int_{\Omega} R\cdot (A_1-A_2)\\
& = i\int_\Omega (\overline{\rho_2}-\rho_1)\cdot (A_1-A_2)e^{i(\rho_1+\overline{\rho_2})\cdot x} e^{\Phi_1^\sharp+ \overline{\Phi_2^\sharp}}\\
& + i \int_\Omega (\overline{\rho_2^*}-\rho_1^*)\cdot (A_1-A_2)e^{i(\rho_1^*+\overline{\rho_2^*})\cdot x} e^{\Phi_1^{\sharp^*}+ \overline{\Phi_2^{\sharp^*}}}\\   
& + i\int_{\Omega} (\rho_1- \overline{\rho_2^*})\cdot (A_1-A_2) e^{i(\ \rho_1 + \overline{\rho_2^*} )\cdot x}  e^{\Phi_1^{\sharp}+ \overline{\Phi_2^{\sharp^*}}}\\
& + i\int_{\Omega} (\rho_1^*- \overline{\rho_2})\cdot (A_1-A_2) e^{i( \rho_1^* + \overline{\rho_2} )\cdot x}  e^{\Phi_1^{\sharp^*}+ \overline{\Phi_2^{\sharp}}} + \int_{\Omega} R\cdot (A_1-A_2),
\end{aligned}
\end{equation}
where $R$ denotes the following expression:
\begin{align*}
& R =   i(\overline{\rho_2}-\rho_1)  e^{\tau(\rho_1+\overline{\rho_2})\cdot x}    ( e^{\Phi_1^\sharp} \overline{r_2} +r_1 \overline{e^{\Phi_2^\sharp}} +r_1\overline{r_2}    ) \\
&\quad  +  i(\overline{\rho_2^*}-\rho_1^*)e^{\tau(\rho_1^*+\overline{\rho_2^*})\cdot x} ( e^{\Phi_1^{\sharp^*}} \overline{r_2^*} +r_1^* \overline{e^{\Phi_2^{\sharp^*}}} +r_1^* \overline{r_2^*}    )\\
& \quad  + i (\rho_1-\overline{\rho_2^*})  e^{\tau(\rho_1+\overline{\rho_2^*})\cdot x} (e^{\Phi_1^\sharp} \overline{r_2^*} +r_1 \overline{e^{\Phi_2^{\sharp^*}} } + r_1\overline{r_2^*})\\
& \quad   + i (\rho_1^*-\overline{\rho_2})  e^{\tau(\rho_1^*+\overline{\rho_2})\cdot x} (e^{\Phi_1^{\sharp^*}} \overline{r_2} +r_1^* \overline{e^{\Phi_2^{\sharp}} } + r_1^*\overline{r_2})\\
&\quad+ i\tau^{-1}e^{\tau(\rho_1+\overline{\rho_2})\cdot x} \left [   (e^{\Phi_1^\sharp}+r_1) \overline{ \nabla( e^{\Phi_2^\sharp} + r_2) }   -(\overline{ e^{\Phi_2^\sharp} +r_2 }  ) \nabla (e^{\Phi_1^\sharp} + r_1)       \right ]\\
&\quad  + i\tau^{-1}e^{\tau(\rho_1^*+\overline{\rho_2^*})\cdot x} \left [   (e^{\Phi_1^{\sharp^*}}+r_1^*) \overline{ \nabla( e^{\Phi_2^{\sharp^*}} + r_2^*) }   -(\overline{ e^{\Phi_2^{\sharp^*}} +r_2^* }  ) \nabla (e^{\Phi_1^{\sharp^*}} + r_1^*)       \right ] \\
&\quad  + i\tau^{-1} e^{\tau(\rho_1+ \overline{\rho_2^*})\cdot x}  \left [  ( \overline{e^{\Phi_2^{\sharp^*}}+r_2^*})   \nabla (e^{\Phi_1^\sharp} +r_1)  - (e^{\Phi_1^\sharp} +r_1) \overline{\nabla ( e^{\Phi_2^{\sharp^*}}  +r_2^*  )}  \right ]\\
&\quad  + i\tau^{-1} e^{\tau(\rho_1^*+ \overline{\rho_2})\cdot x}  \left [  ( \overline{e^{\Phi_2^{\sharp}}+r_2})   \nabla (e^{\Phi_1^{\sharp^*}} +r_1^*)  - (e^{\Phi_1^{\sharp^*}} +r_1^*) \overline{\nabla ( e^{\Phi_2^{\sharp}}  +r_2  )}  \right ].
\end{align*} 
Since we have done an even extension for $A_i^{(j)}$ with $j=1,2, \ldots, n-1$ and odd extension for $A_i^{(n)}$ for $i=1,2$; we have
%\begin{equation}\label{lilax2}
\begin{align*}
&   i\int_\Omega (\overline{\rho_2}-\rho_1)\cdot (A_1-A_2)e^{i(\rho_1+\overline{\rho_2})\cdot x} e^{\Phi_1^\sharp+ \overline{\Phi_2^\sharp}}\\
& + i \int_\Omega (\overline{\rho_2^*}-\rho_1^*)\cdot (A_1-A_2)e^{i(\rho_1^*+\overline{\rho_2^*})\cdot x} e^{\Phi_1^{\sharp^*}+ \overline{\Phi_2^{\sharp^*}}}\\
& =  i\int_{\mathbb{R}^n}  (\overline{\rho_{2}}-\rho_{1})\cdot \left[ \chi_{\Omega\cup\Omega^*}(\widetilde{A_1}-\widetilde{A_2})\right]e^{i\xi\cdot x} e^{\Phi_1^\sharp+ \overline{\Phi_2^\sharp}}.\\
\end{align*}
%\end{equation}
Replacing this equality, (\ref{jkw})-(\ref{jkwww}) and (\ref{prs}) into (\ref{lilax1}), we obtain
\begin{align*}
& \tau^{-1} \int_{\Omega}  (A_1-A_2)\cdot(DU_1 \overline{U_2} + U_1 \overline{DU_2})  \\
& =  i\int_{\mathbb{R}^n}  (\overline{\rho_{2}}-\rho_{1})\cdot \left[ \chi_{\Omega\cup\Omega^*}(\widetilde{A_1}-\widetilde{A_2})\right]e^{i\xi\cdot x} e^{\Phi_1^\sharp+ \overline{\Phi_2^\sharp}}  + \int_{\Omega} R\cdot (A_1-A_2) \\
&\quad + i\int_{\Omega}                                                 (\rho_1- \overline{\rho_2^*})\cdot(A_1-A_2) e^{i\left( \xi^\prime, 2\tau  \sqrt{1-\tau^{-2} \frac{\left | \xi \right |^2}{4}} \frac{ \left |  \xi^\prime \right |}{ \left |  \xi \right |} \right)\cdot x} e^{\Phi_1^{\sharp}+ \overline{\Phi_2^{\sharp^*}}} \\
&\quad + i\int_{\Omega} (\rho_1^*- \overline{\rho_2})\cdot (A_1-A_2) e^{i\left( \xi^\prime, -2\tau  \sqrt{1-\tau^{-2} \frac{\left | \xi \right |^2}{4}} \frac{ \left |  \xi^\prime \right |}{ \left |  \xi \right |} \right)\cdot x} e^{\Phi_1^{\sharp^*}+ \overline{\Phi_2^{\sharp}}} \\
& =  i\int_{\mathbb{R}^n}  (\overline{\rho_{2,0}}-\rho_{1,0})\cdot \left[ \chi_{\Omega\cup\Omega^*}(\widetilde{A_1}-\widetilde{A_2})\right] e^{i\xi\cdot x} e^{\Phi_1^\sharp+ \overline{\Phi_2^\sharp}}\\
& \quad+ \int_{\mathbb{R}^n} \mathcal{O}(\tau^{-1}) \cdot \left[ \chi_{\Omega\cup\Omega^*}(\widetilde{A_1}-\widetilde{A_2})\right]  e^{i\xi\cdot x} e^{\Phi_1^\sharp+ \overline{\Phi_2^\sharp}} + \int_{\Omega} R\cdot (A_1-A_2) \\
&\quad + i\int_{\mathbb{R}^n}                                                 (\rho_1- \overline{\rho_2^*})\cdot (\chi_\Omega(A_1-A_2)) e^{i\left( \xi^\prime, 2\tau  \sqrt{1-\tau^{-2} \frac{\left | \xi \right |^2}{4}} \frac{ \left |  \xi^\prime \right |}{ \left |  \xi \right |} \right)\cdot x}  e^{\Phi_1^{\sharp}+ \overline{\Phi_2^{\sharp^*}}}\\
&\quad  + i\int_{\mathbb{R}^n} (\rho_1^*- \overline{\rho_2})\cdot (\chi_\Omega(A_1-A_2)) e^{i\left( \xi^\prime, -2\tau  \sqrt{1-\tau^{-2} \frac{\left | \xi \right |^2}{4}} \frac{ \left |  \xi^\prime \right |}{ \left |  \xi \right |} \right)\cdot x} e^{\Phi_1^{\sharp^*}+ \overline{\Phi_2^{\sharp}}}.
\end{align*}
Hence, from this identity and also adding and subtracting terms, we get 
\begin{equation}\label{lilax3}
\begin{aligned}
 & i\int_{\mathbb{R}^n}  (\overline{\rho_{2,0}}-\rho_{1,0})\cdot  \left[ \chi_{\Omega\cup\Omega^*}(\widetilde{A_1}-\widetilde{A_2})\right] e^{i\xi\cdot x} e^{\Phi_1+ \overline{\Phi_2}}\\
 & \qquad  = I + II + III + IV + V+VI + VII,
 \end{aligned}
\end{equation}
where
\begin{equation}\label{lilax4}
I=   i\int_{\mathbb{R}^n}  (\overline{\rho_{2,0}}-\rho_{1,0})\cdot \left[ \chi_{\Omega\cup\Omega^*}(\widetilde{A_1}-\widetilde{A_2})\right] e^{i\xi\cdot x} ( e^{\Phi_1+ \overline{\Phi_2}} -  e^{\Phi_1^\sharp+ \overline{\Phi_2^\sharp}} ), 
\end{equation}
\begin{equation}\label{lilax5}
\begin{aligned}
II&= \tau^{-1} \int_{\Omega} \left [ (A_1-A_2)\cdot(DU_1 \overline{U_2} + U_1 \overline{DU_2}) \right.\\
&\qquad \qquad \qquad  \left. + (A_1^2-A_2^2+q_1-q_2)U_1\overline{U_2}\right],
\end{aligned}
\end{equation}
\begin{equation}\label{lilax6}
III=-\tau^{-1} \int_\Omega  (A_1^2-A_2^2+q_1-q_2)U_1\overline{U_2}, 
\end{equation}
\begin{equation}\label{lilax7}
IV= -  \int_{\mathbb{R}^n} \mathcal{O}(\tau^{-1}) \cdot  \left[ \chi_{\Omega\cup\Omega^*}(\widetilde{A_1}-\widetilde{A_2})\right]  e^{i\xi\cdot x} e^{\Phi_1^\sharp+ \overline{\Phi_2^\sharp}},
\end{equation}
\begin{equation}\label{lilax8}
V= \int_{\Omega} R\cdot (A_1-A_2),
\end{equation}
\begin{equation}\label{lilax9}
VI=-  i\int_{\mathbb{R}^n} (\rho_1- \overline{\rho_2^*})\cdot (\chi_\Omega(A_1-A_2)) e^{i\left( \xi^\prime, 2\tau  \sqrt{1-\tau^{-2} \frac{\left | \xi \right |^2}{4}} \frac{ \left |  \xi^\prime \right |}{ \left |  \xi \right |} \right)\cdot x}  e^{\Phi_1^{\sharp}+ \overline{\Phi_2^{\sharp^*}}},
\end{equation}
\begin{equation}\label{lilax10}
VII= - i\int_{\mathbb{R}^n} (\rho_1^*- \overline{\rho_2})\cdot (\chi_\Omega(A_1-A_2)) e^{i\left( \xi^\prime, -2\tau  \sqrt{1-\tau^{-2} \frac{\left | \xi \right |^2}{4}} \frac{ \left |  \xi^\prime \right |}{ \left |  \xi \right |} \right)\cdot x} e^{\Phi_1^{\sharp^*}+ \overline{\Phi_2^{\sharp}}}.
\end{equation}
The task is now to estimate each one of the above terms. To estimate the first term $I$ we will use the following fact:
\begin{equation}\label{ineq}
\left | e^{z_1}- e^{z_2} \right |\leq \left | z_1 -z_2 \right |e^{\max\left \{ \Re z_1, \Re z_2 \right \}},
\end{equation}
for all $z_1,z_2\in \mathbb{C}$. Thus, from (\ref{jkwww}), the boundedness of $\Omega\cup\Omega^*$ and (\ref{zuno1}), we obtain
\begin{equation}\label{two2}
\begin{aligned}
 \left | I \right | 
 %= \left | i\int_{\mathbb{R}^n}  (\overline{\rho_{2,0}}-\rho_{1,0})\cdot \left[ \chi_{\Omega\cup\Omega^*}(\widetilde{A_1}-\widetilde{A_2})\right]  e^{i\xi\cdot x} ( e^{\Phi_1+ \overline{\Phi_2}} -  e^{\Phi_1^\sharp+ \overline{\Phi_2^\sharp}} )  \right | \\
& \leq C_1  \left \|  \Phi_1 - \Phi_1^{\sharp^*}+ \overline{\Phi_2}-  \Phi_2^{\sharp^*} \right \|_{L^2(\mathbb{R}^n)} \leq C_2\,  \tau^{-s/(s+2)}.
\end{aligned}
\end{equation}
From Corollary \ref{coAlid} and (\ref{jaci11}), we obtain
\begin{equation}\label{three3}
\begin{aligned}
& \left | I I\right |  
%=\left |   \tau^{-1}\int_{\Omega} (A_1-A_2)\cdot(DU_1 \overline{U_2} + U_1 \overline{DU_2}) \right.\\
%& \qquad \qquad \qquad \left.+ (A_1^2-A_2^2+q_1-q_2)U_1\overline{U_2}     \right |\\
 \leq C_3 \tau^{-1} dist \, (C_1^\Gamma, C_2^\Gamma) \left \| U_1 \right \|_{H^1(\Omega)} \left \| U_2 \right \|_{H^1(\Omega)} \\
& \qquad  \leq C_4 \tau^{-1} e^{2\tau \kappa }  dist \, (C_1^\Gamma, C_2^\Gamma).
\end{aligned}
\end{equation}
We continue in this fashion to estimate the other terms. The identities (\ref{unos})-(\ref{doss}), (\ref{jkw}), (\ref{zlk}), the estimates (\ref{prunozss})-(\ref{prdoszss}) and triangular inequality, imply that 
\begin{equation}\label{four4}
\begin{aligned}
 & \left | II I\right |  = 
 %\left |  -  \tau^{-1}  \int_\Omega  (A_1^2-A_2^2+q_1-q_2)U_1\overline{U_2}     \right |\\
    \tau^{-1} \left |   \int_\Omega  (A_1^2-A_2^2+q_1-q_2)e^{\tau\overline{\rho_2}\cdot x}U_1(  \overline{e^{-\tau \rho_2\cdot x}U_2 } )    \right |\\
  &  \leq C_6\tau^{-1}   \left \| e^{\tau\overline{\rho_2}\cdot x}U_1 \right \|_{L^2(\Omega)}    \left \| e^{-\tau\rho_2\cdot x}U_2 \right \|_{L^2(\Omega)}\\
  &= C_6 \tau^{-1}   \left \| e^{\tau\overline{\rho_2}\cdot x}  \left[ e^{\tau \rho_1\cdot x} \left(   e^{\Phi_1^\sharp} +r_1  \right)- e^{\tau \rho_1^*\cdot x} \left(   e^{\Phi_1^{\sharp^*}} +r_1^* \right) \right]\right \|_{L^2(\Omega)} \\
  &   \times  \left \| e^{-\tau\rho_2\cdot x}  \left[ e^{\tau\rho_2\cdot x} (e^{\Phi_2^\sharp} +r_2) - e^{\tau \rho_2^*\cdot x} \left(   e^{\Phi_2^{\sharp^*}} +r_2^* \right)\right] \right \|_{L^2(\Omega)}\\
  & = C_6 \tau^{-1}   \left \| e^{i \xi\cdot x}   \left(   e^{\Phi_1^\sharp} +r_1  \right)- e^{i  \xi^*\cdot x} \left(   e^{\Phi_1^{\sharp^*}} +r_1^* \right) \right \|_{L^2(\Omega)} \\
  &   \times  \left \| e^{\Phi_2^\sharp} +r_2 - e^{\tau (\rho_2^*-\rho_2)\cdot x} \left(   e^{\Phi_2^{\sharp^*}} +r_2^* \right)\right \|_{L^2(\Omega)}\\
  & \leq C_6\tau^{-1}  \left( \left \|  e^{\Phi_1^\sharp} +r_1 \right \|_{L^2(\Omega)} +  \left \|   e^{\Phi_1^{\sharp^*}} +r_1^* \right \|_{L^2(\Omega)} \right. \\
  & \qquad  \qquad   \left. + \left \| e^{\Phi_2^\sharp} +r_2 \right \|_{L^2(\Omega)} + \left \| e^{\Phi_2^{\sharp^*}} +r_2^* \right \|_{L^2(\Omega)}  \right) \leq  C_7\, \tau^{-1},
\end{aligned}
\end{equation}
where in the last line we have used the identity
\[
\rho_2^*-\rho_2= i\left(  -\frac{1}{2}(\xi +\xi^*)-2 \sqrt{1-\tau^{-2} \frac{\left | \xi \right |^2}{4}} \frac{\left | \xi^\prime \right |}{\left | \xi \right |}e(n) \right).
\]
Again from (\ref{prunozss}) and the boundedness of $\Omega\cup\Omega^*$, it follows easily that
\begin{equation}\label{five5}
\left | IV \right | = \left |  -  \int_{\mathbb{R}^n} \mathcal{O}(\tau^{-1}) \cdot \left[ \chi_{\Omega\cup\Omega^*}(\widetilde{A_1}-\widetilde{A_2})\right]   e^{i\xi\cdot x} e^{\Phi_1^\sharp+ \overline{\Phi_2^\sharp}}  \right |   \leq C_8\,  \tau^{-1}.
\end{equation}
From (\ref{jkw}), (\ref{prs}) and (\ref{prunozss})-(\ref{prdoszss}), we get
\begin{equation}\label{six6}
\left | V \right |= \left |  \int_{\Omega} R\cdot (A_1-A_2) \right |\leq C_{10}\, \tau^{- s/(s+2)}.
\end{equation}
To estimate the term $VI$ we will use an extra argument due to Heck and Wang \cite{HW1}. From (\ref{jkww}), we have
\begin{equation}\label{cinco}
\begin{aligned}
 \left | VI \right |&
 %= \left | -  i\int_{\mathbb{R}^n} (\rho_1- \overline{\rho_2^*})\cdot (\chi_\Omega(A_1-A_2)) e^{i\left( \xi^\prime, 2\tau  \sqrt{1-\tau^{-2} \frac{\left | \xi \right |^2}{4}} \frac{ \left |  \xi^\prime \right |}{ \left |  \xi \right |} \right)\cdot x}  e^{\Phi_1^{\sharp}+ \overline{\Phi_2^{\sharp^*}}}   \right |\\
 \leq \left | \rho_1- \overline{\rho_2^*} \right |   \left |  \int_{\mathbb{R}^n}   (\chi_\Omega(A_1-A_2)) e^{i\left( \xi^\prime, 2\tau  \sqrt{1-\tau^{-2} \frac{\left | \xi \right |^2}{4}} \frac{ \left |  \xi^\prime \right |}{ \left |  \xi \right |} \right)\cdot x}  e^{\Phi_1^{\sharp}+ \overline{\Phi_2^{\sharp^*}}}   \right |\\
& \leq 2\sqrt{2} \left |  \mathcal{F} \left[ \chi_\Omega(A_1-A_2) e^{\Phi_1^{\sharp}+ \overline{\Phi_2^{\sharp^*}}} \right]  \left( \xi^\prime, 2\tau  \sqrt{1-\tau^{-2} \frac{\left | \xi \right |^2}{4}} \frac{ \left |  \xi^\prime \right |}{ \left |  \xi \right |} \right) \right |.
%&\leq  C_9  \left | \xi \right | (e^{-\pi \epsilon^2 (1+4\tau^2)\left | \xi^\prime \right |}  + \epsilon^\alpha).
\end{aligned}
\end{equation}
Now in order to apply  Lemma \ref{RLe} to the function $\chi_\Omega(A_1-A_2) e^{\Phi_1^{\sharp}+ \overline{\Phi_2^{\sharp^*}}}$ we claim such function satisfies the condition (\ref{xzse1}) of Lemma \ref{RLe} with $\sigma=1$. Indeed, if we denote by $\Phi^\sharp=\Phi_1^{\sharp}+ \overline{\Phi_2^{\sharp^*}}$ then from (\ref{prunozss}) and by a standard interpolation between the spaces $C(\mathbb{R}^n)$ and $C^1(\mathbb{R}^n)$ we get 
\begin{equation}\label{onuevt}
\left \| \Phi^\sharp \right \|_{C^s(\mathbb{R}^n)}\leq C_{10} \left \| \Phi^\sharp \right \|_{C(\mathbb{R}^n)}^{1-s} \left \| \Phi^\sharp \right \|_{C^1(\mathbb{R}^n)}^s\leq C_{11}\, \tau^{s/(s+2)}.
\end{equation}
For convenience we denote $A:= \chi_\Omega(A_1-A_2)$. Thus for any $y\in \mathbb{R}^n$, Cauchy-Schwarz inequality and  (\ref{ineq}) imply that
\begin{align*}
&\left \| \left [   \chi_\Omega(A_1-A_2) e^{\Phi_1^{\sharp}+ \overline{\Phi_2^{\sharp^*}}}\right ](\cdot +y) -\left[  \chi_\Omega(A_1-A_2) e^{\Phi_1^{\sharp}+ \overline{\Phi_2^{\sharp^*}}}  \right] (\cdot)  \right \|_{L^1(\mathbb{R}^n)} \\
& = \left \| (  A e^{\Phi^{\sharp}}) (\cdot +y) - (  A e^{\Phi^{\sharp}}) (\cdot) \right \|_{L^1(\mathbb{R}^n)} \\
&= \int_{\mathbb{R}^n} \left |  \left [ A(x +y)-A(x) \right ]e^{\Phi^\sharp(x+y)}   + \left[ e^{\Phi^\sharp(x+y)}- e^{\Phi^\sharp(x)}\right] A(x)  \right |dx\\
& \leq \int_{\mathbb{R}^n} \left |  \left [ A(x +y)-A(x) \right ]e^{\Phi^\sharp(x+y)}    \right |dx + \int_{\mathbb{R}^n} \left |  \left[ e^{\Phi^\sharp(x+y)}- e^{\Phi^\sharp(x)}\right] A(x)  \right |dx\\
& \leq C_{13}\left( \left \| A(\cdot +y)-A(\cdot)  \right \|_{L^2(\mathbb{R}^n)} + \left \| \Phi^\sharp (\cdot +y)-\Phi^\sharp (\cdot)  \right \|_{L^\infty(\mathbb{R}^n)}   \right)\\
& \leq C_{14} \left(  \left \| A \right \|_{B^{2,\infty}_s} + \left \| \Phi^\sharp \right \|_{C^s(\mathbb{R}^n)}     \right)  \left |  y\right |^s. 
\end{align*}
Then, by combining the above inequality with (\ref{onuevt}), we obtain
\[
 \left \| (  A e^{\Phi^{\sharp}}) (\cdot +y) - (  A e^{\Phi^{\sharp}}) (\cdot) \right \|_{L^1(\mathbb{R}^n)}\leq C_{15}\, \tau^{s/(s+2)} \left | y \right |^s,
\]
for any $y\in \mathbb{R}^n$ with $\left | y \right |<1$. So the claim have been proved.\\

Hence, applying  Lemma \ref{RLe} to $f= \chi_\Omega(A_1-A_2) e^{\Phi_1^{\sharp}+ \overline{\Phi_2^{\sharp^*}}}$, $C_0= C_{15} \tau^{s/(s+2)}$ and $\sigma=1$, there exist $C_{15}>0$ and $\epsilon_0>0$ such that the following inequality
\[
 \left |  \mathcal{F}\left[ \chi_\Omega(A_1-A_2) e^{\Phi_1^{\sharp}+ \overline{\Phi_2^{\sharp^*}}}  \right](\eta) \right | \leq C_{15}\, \tau^{s/(s+2)} \left( e^{-\pi \epsilon^2 \left | \eta \right |^2} +\epsilon^s  \right),
\]
holds true for all $0<\epsilon<\epsilon_0$ and for all $\eta \in \mathbb{R}^n$. Hence, from (\ref{cinco}) and the above inequality, we get
%\[
%\left | \left( \xi^\prime, 2\tau  \sqrt{1-\tau^{-2} \frac{\left | \xi \right |^2}{4}} \frac{ \left |  \xi^\prime \right |}{ \left |  \xi \right |} \right) \right |= 4\tau^2 \frac{ \left | \xi^\prime \right |^2}{\left | \xi \right |^2}\leq 4\tau^2,
%\]
%we get
\begin{equation}\label{seven7}
\left | VI \right |\leq C_{16}\,  \tau^{s/(s+2)} \left( e^{-4 \pi \epsilon^2 \tau^2 \frac{\left |\xi^\prime  \right |^2}{\left |  \xi \right |^2} }  + \epsilon^s \right).
\end{equation}
Analogously as was estimated $VI$, we obtain
\begin{equation}\label{seven77}
\left | VII \right |\leq C_{17}\,  \tau^{s/(s+2)} \left( e^{-4 \pi \epsilon^2 \tau^2 \frac{\left |\xi^\prime  \right |^2}{\left |  \xi \right |^2} }  + \epsilon^s \right).
\end{equation}
Hence, by combining (\ref{two2})-(\ref{six6}) and (\ref{seven7})-(\ref{seven77}) into (\ref{lilax3}) and taking into account that $\overline{\rho_{2,0}}-\rho_{1,0}=-2(i\mu_1+\mu_2)$; we get
\begin{equation}\label{removing}
\begin{aligned}
& \left |\int_{\mathbb{R}^n}  (i\mu_1+\mu_2)\cdot (\chi_{\Omega\cup\Omega^*}(\widetilde{A_1}-\widetilde{A_2}))e^{i\xi\cdot x} e^{\Phi_1+ \overline{\Phi_2}}  \right | \\
& = \frac{1}{2} \left |\int_{\mathbb{R}^n}  (\overline{\rho_{2,0}}-\rho_{1,0})\cdot (\chi_{\Omega\cup\Omega^*}(\widetilde{A_1}-\widetilde{A_2}))e^{i\xi\cdot x} e^{\Phi_1+ \overline{\Phi_2}}  \right | \\
& \leq C_{18}  \left[  \tau^{-s/(s+2)}  + e^{2\tau \kappa}  dist \, (C_1^\Gamma, C_2^\Gamma) +\tau^{s/(s+2)}  \left( e^{-4 \pi \epsilon^2 \tau^2 \frac{\left |\xi^\prime  \right |^2}{\left |  \xi \right |^2} }  + \epsilon^s \right) \right].
\end{aligned}
\end{equation}
Now the next task will be to remove the function $e^{\Phi_1+ \overline{\Phi_2}}$ from the left-hand side of the above inequality. Since (\ref{mmmmmz12}) and (\ref{mmmmmz123}), it follow easily that
%\begin{equation}\label{last333}
\[
(\overline{\rho_{2,0}}-\rho_{1,0})\cdot \nabla (\Phi_1+ \overline{\Phi_2})- i (\overline{\rho_{2,0}}-\rho_{1,0})\cdot (\chi_{\Omega\cup\Omega^*}(\widetilde{A_1}-\widetilde{A_2}))=0,
\]
which imply that
\begin{equation}\label{xvtyu}
(i\mu_1+\mu_2)\cdot \nabla (\Phi_1+ \overline{\Phi_2})- i (i\mu_1+\mu_2)\cdot (\chi_{\Omega\cup\Omega^*}(\widetilde{A_1}-\widetilde{A_2}))=0.
\end{equation}
%\end{equation}
From this equation, (\ref{–n–z}), (\ref{prtresz15}) and applying Lemma \ref{KyUh} to $\Phi=\Phi_1+ \overline{\Phi_2}$ and $W=-i\chi_{\Omega\cup\Omega^*}(\widetilde{A_1}-\widetilde{A_2})$, we can remove the term $e^{\Phi_1+ \overline{\Phi_2}}$ from the left-hand side of (\ref{removing}). Thus we obtain
\begin{equation}\label{–po}
\begin{aligned}
&  \left |\int_{\mathbb{R}^n}  (i\mu_1+\mu_2)\cdot (\chi_{\Omega\cup\Omega^*}(\widetilde{A_1}-\widetilde{A_2}))e^{i\xi\cdot x}  \right | \\
& \leq  C_{19}  \left[  \tau^{-s/(s+2)}  + e^{2\tau k}  dist \, (C_1^\Gamma, C_2^\Gamma) +\tau^{s/(s+2)}  \left( e^{-4 \pi \epsilon^2 \tau^2 \frac{\left |\xi^\prime  \right |^2}{\left |  \xi \right |^2} }  + \epsilon^s \right) \right].
\end{aligned}
\end{equation}
We can apply the previous arguments with $\rho_1$ replaced by $\overline{\rho_1}$ and $\rho_2$ replaced by $\overline{\rho_2}$, to obtain
\begin{equation}\label{–po1}
\begin{aligned}
&  \left |\int_{\mathbb{R}^n}  (-i\mu_1+\mu_2)\cdot (\chi_{\Omega\cup\Omega^*}(\widetilde{A_1}-\widetilde{A_2}))e^{i\xi\cdot x}  \right | \\
& \leq  C_{20}  \left[  \tau^{-s/(s+2)}  + e^{2\tau \kappa}  dist \, (C_1^\Gamma, C_2^\Gamma) +\tau^{s/(s+2)}  \left( e^{-4 \pi \epsilon^2 \tau^2 \frac{\left |\xi^\prime  \right |^2}{\left |  \xi \right |^2} }  + \epsilon^s \right) \right].
\end{aligned}
\end{equation}
Thus, adding and subtracting (\ref{–po}) and (\ref{–po1}), we get
\begin{equation}\label{–po2}
\begin{aligned}
&  \left |\int_{\mathbb{R}^n}  \mu\cdot (\chi_{\Omega\cup\Omega^*}(\widetilde{A_1}-\widetilde{A_2}))e^{i\xi\cdot x}  \right | \\
& \leq  C_{21} \left |\mu\right | \left[  \tau^{-s/(s+2)}  + e^{2\tau \kappa}  dist \, (C_1^\Gamma, C_2^\Gamma) +\tau^{s/(s+2)}  \left( e^{-4 \pi \epsilon^2 \tau^2 \frac{\left |\xi^\prime  \right |^2}{\left |  \xi \right |^2} }  + \epsilon^s \right) \right].
\end{aligned}
\end{equation}
for all $\mu\in \spn\left \{ \mu_1, \mu_2 \right \}$ and for all $\xi\in \bigcap_{l=1}^{n-1}E_{l}$. Now Lemma \ref{muum} ensures that for every $j,k=1,2, \ldots, n$ the vector defined by $(\mu)_{j,k}:= \xi_j e_k-\xi_k e_j$ belongs to $\spn\left \{ \mu_1, \mu_2 \right \}$. Thus, replacing these vectors into (\ref{–po2}), it immediately implies the desired assertion and so the proof is completed.
\end{proof}
\subsection{Proof of Theorem \ref{SMP}}
By Proposition \ref{ggggg5}, taking into account that the constant $C>0$ in the estimate (\ref{final2}) is independent of $\xi\in \bigcap_{l=1}^{n-1}E_{l}$ and since the set $\bigcap_{l=1}^{n-1}E_{l}$ is dense in $\mathbb{R}^n$, it follows that the following estimate
\begin{equation}\label{final2zq} 
\begin{aligned}
& \left | \mathcal{F}\left [ d(\chi_{\Omega\cup\Omega^*}\widetilde{A_1})-  \right ](\xi) - \mathcal{F}\left [ d(\chi_{\Omega\cup\Omega^*}\widetilde{A_2})  \right ](\xi)   \right | \\
&\quad  \leq C  \left | \xi \right |\left[  \tau^{-s/(s+2)}  + e^{2\tau \kappa}  dist \, (C_1^\Gamma, C_2^\Gamma) +\tau^{s/(s+2)}  \left( e^{-4 \pi \epsilon^2 \tau^2 \frac{\left |\xi^\prime  \right |^2}{\left |  \xi \right |^2} }  + \epsilon^s \right) \right],
\end{aligned}
\end{equation}
holds true for all $\xi\in \mathbb{R}^n$. Now consider $R\geq 1$ (which will be fixed later) and denote by $B_R(0)$ the open ball in $\mathbb{R}^n$ centered at $0$ of radius $R$. For convenience we denote $\widetilde{A}:=\chi_{\Omega\cup\Omega^*}(\widetilde{A_1}-\widetilde{A_2})$. Then
%We recall that the $H^{-1}(\mathbb{R})^n$-norm can be defined as
%\[
%\left \| f \right \|_{H^{-1}(\mathbb{R}^n)}^2=\int_{\mathbb{R}^n}(1+\left | \xi \right |^2   )^{-1}\left | \widehat{f}(\xi) \right |^2 d\xi.
%\
\begin{equation}\label{w3}
\left \| d\widetilde{A} \right \|_{H^{-1}(\mathbb{R}^n)}^2= E_1 + E_2,
\end{equation}
where 
\[
E_1=  \int_{B_R(0)\setminus \left \{0  \right \}}(1+\left | \xi \right |^2   )^{-1}\left | \mathcal{F}\left[ d\widetilde{A} \right]  (\xi) \right |^2 d\xi
\]
and
\[
E_2= \int_{\mathbb{R}^n\setminus B_R(0)}(1+\left | \xi \right |^2   )^{-1}\left | \mathcal{F}\left[ d\widetilde{A} \right]  (\xi) \right |^2 d\xi.
\]
%\begin{equation}\label{w3}
%\begin{aligned}
%\left \| d\widetilde{A} \right \|_{H^{-1}(\mathbb{R}^n)}^2& =  \int_{B_R(0)\setminus \left \{0  \right \}}(1+\left | \xi \right |^2   )^{-1}\left | \mathcal{F}\left[ d\widetilde{A} \right]  (\xi) \right |^2 d\xi\\
%&\quad   +  \int_{\mathbb{R}^n\setminus B_R(0)}(1+\left | \xi \right |^2   )^{-1}\left | \mathcal{F}\left[ d\widetilde{A} \right]  (\xi) \right |^2 d\xi.\\
%\end{aligned}
%\end{equation}
From (\ref{final2zq}) and taking $\epsilon=\tau^{-3/2(s+2)}$, we get 
\begin{equation}\label{w1}
\begin{aligned}
 E_1&=   \int_{B_R(0)\setminus \left \{0  \right \}}(1+\left | \xi \right |^2   )^{-1}\left | \mathcal{F}\left[ d\widetilde{A} \right]  (\xi) \right |^2 d\xi\\
 & \quad \leq C_1\, R^n\left(  \tau^{-2s/(s+2)}  + e^{4\tau \kappa  }  dist \, (C_1^\Gamma, C_2^\Gamma)^{2} +\tau^{2s/(s+2)} \epsilon^{2s} \right)\\
 & \qquad  + C_1\,\tau^{2s/(s+2)} \int_{B_R(0)\setminus \left \{0  \right \}} e^{-8 \pi \epsilon^2 \tau^2 \frac{\left |\xi^\prime  \right |^2}{\left |  \xi \right |^2} }\\
 &  \quad \leq C_1\, R^n\left(  \tau^{-2s/(s+2)}  + e^{4\tau \kappa  }  dist \, (C_1^\Gamma, C_2^\Gamma)^{2} +\tau^{2s/(s+2)} \epsilon^{2s} \right)\\
 & \qquad  + C_1\tau^{2s/(s+2)}  R^n\epsilon^{-2} \tau^{-2}\\
&\quad  \leq C_2\, R^n\left(\tau^{-s/(s+2)} +  e^{4\tau \kappa  }  dist \, (C_1^\Gamma, C_2^\Gamma)^{2}    \right).
 \end{aligned}
\end{equation}
%Thus, by taking $\epsilon=\tau^{-3/2(s+2)}$, we obtain
%\begin{equation}\label{w1}
%\begin{aligned}
%&  \int_{B_R(0)\setminus \left \{0  \right \}}(1+\left | \xi \right |^2   )^{-1}\left | \mathcal{F}\left[ d\widetilde{A} \right]  (\xi) \right |^2 d\xi\\
%&\quad  \leq C_2\, R^n\left(\tau^{-s/(s+2)} +  e^{4\tau \kappa  }  dist \, (C_1^\Gamma, C_2^\Gamma)^{2}    \right).
%\end{aligned}
%\end{equation}

To estimate the integral over $\mathbb{R}^n\setminus B_R(0)$ in (\ref{w3}) we shall use a mollification argument to $\widetilde{A}$, as was done in the proof of Theorem \ref{Exso}. By Lemma \ref{heyw}, we have $\widetilde{A}\in L^{\infty}\cap B^{2, \infty}_s(\mathbb{R}^n)$. Thus, consider $\varphi \in C^{\infty}_0(\mathbb{R}^n)$ such that $0\leq \varphi \leq 1$ and $\supp \varphi \subset \overline{B_1(0)}$, where $\overline{B_1(0)}$ denotes the closure of the  ball in $\mathbb{R}^n$ of radius 1 centered at the origin. For each $\delta>0$ we define $\varphi_\delta(x)=\delta^{-n}\varphi(x/\delta)$ and set ${\widetilde{A}}^{\sharp}_{\delta}= \widetilde{A}*\varphi_\delta$ which belongs to $C^{\infty}_0(\mathbb{R}^n; \mathbb{C}^n)$. Then there exists a positive constant $C_3>0$ (depending on $\Omega$ and $n$) such that 
\begin{equation}\label{Asharpz1}
\left \| \widetilde{A}-{\widetilde{A}}^{\sharp}_{\delta}  \right \|_{L^2(\mathbb{R}^n)} \leq C_3\, \delta^{s} \left \| \widetilde{A} \right \|_{B^{2,\infty}_s(\mathbb{R}^n)}, 
\end{equation}
and for each $\alpha \in \mathbb{N}^n$, we have
\begin{equation}\label{Asharp1z2}
\left \| \partial^{\alpha} {\widetilde{A}}^{\sharp}_{\delta} \right \|_{L^\infty(\mathbb{R}^n)} \leq C_3\, \delta^{-\left | \alpha \right |} \left \| \widetilde{A} \right \|_{L^{\infty}(\mathbb{R}^n)}.
\end{equation}
Now Plancherel's identity and (\ref{Asharpz1})-(\ref{Asharp1z2}), imply that
\begin{align*}
 E_2&= \int_{\mathbb{R}^n\setminus B_R(0)}(1+\left | \xi \right |^2   )^{-1}\left | \mathcal{F}\left[ d\widetilde{A} \right]   (\xi) \right |^2 d\xi\\
 %& \leq C_2\int_{\mathbb{R}^n\setminus B_R(0)}(1+\left | \xi \right |^2   )^{-1}\left |  \widehat{d\widetilde{A}} (\xi) \right |^2 d\xi\\
 &\leq  C_4\int_{\mathbb{R}^n\setminus B_R(0)}(1+\left | \xi \right |^2   )^{-1}\left |  \mathcal{F}\left[ d\widetilde{A}^\sharp_\delta \right]  (\xi) \right |^2 d\xi \\
 & \quad + C_4  \int_{\mathbb{R}^n\setminus B_R( 0)}(1+\left | \xi \right |^2   )^{-1}\left |   \mathcal{F}\left[ d\widetilde{A} \right]  (\xi)  -\mathcal{F}\left[ d\widetilde{A}^\sharp_\delta \right]  (\xi)  \right |^2 d\xi\\
 & \leq C_5\int_{\mathbb{R}^n\setminus B_R(0)}(1+\left | \xi \right |^2   )^{-1}\left |  \mathcal{F}\left[ d\widetilde{A}^\sharp_\delta \right]  (\xi) \right |^2 d\xi \\
 & \qquad + C_5  \int_{\mathbb{R}^n\setminus B_R(0)} \left | \xi \right |^2 (1+\left | \xi \right |^2   )^{-1}\left | \mathcal{F} \left[ A  -A^\sharp\right] (\xi)\right |^2 d\xi\\
 & \leq C_6 \left( R^{-2}  \left \| d{A}^\sharp  \right \|^2_{L^2(\mathbb{R}^n)} + \left \| \widetilde{A}- \widetilde{A}^\sharp_\delta \right \|_{L^2(\mathbb{R}^n)}^2 \right)\\
 & \leq C_7 \left(   R^{-2} \delta^{-2}+ \delta^{2s}  \right).
\end{align*}
By equating the terms $R^{-2} \delta^{-2}$ and $ \delta^{2s}$, that is taking $\delta=R^{-1/(s+1)}$, we get
%\section*{Acknowledgements} The author would like to thank Alberto Ruiz for the very nice  discussions about Mathematics. For his support and encouragement during the preparation of this article. This article is part of my PhD. thesis and it is supporting by the Project MTM$2011-28198$ of Ministerio de Econom\'ia y Competividad de Espa\~na.
\begin{equation}\label{w2}
 \int_{\mathbb{R}^n\setminus B_R(0)}(1+\left | \xi \right |^2   )^{-1}\left | \mathcal{F}\left[ d\widetilde{A} \right](\xi)  \right |^2 d\xi \leq C_8 R^{-2s/(s+1)}. 
\end{equation}
Then combining (\ref{w1}) and (\ref{w2}) into (\ref{w3}) we have that there exist two positive constants $C_9$ and $\tau_1$ such that the estimate
\begin{equation}\label{wwww3}
\begin{aligned}
& \left \| dA \right \|_{H^{-1}(\mathbb{R}^n)}^2\\
& \quad \leq C_9\left( R^ne^{4\tau \kappa}  dist \, (C_1^\Gamma, C_2^\Gamma)^{2}+ R^n\tau^{-s/(s+2)}   +  R^{-2s/(s+1)}  \right),
\end{aligned}
\end{equation}
holds true for all $\tau\geq \tau_1$. By equating the two last terms of the right-hand side of the above inequality, and then we take
\[
R=\tau^{\frac{s(s+1)}{(s+2)(n+ns+2s)}}.
\]
We also have that there exist another two positive constants $C_{10}$ and $\tau_2$ such that
\[
R^n=\tau^{\frac{ns(s+1)}{(s+2)(n+ns+2s)}}\leq C_{10} e^{\tau k},
\]
for all $\tau\geq \tau_2$. Thus, replacing this facts into (\ref{wwww3}), we have
\begin{equation}\label{ultimooo}
\begin{aligned}
&  \left \| d\widetilde{A} \right \|_{H^{-1}(\mathbb{R}^n)}^2\\
&\quad  \leq C_{11}\left( e^{5\tau \kappa}   dist \, (C_1^\Gamma, C_2^\Gamma)^{2} + \tau^{-\frac{2s^2}{(s+2)(n+ns+2s)}} \right).
\end{aligned}
\end{equation}
On the other hand, there exists $\tau_3>0$ such that 
\[
e^{-\tau \kappa} \leq  \tau^{-\frac{2s^2}{(s+2)(n+ns+2s)}},
\]
for all $\tau\geq \tau_3$. Now taking $\tau_0\geq \max \left( \tau_1, \tau_2, \tau_3 \right)$ such that $3\kappa \tau_0\geq 1$, it is easy to check that
\begin{equation}\label{casiultimo}
\tau:= \frac{1}{3k} \left | \log dist(C_1^\Gamma, C_2^\Gamma) \right |\geq \tau_0,
\end{equation}
whenever 
\[
dist(C_1^\Gamma, C_2^\Gamma)\leq e^{-3\kappa\tau_0}.
\]
Thus, from (\ref{casiultimo}) it follows that $dist(C_1^\Gamma, C_2^\Gamma)\leq e^{-3\tau k}$ and then from (\ref{ultimooo}) we get
\begin{equation}\label{dmpoik}
  \left \| d\widetilde{A} \right \|_{H^{-1}(\mathbb{R}^n)} \leq C_{10} \left | \log dist(C_1^\Gamma, C_2^\Gamma) \right |^{-\frac{s^2}{(s+2)(n+ns+2s)}}.
\end{equation}
Since $s\in \left( 0, 1/2\right)$, it is immediate to see that
\[
\frac{s^2}{(s+2)(n+ns+2s)} \geq \frac{s^2}{5n}.
\]
Thus, we conclude the proof by considering the above inequality into (\ref{dmpoik}), taking $C=\max \left\{ 3\kappa\tau_0, C_{10}\right\}$, $\lambda=1/5$ and finally taking into account that
\[
\left \| d(A_1-A_2) \right \|_{H^{-1}(\Omega)}\leq \left \| d\widetilde{A} \right \|_{H^{-1}(\mathbb{R}^n)}.  
\]
\section{Stability for the electric potential}

The goal of this section is to prove Theorem \ref{SEP}. The idea will be to combine the gauge invariance of the Cauchy data sets (see Lemma $3.1$ in \cite{KU}) and the stability result already proved for the magnetic potentials in the section $2$. Our proof involves a Hodge decomposition as in Caro and Pohjola \cite{CP}, see Lemma $6.2$ therein. Our starting point is the following lemma. This is the analogous to Lemma $5.1$ in \cite{CP}. For this reason we will only give the main ideas of the proof.
\begin{lem}\label{coAlidww}
Let $\Omega\subset \mathbb{R}^n$ be an open set. Let $B\subset \mathbb{R}^n$ be an open ball with $\overline{\Omega}\subset B$ and $\Gamma_0 \subset B\cap \left \{x\in \mathbb{R}^n: x_n=0  \right \}$. Let $A_1, A_2 \in L^\infty(\Omega; \mathbb{C}^n)$ and $q_1, q_2\in L^\infty(\Omega; \mathbb{C})$. Given $M>0$ and $s\in \left( 0,1/2 \right)$, assume that $\chi_\Omega A_1$ and $\chi_\Omega A_2$ belong to $\mathcal{A}(\Omega, M,s)$. Let $\varphi \in W^{1,n}(B)\cap L^{\infty}(B)$ with $\varphi|_{\partial B}=0$. Then for any $U_1, U_2\in H^{1}(B)$ satisfying in $B$ the equations: 
\begin{equation}\label{aver1}
\mathcal{L}_{\chi_{\Omega\cup\Omega^*} \widetilde{A_1}, \chi_{\Omega\cup\Omega^*} \widetilde{q_1}}\, U_1=0,\quad U_1|_{B\cap \left \{x\in \mathbb{R}^n: x_n=0\right\}}=0,
\end{equation}
\begin{equation}\label{aver2}
\mathcal{L}_{\overline{\chi_{\Omega\cup\Omega^*} \widetilde{A_2}}, \overline{\chi_{\Omega\cup\Omega^*} \widetilde{q_2}}}\,U_2=0,\quad U_2|_{B\cap \left \{x\in \mathbb{R}^n: x_n=0\right\}}=0,
\end{equation}
there exists a positive constant $C$ (depending on $\Omega, n,M, s, \left \| q_1 \right \|_{L^\infty}, \left \| q_2 \right \|_{L^\infty}$) such that
\begin{equation}\label{uauww}
\begin{aligned}
& \left |  \int_{B} e^{i\varphi} \left\{ ( \chi_{\Omega\cup\Omega^*} \widetilde{A_1}-(\chi_{\Omega\cup\Omega^*} \widetilde{A_2}+\nabla \varphi))\cdot \left( DU_1\overline{U_2} +  U_1 \overline{DU}_2 \right) \right. \right.\\
&\qquad + \left[    \chi_{\Omega\cup\Omega^*} \widetilde{A_1}^2-(\chi_{\Omega\cup\Omega^*} \widetilde{A_2}+\nabla \varphi)^2 +  \chi_{\Omega\cup\Omega^*} (\widetilde{q_1}-\widetilde{q_2})   \right. \\
& \qquad \qquad   \left. \left. \left. -(\chi_{\Omega\cup\Omega^*}( \widetilde{A_1}-\widetilde{A_2}) -\nabla \varphi)\cdot \nabla \varphi \right] U_1\overline{U}_2 \right\}  \right |\\
&  \qquad  \leq C\, dist \, (C_1^\Gamma, C_2^\Gamma) \left \| U_1 \right \|_{H^1(\Omega\cup\Omega^*)} \left \| U_2 \right \|_{H^1(\Omega\cup\Omega^*)},
\end{aligned}
\end{equation}
where $C_j^\Gamma$ denotes the local Cauchy data set $C_{A_j, q_j}^\Gamma$ for $j=1,2$.
\end{lem}
\begin{proof}
Our proof starts by restricting (\ref{aver1})-(\ref{aver2}) to $\Omega$ and $\Omega^*$ to obtain
\begin{equation}
\begin{aligned}
\mathcal{L}_{\widetilde{A_1}, \widetilde{q_1}}\, (U_1|_{\Omega})=0\quad \mbox{in}\, \Omega,\quad (U_1|_\Omega)|_{\Gamma_0}= 0,\\
\mathcal{L}_{\overline{\widetilde{A_2}}, \overline{\widetilde{q_2}}}\, (U_2|_{\Omega})=0\quad \mbox{in}\, \Omega,\quad (U_2|_\Omega)|_{\Gamma_0}= 0.
\end{aligned}
\end{equation}
and
\begin{equation}
\begin{aligned}
\mathcal{L}_{\widetilde{A_1}, \widetilde{q_1}}\, (U_1|_{\Omega^*})=0\quad \mbox{in}\, \Omega^*,\quad (U_1|_{\Omega^*})|_{\Gamma_0}= 0,\\
\mathcal{L}_{\overline{\widetilde{A_2}}, \overline{\widetilde{q_2}}}\, (U_2|_{\Omega^*})=0\quad \mbox{in}\, \Omega^*,\quad (U_2|_{\Omega^*})|_{\Gamma_0}= 0.
\end{aligned}
\end{equation}
Hence, by Corollary \ref{coAlid} applied to $\Omega$ and for the magnetic potentials $\widetilde{A_1},\widetilde{A_2}$, the electric potentials $\widetilde{q_1}, \widetilde{q_2}$ and $U_1|_\Omega, U_2|_\Omega\in H^1(\Omega)$, we get
\begin{equation}\label{uauzs1}
\begin{aligned}
& \left |  \int_{\Omega}  (\widetilde{A_1}-\widetilde{A_2})\cdot \left( DU_1\overline{U_2} +  U_1 \overline{DU_2} \right) + \left( \widetilde{A_1}^2-\widetilde{A_2}^2+\widetilde{q_1}-\widetilde{q_2} \right) U_1\overline{U_2}  \right |\\
&  \qquad  \leq C_1\,  dist \, (C_1^\Gamma, C_2^\Gamma) \left \| U_1 \right \|_{H^1(\Omega)} \left \| U_2 \right \|_{H^1(\Omega)}\\
& \qquad  \leq C_1\,  dist \, (C_1^\Gamma, C_2^\Gamma) \left \| U_1 \right \|_{H^1(\Omega\cup\Omega^*)} \left \| U_2 \right \|_{H^1(\Omega\cup\Omega^*)}.
\end{aligned}
\end{equation}
Applying again Corollary \ref{coAlid} to $\Omega^*$ for the magnetic potentials $\widetilde{A_1}, \widetilde{A_2}$, the electric potentials $\widetilde{q_1}, \widetilde{q_2}$ and $U_1|_{\Omega^*}, U_2|_{\Omega^*}\in H^1(\Omega^*)$, we get
\begin{equation}\label{uauzs2}
\begin{aligned}
& \left |  \int_{\Omega^*}  (\widetilde{A_1}-\widetilde{A_2})\cdot \left( DU_1\overline{U_2} +  U_1 \overline{DU_2} \right) + \left( \widetilde{A_1}^2-\widetilde{A_2}^2+\widetilde{q_1}-\widetilde{q_2} \right) U_1\overline{U_2}  \right |\\
&  \qquad  \leq C_2\,  dist^* \, (C_1^\Gamma, C_2^\Gamma) \left \| U_1 \right \|_{H^1(\Omega^*)} \left \| U_2 \right \|_{H^1(\Omega^*)}\\
&  \qquad  \leq C_2\,  dist^* \, (C_1^\Gamma, C_2^\Gamma) \left \| U_1 \right \|_{H^1(\Omega\cup\Omega^*)} \left \| U_2 \right \|_{H^1(\Omega\cup\Omega^*)},
\end{aligned}
\end{equation}
where $dist^* \, (C_1^\Gamma, C_2^\Gamma) $ is defined by $dist \, (C_1^\Gamma, C_2^\Gamma)$ (see (\ref{distanciaa})) with $\Omega$ replaced by $\Omega^*$. Then, adding (\ref{uauzs1}) and (\ref{uauzs2}), we obtain
\begin{equation}\label{omg098}
\begin{aligned}
& \left |  \int_{B}  \left[\chi_{\Omega\cup\Omega^*}(\widetilde{A_1}-\widetilde{A_2})\right]\cdot \left( DU_1\overline{U_2} +  U_1 \overline{DU_2} \right)\right.\\
&\qquad \qquad  \left. +\chi_{\Omega\cup\Omega^*} \left[ \widetilde{A_1}^2-\widetilde{A_2}^2+\widetilde{q_1}-\widetilde{q_2} \right] U_1\overline{U_2}  \right |\\
&  \qquad  \leq C_3\,  (dist \, (C_1^\Gamma, C_2^\Gamma) + dist^* \, (C_1^\Gamma, C_2^\Gamma)) \left \| U_1 \right \|_{H^1(\Omega\cup\Omega^*)} \left \| U_2 \right \|_{H^1(\Omega\cup\Omega^*)}\\
&  \qquad  = 2C_3\, dist \, (C_1^\Gamma, C_2^\Gamma)  \left \| U_1 \right \|_{H^1(\Omega\cup\Omega^*)} \left \| U_2 \right \|_{H^1(\Omega\cup\Omega^*)},
\end{aligned}
\end{equation}
where we have used that $ dist^* \, (C_1^\Gamma, C_2^\Gamma)= dist \, (C_1^\Gamma, C_2^\Gamma)$.
On the other hand, by the gauge invariance of the Cauchy data sets (see Lemma $3.1$ in \cite{KU}), we get
\begin{equation}\label{––pl}
\begin{aligned}
& \left |  \int_{B}  \left[\chi_{\Omega\cup\Omega^*}(\widetilde{A_1}-\widetilde{A_2})\right]\cdot \left( DU_1\overline{U_2} +  U_1 \overline{DU_2} \right)\right.\\
&\qquad \qquad  \left. +\chi_{\Omega\cup\Omega^*} \left[ \widetilde{A_1}^2-\widetilde{A_2}^2+\widetilde{q_1}-\widetilde{q_2} \right] U_1\overline{U_2}  \right |\\
%& =  \left |  \int_{B} \chi_\Omega (A_1-A_2)\cdot \left( DU_1\overline{U_2} +  U_1 \overline{DU}_2 \right) + \chi_\Omega \left( A_1^2-A_2^2+q_1-q_2 \right) U_1\overline{U}_2  \right |\\
& = \left \langle N_{\chi_{\Omega\cup\Omega^*} \widetilde{A_1}, \chi_{\Omega\cup\Omega^*} \widetilde{q_1}}U_1, T_r U_2   \right \rangle_B - \overline{\left \langle N_{\overline{\chi_{\Omega\cup\Omega^*} \widetilde{A_2}}, \overline{\chi_{\Omega\cup\Omega^*} \widetilde{q_2}}} U_2 , T_r U_1 \right \rangle}_B\\
&=  \left \langle N_{\chi_{\Omega\cup\Omega} \widetilde{A_1}, \chi_{\Omega\cup\Omega^*} \widetilde{q_1}}U_1, T_r U_2   \right \rangle_B \\
& \qquad- \overline{\left \langle N_{\overline{\chi_{\Omega\cup\Omega^*} \widetilde{A_2}+\nabla \varphi}, \overline{\chi_{\Omega\cup\Omega^*} \widetilde{q_2}}} (e^{-i\overline{\varphi}}U_2) , T_r (e^{-i\varphi }U_1) \right \rangle}_B\\
&=  \int_{B} e^{i\varphi} \left\{ ( \chi_{\Omega\cup\Omega^*} \widetilde{A_1}-(\chi_{\Omega\cup\Omega^*} \widetilde{A_2}+\nabla \varphi))\cdot \left( DU_1\overline{U_2} +  U_1 \overline{DU}_2 \right) \right. \\
&\qquad + \left[    \chi_{\Omega\cup\Omega^*} \widetilde{A_1}^2-(\chi_{\Omega\cup\Omega^*} \widetilde{A_2}+\nabla \varphi)^2 +  \chi_{\Omega\cup\Omega^*} (\widetilde{q_1}-\widetilde{q_2})   \right. \\
& \qquad \qquad    \left. \left. -(\chi_{\Omega\cup\Omega^*}( \widetilde{A_1}-\widetilde{A_2}) -\nabla \varphi)\cdot \nabla \varphi \right] U_1\overline{U}_2 \right\},
\end{aligned}
\end{equation}
where we have used that if $U_1\in H^1(B)$ satisfies $\mathcal{L}_{\chi_{\Omega\cup\Omega^*} \widetilde{A_1}, \chi_{\Omega\cup\Omega^*} \widetilde{q_1}}U_1=0$ in $B$ then $e^{-i\varphi}U_1\in H^1(B)$ satisfies $\mathcal{L}_{\chi_{\Omega\cup\Omega^*} \widetilde{A_1}+\nabla \varphi,\, \chi_{\Omega\cup\Omega^*} \widetilde{q_1}}(e^{-i\varphi}U_1)=0$ in $B$. Analogously, if $U_2\in H^1(B)$ satisfies $\mathcal{L}_{\overline{\chi_\Omega A_2}, \overline{\chi_\Omega q_2}}U_2=0$ in $B$ then $e^{-i\overline{\varphi}}U_2\in H^1(B)$ satisfies $\mathcal{L}_{\overline{\chi_{\Omega\cup\Omega} \widetilde{A_2}+\nabla \varphi},\,  \overline{\chi_{\Omega\cup\Omega} \widetilde{q_2}}}(e^{-i\overline{\varphi}}U_2)=0$ in $B$. Thus, from Lemma \ref{Alid} with $\Omega$  replaced by $B$ and replacing (\ref{––pl}) into (\ref{omg098}), we conclude the proof.
\end{proof}

In the section \ref{sde} we have used Corollary \ref{coAlid} to isolate $A_1-A_2$,  and then using solutions for the Schr\"odinger operator given by Proposition \ref{construcsolutions} we obtained the estimate (\ref{final2}) of Proposition \ref{ggggg5}. Now we follow the same ideas. Our task will be to isolate $\chi_{\Omega\cup\Omega^*}(\widetilde{q_1}-\widetilde{q_2})$ from the left-hand side of (\ref{uauww}) and taking into account two facts. The first fact will be to obtain an estimate for the function  $\chi_{\Omega\cup\Omega^*}(\widetilde{A_1}-\widetilde{A_2})-\nabla \varphi$. This can be done by using the following Hodge decomposition derived in \cite{CP}.
%\begin{figure}[h]
%\caption{Example}
%\includegraphics[width=0.49\textwidth]{leytera}
%\hspace{0.1\linewidth}
%\includegraphics[width=0.49\textwidth]{leyterb}
%\includegraphics[width=0.4\textwidth]{leyterc}
%\hspace{0.1\linewidth}
%\includegraphics[width=0.4\textwidth]{leyterd}
%\end{figure}
\begin{lem}\label{hdecom}
Let $B\subset \mathbb{R}^n$ be an open ball satisfying $\overline{\Omega}\subset B$. Let $A_1$ and $A_2$ belong to $L^\infty(\Omega; \mathbb{C}^n)$ then there exist $\psi \in W^{1,p}(B)$ with $p\geq 2$ and $C>0$ satisfying the following conditions:
\begin{equation}\label{mreyw}
\left \|  \psi \right \|_{W^{1,p}(B)} \leq C \left \| \chi_{\Omega\cup\Omega^*} (\widetilde{A_1}-\widetilde{A_2} ) \right \|_{L^p(\mathbb{R}^n)}
\end{equation}
and
\begin{equation}\label{mreyww}
\left \| \chi_{\Omega\cup\Omega^*}(\widetilde{A_1}-\widetilde{A_2})  - \nabla \psi \right \|_{L^2(B)} \leq C \left \| d (\chi_{\Omega\cup\Omega^*}(\widetilde{A_1}-\widetilde{A_2})) \right \|_{H^{-1}(B)}.
\end{equation}
Moreover, if $B^\prime$ is another open ball with $\overline{\Omega} \subset B^\prime \subset \subset B$ then 
\begin{equation}\label{rywww}
\left \|  \psi - \underline{\psi} \right \|_{H^1(B\setminus \overline{B^\prime})} \leq C\left \| d (\chi_{\Omega\cup\Omega^*}(\widetilde{A_1}-\widetilde{A_2})) \right \|_{H^{-1}(B)},
\end{equation}
where $\underline{\psi}$ denotes the average of $\psi$ in $B\setminus \overline{B^\prime}$.
\end{lem}
%\begin{figure}[h]
%   \centering
%   %%----primera subfigura----
%   \subfloat {
%        %\label{fig:museo:a}         %% Etiqueta para la primera subfigura
%        \includegraphics[width=0.4\textwidth]{leyterc.png}}
%   \hspace{0.1\linewidth}
%   %%----segunda subfigura----
%   \subfloat{
%        %\label{fig:museo:b}         %% Etiqueta para la segunda subfigura
%        \includegraphics[width=0.4\textwidth]{leyterd.png}}
%   %%----tercera subfigura----
%  \caption{Shape of the sets $\Omega$ and $\Omega^*$. Also the balls $B^\prime$ and $B$.}
%   \label{figure2}                %% Etiqueta para la figura entera
%\end{figure}
The second fact will be to use the solutions $U_j\in H^1(B)$ (j=1,2) (with the requirements of Lemma \ref{coAlidww}) constructed implicitly in the proof of Proposition \ref{construcsolutions}.
%\begin{prop} \label{construcsolutionsww}
%Let $\Omega\subset \mathbb{R}^n$ be a bounded open set. Let $B\subset \mathbb{R}^n$ be an open ball with $\overline{\Omega\cup \Omega^*}\subset B$. Let $M>0$ and $s\in \left( 0,1/2 \right)$. Consider $A_1,A_2\in \mathcal{A}(\Omega, M,s)$ and $q_1,q_2\in L^{\infty}(\Omega)$. Then there exist $U_1, U_2\in H^{1}(B)$ satisfying in $B$: $\mathcal{L}_{A_1,q_1}U_1=0$ with $U_1|_{B\cap \left\{ x\in \mathbb{R}^n: x_n=0  \right\}}=0$ and $\mathcal{L}_{\overline{A_2}, \overline{q_2}}U_2=0$ with $U_2|_{B\cap \left\{ x\in \mathbb{R}^n: x_n=0  \right\}}=0$. Moreover these solutions have the form:
%\begin{equation}\label{unosww}
%U_1(x)= e^{\tau \rho_1\cdot x} \left(   e^{\Phi_1^\sharp} +r_1  \right)- e^{\tau \rho_1\cdot x^*} \left(   e^{\Phi_1^{\sharp^*}} +r_1^* \right), \quad x\in B
%\end{equation}
%and
%\begin{equation}\label{dossww}
%U_2(x)= e^{\tau \rho_2\cdot x} \left(   e^{\Phi_2^\sharp} +r_2  \right)- e^{\tau \rho_2\cdot x^*} \left(   e^{\Phi_2^{\sharp^*}} +r_2^* \right),\quad x\in B.
%\end{equation}
%where $\rho_1$ and $\rho_2$ are as in (\ref{rhooooo}). The functions $\Phi_1^\sharp$ satisfy (\ref{xiequation}) and $r_1$ satisfy (\ref{prdos}). Both with $\chi_{\Omega\cup\Omega^*}\widetilde{A_1}$. Analogously, the functions $\Phi_2^\sharp$ satisfy (\ref{xiequation}) and $r_2$ satisfy (\ref{prdos}). Both with $\chi_{\Omega\cup\Omega^*}\widetilde{A_2}$. 
%\end{prop}

\subsection{A Fourier estimate for the electric potentials} This section will be devoted to prove the following Proposition. This is the analogous to Proposition \ref{ggggg5}.
\begin{prop}\label{electyc}
Let $\Omega\subset \mathbb{R}^n$ be a bounded open set. Let $B^\prime \subset \mathbb{R}^n$ and $B\subset \mathbb{R}^n$ be an open balls with $\overline{\Omega\cup \Omega^*}\subset B^\prime \subset \subset B$. Let $A_1, A_2\in L^{\infty}(\Omega, \mathbb{C}^n)$ and $q_1, q_2\in L^{\infty}(\Omega, \mathbb{C})$. Given $M>0$ and $s\in \left(0,1/2\right)$, assume that $\chi_\Omega A_1, \chi_\Omega A_2\in \mathcal{A}(\Omega, M, s)$ and $\chi_\Omega q_1, \chi_\Omega q_2\in \mathcal{Q}(\Omega, M, s)$. Then there exist three positive constants $C$, $\tau_0$ and $\epsilon_0$ (all depending on $\Omega, n,M, s$) such that the following estimate
\begin{equation}\label{final21}
\begin{aligned}
& \left |  \mathcal{F}\left[ \chi_{\Omega\cup\Omega^*}\widetilde{q_1}  \right] (\xi) -  \mathcal{F}\left[ \chi_{\Omega\cup\Omega^*}\widetilde{q_2}  \right] (\xi)     \right | \\
&  \leq C\left( e^{2\tau \kappa}  dist \, (C_1^\Gamma, C_2^\Gamma)  +\left |   \log  dist \, (C_1^\Gamma, C_2^\Gamma)   \right |^{-\frac{\theta s^2}{(s+2)(n+ns+2s)}}  \tau^{(s+4)/(s+2)}   \right.\\
&\qquad \qquad \qquad  \left. +  \tau^{-s/(s+2)} + e^{-4 \pi \epsilon^2 \tau^2 \frac{\left |\xi^\prime  \right |^2}{\left |  \xi \right |^2} }  + \epsilon^s\right),
\end{aligned}
\end{equation}
holds true for every $\xi\in \bigcap_{l=1}^{n-1}E_l$ (see (\ref{Er})), for all $\tau\geq \tau_0$ and for all $0<\epsilon<\epsilon_0$. Here $\theta\in \left( 0, 2/n\right)$.
\end{prop}
\begin{proof}
The proof is similar to the proof of Lemma $5.3$ in \cite{CP}. Consider the function $\psi$ given by Lemma \ref{hdecom} with $p>n$. Notice that this function not necessarily satisfies the vanishing condition on $\partial B$. We will remedy this by using a cutoff argument. So consider a smooth function $\chi\in C_0^\infty(B)$ such that $\chi(x)=1$ in $B^\prime$ and set $\varphi= \chi (\psi-\underline{\psi})$ (note that $\varphi|_{\partial B}=0$). Thus, by (\ref{mreyw}), Morrey's inequality and the boundedness of $B$ we get
\begin{equation}\label{phiiiii111}
\left \| \varphi  \right \|_{L^\infty(B)} + \left \| \nabla \varphi  \right \|_{L^n(B)} + \left \| \nabla \psi  \right \|_{L^n(B)}\leq C_1.
\end{equation}
Now we compute the left-hand side of (\ref{uauww}) by using the functions $U_1, U_2\in H^{1}(B)$ satisfying in $B$: $\mathcal{L}_{\chi_{\Omega	\cup\Omega^*} \widetilde{A_1},\, \chi_{\Omega\cup\Omega^*} \widetilde{q_1}}U_1=0$ with $U_1|_{\partial B^+\cap \left\{ x_n=0\right\}}=0$ and $\mathcal{L}_{\overline{\chi_{\Omega\cup\Omega^*} \widetilde{A_2}}, \overline{\chi_{\Omega\cup\Omega^*} \widetilde{q_2}}}U_2=0$ with $U_2|_{\partial B^+\cap \left\{ x_n=0\right\}}=0$. Such functions were constructed implicitly in the proof of Proposition \ref{construcsolutions}. Also consider $\rho_1$ and $\rho_2$ defined in the section $2$, see (\ref{jkw})-(\ref{jkwww}). Now we will divide the proof into three steps. The first step will be to prove the following claim. 
\subsection*{Claim $1$} If $\theta\in \left( 0, 2/n\right)$ then there exist a constant $C_2>0$ such that 
\begin{equation}\label{claim1}
\begin{aligned}
& \left | \int_B e^{i\varphi} \chi_{\Omega\cup\Omega^*}(\widetilde{q_1}-\widetilde{q_2}) U_1\overline{U_2} \right | \\
&  \leq C_2\left(   e^{2\tau \kappa}  dist \, (C_1^\Gamma, C_2^\Gamma)  +\left |   \log  dist \, (C_1^\Gamma, C_2^\Gamma)   \right |^{-\frac{\theta s^2}{(s+2)(n+ns+2s)}}  \tau^{(s+4)/(s+2)}    \right),
\end{aligned}
\end{equation}
for every $\xi \in E_R$. In fact, adding and subtracting the same terms we have the obvious identity:
\begin{equation}\label{cambio}
\begin{aligned}
&\int_B e^{i\varphi} \chi_{\Omega\cup\Omega^*}(\widetilde{q_1}-\widetilde{q_2}) U_1\overline{U_2}    \\
&  = \int_{B} e^{i\varphi} ( \chi_{\Omega\cup\Omega^*} \widetilde{A_1}-(\chi_{\Omega\cup\Omega^*} \widetilde{A_2}+\nabla \varphi))\cdot \left( DU_1\overline{U_2} +  U_1 \overline{DU}_2 \right)\\
%&= \int_{B} e^{i\varphi} ( \chi_\Omega A_1-(\chi_\Omega A_2+\nabla \varphi))\cdot \left( DU_1\overline{U_2} +  U_1 \overline{DU}_2 \right)\\
&   + \int_B e^{i\varphi}\left(  \chi_{\Omega\cup\Omega^*} A_1^2-(\chi_{\Omega\cup\Omega^*} \widetilde{A_2}+\nabla \varphi)^2 + \chi_{\Omega\cup\Omega^*} (\widetilde{q_1}-\widetilde{q_2})\right.\\
& \qquad    \qquad  \left.  -(\chi_{\Omega\cup\Omega^*}( \widetilde{A_1}- \widetilde{A_2}) -\nabla \varphi)\cdot \nabla \varphi \right) U_1\overline{U}_2\\
&   - \int_{B} e^{i\varphi} ( \chi_{\Omega\cup\Omega^*} \widetilde{A_1}-(\chi_{\Omega\cup\Omega^*} \widetilde{A_2}+\nabla \varphi))\cdot \left( DU_1\overline{U_2} +  U_1 \overline{DU}_2 \right)\\
&  -  \int_B e^{i\varphi}\left[ \chi_{\Omega\cup\Omega^*} \widetilde{A_1}^2-(\chi_{\Omega\cup\Omega^*} \widetilde{A_2}+\nabla \varphi)^2 \right.\\
& \qquad  \qquad  \qquad  \qquad  \qquad \left.- (\chi_{\Omega\cup\Omega^*}( \widetilde{A_1}- \widetilde{A_2}) -\nabla \varphi)\cdot \nabla \varphi \right]  U_1\overline{U}_2.
\end{aligned}
\end{equation}
For convenience and as in \cite{CP}, we set $\varphi^\prime:=(1-\chi)(\psi-\underline{\psi})$.  Then we have  $\psi-\underline{\psi}= \chi(\psi-\underline{\psi})+ (1-\chi)(\psi-\underline{\psi})=\varphi + \varphi^\prime$, which implies that
\begin{equation}\label{diver}
\nabla \psi = \nabla \varphi+\nabla \varphi^\prime.
\end{equation}
From this identity, it follows that
\begin{equation}\label{decz}
\chi_{\Omega\cup\Omega^*} \widetilde{A_1}-(\chi_{\Omega\cup\Omega^*} \widetilde{A_2}+\nabla \psi)= \chi_{\Omega\cup\Omega^*} \widetilde{A_1}-(\chi_{\Omega\cup\Omega^*} \widetilde{A_2}+\nabla \varphi) -\nabla\varphi^\prime
\end{equation}
and
\begin{equation}
\begin{aligned}\label{diver1}
& \chi_{\Omega\cup\Omega^*}\widetilde{A_1}^2-(\chi_{\Omega\cup\Omega^*} \widetilde{A_2}+\nabla \psi )^2\\
&\;  = \left[ \chi_{\Omega\cup\Omega^*} \widetilde{A_1}-(\chi_{\Omega\cup\Omega^*}  \widetilde{A_2}+\nabla \psi )\right]\cdot \left[\chi_{\Omega\cup\Omega^*}  \widetilde{A_1}+(\chi_{\Omega\cup\Omega^*}  \widetilde{A_2}+\nabla \psi )\right]\\
&\;  = \left [ \chi_{\Omega\cup\Omega^*}  \widetilde{A_1}-(\chi_{\Omega\cup\Omega^*}  \widetilde{A_2}+\nabla \varphi) - \nabla\varphi^\prime   \right ]\\
& \qquad   \qquad   \qquad  \qquad   \qquad    \qquad   \qquad  \cdot  \left [\chi_{\Omega\cup\Omega^*}  \widetilde{A_1}+(\chi_{\Omega\cup\Omega^*}  \widetilde{A_2}+\nabla \varphi) +\nabla\varphi^\prime   \right ]\\
&\;  =  \chi_{\Omega\cup\Omega^*}  \widetilde{A_1}^2-(\chi_{\Omega\cup\Omega^*}  \widetilde{A_2}+\nabla \varphi )^2 +  \left [  \chi_{\Omega\cup\Omega^*}  \widetilde{A_1}-(\chi_{\Omega\cup\Omega^*}  \widetilde{A_2}+\nabla \varphi)  \right ]\cdot  \nabla\varphi^\prime\\
&\quad -   \left [  \chi_{\Omega\cup\Omega^*}  \widetilde{A_1}+(\chi_{\Omega\cup\Omega^*}  \widetilde{A_2}+\nabla \varphi)  \right ]\cdot  \nabla\varphi^\prime - \nabla \varphi^\prime \cdot \nabla \varphi^\prime.
\end{aligned}
\end{equation}
Hence, replacing (\ref{diver})-(\ref{diver1}) into (\ref{cambio}) and by a straightforward computation, we obtain 
\[
\int_B e^{i\varphi} \chi_{\Omega\cup\Omega^*} (\widetilde{q_1}-\widetilde{q_2}) U_1\overline{U_2}:= I +II + III + IV,
\]
where
\begin{equation}
I:=- \int_B e^{i\varphi} \left (\chi_{\Omega\cup\Omega^*}  \widetilde{A_1}^2 - (\chi_{\Omega\cup\Omega^*}  \widetilde{A_2} + \nabla \psi)^2\right) U_1\overline{U_2} ,
\end{equation}
\begin{equation}
II: \int_B e^{i\varphi} \left(\chi_{\Omega\cup\Omega^*}  \widetilde{A_1} - \chi_{\Omega\cup\Omega^*}  \widetilde{A_2} - \nabla \psi\right)\cdot \nabla \varphi \, U_1\overline{U_2},
\end{equation}
\begin{equation}
III:=- \int_B e^{i\varphi} \left(\chi_{\Omega\cup\Omega^*}  \widetilde{A_1} - \chi_{\Omega\cup\Omega^*}  \widetilde{A_2} - \nabla \psi\right)\cdot (DU_1 \overline{U_2}+U_1\overline{DU_2} ),
\end{equation}
and
\begin{equation}\label{cuatrr}
\begin{aligned}
& IV:= \int_B e^{i\varphi}  (\chi_{\Omega\cup\Omega^*}  \widetilde{A_1} - \chi_{\Omega\cup\Omega^*} \widetilde{A_2} - \nabla \psi) \cdot  (DU_1 \overline{U_2}+U_1\overline{DU_2} ) \\
&  \qquad  \quad +\int_B e^{i\varphi} \left[  \chi_{\Omega\cup\Omega^*}  \widetilde{A_1}^2 - (\chi_{\Omega\cup\Omega^*}  \widetilde{A_2} + \nabla \psi)^2 + \chi_{\Omega\cup\Omega^*}  (\widetilde{q_1}-\widetilde{q_2})\right.\\
& \qquad  \qquad \qquad  \qquad  \left. - (\chi_{\Omega\cup\Omega^*}  \widetilde{A_1} - \chi_{\Omega\cup\Omega^*} \widetilde{A_2} - \nabla \psi)\cdot \nabla \varphi \right] U_1\overline{U_2}.
\end{aligned}
\end{equation}
Then, by the triangle inequality, it is immediate that
\begin{equation}\label{ine11}
\left | \int_B e^{i\varphi} \chi_{\Omega\cup\Omega^*} (\widetilde{q_1}-\widetilde{q_2}) U_1\overline{U_2} \right | \leq \left |  I \right | + \left |II  \right | + \left | III \right | + \left | IV \right |.
\end{equation}
Now we will estimate  the terms from $\left | I \right |$ to $  \left | IV\right |$. 
%It is clear that 
%\begin{equation}\label{unoone1}
  %\left | I \right | = \left |  \int_B e^{i\varphi} (\chi_\Omega A_1 - (\chi_\Omega A_2 + \nabla \psi))\cdot   (\chi_\Omega A_1 + (\chi_\Omega A_2 + \nabla \psi)) U_1\overline{U_2}   \right |.
%\end{equation}
From (\ref{unos})-(\ref{doss}) with $\Omega\cup\Omega^*$ replaced by $B$ and taking into account (\ref{prs}), we have 
\begin{equation}\label{u1u2ww}
\begin{aligned}
 U_1\overline{U_2}  &= e^{i\xi\cdot x} ( e^{\Phi_1^{\sharp} + \overline{\Phi_2^\sharp} }  +  e^{\Phi_1^\sharp} \overline{r_2} + e^{\overline{\Phi_2^\sharp}}r_1 + r_1\overline{r_2}  )\\
 &\quad  -e^{ i\left( \xi^\prime, 2  \sqrt{\tau^{2}- \frac{\left | \xi \right |^2}{4}} \frac{ \left |  \xi^\prime \right |}{ \left |  \xi \right |} \right)} ( e^{\Phi_1^{\sharp} + \overline{{\Phi_2^\sharp}^*} }  +  e^{\Phi_1^\sharp} \overline{r_2^*} + e^{\overline{{\Phi_2^\sharp}^*}}r_1 + r_1\overline{r_2^*}  )\\
 & \quad - e^{ i\left( \xi^\prime, -2  \sqrt{\tau^{2}- \frac{\left | \xi \right |^2}{4}} \frac{ \left |  \xi^\prime \right |}{ \left |  \xi \right |} \right)} ( e^{{\Phi_1^{\sharp}}^* + \overline{{\Phi_2^\sharp}^*} }  +  e^{{\Phi_1^\sharp}^*} \overline{r_2} + e^{\overline{{\Phi_2^\sharp}}}r_1^* + r_1^*\overline{r_2}  )\\
 & \quad + e^{i\xi^*\cdot x} ( e^{{\Phi_1^{\sharp}}^* + \overline{{\Phi_2^\sharp}^*} }  +  e^{{\Phi_1^\sharp}^*} \overline{r_2^*} + e^{\overline{{\Phi_2^\sharp}^*}}r_1^* + r_1^*\overline{r_2^*}  ).
\end{aligned}
\end{equation}
For $j=1,2$, from (\ref{prdoszss}) with $\Omega\cup\Omega^*$ replaced by $B$ and Sobolev's embedding, we deduce that there exist two positive constant $C_2$ and $C_3$ such that
\begin{equation}\label{r1r2ww} 
\begin{aligned}
& \left \| r_j \right \|_{L^{2n/(n-2)}(B)} + \left \| r_j^* \right \|_{L^{2n/(n-2)}(B)} \\
& \leq C_2 \left(  \left \| r_j \right \|_{W^{1,2}(B)} +  \left \| r_j^* \right \|_{W^{1,2}(B)}   \right)\\
&  \leq C_3\, \tau^{2/(s+2)}.
\end{aligned}
\end{equation}
Also, from (\ref{prunozss}) and the the boundedness of $B$, we get
\begin{equation}\label{trutyz1}
\left \| e^{\Phi_j^\sharp} \right \|_{L^{2n/(n-2)}(B)} + \left \| e^{\Phi_j^{\sharp^*}} \right \|_{L^{2n/(n-2)}(B)}\leq C_4.
\end{equation}
Thus, from (\ref{phiiiii111}) and the boundedness of $B$ it follows that $\chi_{\Omega\cup\Omega^*}  (\widetilde{A_1}-\widetilde{A_2})- \nabla \psi\in L^n(B)$. For the same reason we have that $\chi_{\Omega\cup\Omega^*}  (\widetilde{A_1}-\widetilde{A_2})+\nabla \psi\in L^n(B)$. Hence, from (\ref{phiiiii111}), (\ref{r1r2ww})-(\ref{trutyz1}), the boundedness of $B$ and applying H\"older's inequality for $1/n + 1/n + (n-2)/(2n)+ (n-2)/(2n)=1$, we get
\begin{equation}\label{uno11}
\begin{aligned}
  \left | I \right | &  =\left |  \int_B e^{i\varphi} (\chi_{\Omega\cup\Omega^*}  \widetilde{A_1} - (\chi_{\Omega\cup\Omega^*}  \widetilde{A_2} + \nabla \psi)) \right.\\
  & \qquad \qquad \qquad  \qquad \qquad \left. \cdot   (\chi_{\Omega\cup\Omega^*}  \widetilde{A_1} + (\chi_{\Omega\cup\Omega^*}  \widetilde{A_2} + \nabla \psi)) U_1\overline{U_2}   \right |. \\
  &\quad   \leq C_5 \left \| \chi_{\Omega\cup\Omega^*} (\widetilde{A_1}-\widetilde{A_2}) -\nabla \psi \right \|_{L^n(B)} \tau^{4/(s+2)}.
\end{aligned}
\end{equation}
To estimate the second term we use that $\chi_{\Omega\cup\Omega^*}  (\widetilde{A_1}-\widetilde{A_2})- \nabla \psi\in L^n(B)$. Moreover, from (\ref{phiiiii111}) we have that $\nabla \varphi \in L^n(B)$. Thus, applying H\"older's inequality for $1/n + 1/n + (n-2)/(2n)+ (n-2)/(2n)=1$, we get
\begin{equation}\label{dos22}
 \left | II \right | \leq C_6  \left \| \chi_{\Omega\cup\Omega^*} (\widetilde{A_1}-\widetilde{A_2}) -\nabla \psi \right \|_{L^n(B)} \tau^{4/(s+2)}.
\end{equation}
To estimate the terms $III$ and $IV$ we take into account the computations done in (\ref{lilax1}). Thus, for $j=1,2$, from (\ref{r1r2ww}), (\ref{prdoszss}) with $\Omega\cup\Omega^*$ replaced by $B$ and H\"older's inequality applied to $1/n+ (n-2)/(2n)+1/2=1$, we obtain
\begin{equation}\label{tres33}
 \left | III \right | \leq C_7  \left \| \chi_{\Omega\cup\Omega^*} (\widetilde{A_1}-\widetilde{A_2}) -\nabla \psi \right \|_{L^n(B)} \tau^{(s+4)/(s+2)}.
\end{equation}
The estimate for $IV$ requires a more delicate analysis. Replacing (\ref{decz})-(\ref{diver1}) into (\ref{cuatrr}) and by a straightforward computation, we get the following identity:
\begin{align*}
& IV:= \int_B e^{i\varphi}  (\chi_{\Omega\cup\Omega^*}  \widetilde{A_1} - \chi_{\Omega\cup\Omega^*}  \widetilde{A_2} - \nabla \varphi) \cdot  (DU_1 \overline{U_2}+U_1\overline{DU_2} ) \\
&  \quad  \quad +\int_B e^{i\varphi} \left[  \chi_{\Omega\cup\Omega^*}  \widetilde{A_1}^2 - (\chi_{\Omega\cup\Omega^*}  \widetilde{A_2} + \nabla \varphi)^2 + \chi_{\Omega\cup\Omega^*}  (\widetilde{q_1}-\widetilde{q_2})\right.\\
& \qquad  \qquad \qquad   \left. - (\chi_{\Omega\cup\Omega^*}  \widetilde{A_1} -\chi_{\Omega\cup\Omega^*}  \widetilde{A_2} - \nabla \varphi)\cdot \nabla \varphi \right] U_1\overline{U_2}\\
&\quad \quad  -\int_{B}e^{i\varphi}\nabla \varphi^{\prime}  \cdot  (DU_1 \overline{U_2}+U_1\overline{DU_2} ) \\
&\quad \quad -\int_B e^{i\varphi}\left(2\chi_{\Omega\cup\Omega^*}  \widetilde{A_2}\cdot \nabla \varphi^\prime+\nabla\psi\cdot \nabla\varphi^\prime\right)U_1\overline{U_2}.
%&\quad \quad -\int_B e^{i\varphi}\left(2\chi_\Omega A_2\cdot \nabla \varphi^\prime+(\nabla\varphi+\nabla\varphi^\prime)\cdot \nabla\varphi^\prime\right)U_1\overline{U_2}.
\end{align*}
Notice that since the functions $\chi_{\Omega\cup\Omega^*}  \widetilde{A_2}$ and $\nabla \varphi^\prime$ have disjoint supports, it follows that $\int_B e^{i\varphi}(\chi_{\Omega\cup\Omega^*}  \widetilde{A_2})\cdot \nabla \varphi^{\prime}=0$. Hence, Lemma \ref{coAlidww} and the triangular inequality imply that
\begin{align*}
& \left | IV \right | \leq C_8 \left(   dist \, (C_1^\Gamma, C_2^\Gamma)  \left \| U_1 \right \|_{H^1(\Omega\cup\Omega^*)}  \left \| U_2 \right \|_{H^1(\Omega\cup\Omega^*)}  \right.\\
& \qquad \qquad \left. + \left | \int_B e^{i\varphi}  \nabla \varphi^\prime\cdot (U_1 \overline{DU_2}+ DU_1\overline{U_2}) + \nabla \psi \cdot \nabla \varphi^\prime U_1 \overline{U_2}  \right |    \right).
\end{align*}
Once again applying H\"older's inequality and by similar arguments as used to estimate $U_1 \overline{DU_2}+ DU_1\overline{U_2}$ and $U_1 \overline{U_2}$, for the terms $I, II$ and $III$;  we get
\begin{equation}\label{cuatro44}
\begin{aligned}
 \left | IV \right | &  \leq C_9 \left( dist \, (C_1^\Gamma, C_2^\Gamma)  \left \| U_1 \right \|_{H^1(\Omega\cup\Omega^*)}  \left \| U_2 \right \|_{H^1(\Omega\cup\Omega^*)} \right.\\
  &\qquad  \qquad + \left.  \left \| \nabla \varphi^\prime \right \|_{L^n(B)}  \tau^{(s+4)/(s+2)} \right). 
\end{aligned}
\end{equation}
Furthermore, an elementary interpolation, the boundedness of $B$, (\ref{mreyww}) and the estimate (\ref{dmpoik}), imply that
\begin{equation}\label{diez1010}
\begin{aligned}
& \left \| \chi_{\Omega\cup\Omega^*} (\widetilde{A_1}-\widetilde{A_2})- \nabla \psi \right \|_{L^n(B)} \\
&   \leq \left \| \chi_{\Omega\cup\Omega^*} (\widetilde{A_1}-\widetilde{A_2})- \nabla \psi \right \|_{L^2(B)}^\theta  \left \| \chi_{\Omega\cup\Omega^*}  (\widetilde{A_1}-\widetilde{A_2})- \nabla \psi \right \|_{L^p(B)}^{1-\theta}\\
& \qquad \leq C_{10}   \left \| d (\chi_{\Omega\cup\Omega^*} (\widetilde{A_1}-\widetilde{A_2})) \right \|_{H^{-1}(B)}^\theta \\
& \qquad \leq C_{10}  \left \| d (\chi_{\Omega\cup\Omega^*} (\widetilde{A_1}-\widetilde{A_2})) \right \|_{H^{-1}(\mathbb{R}^n)}^\theta\\
& \qquad \leq C_{11} \left |   \log  dist \, (C_1^\Gamma, C_2^\Gamma)   \right |^{-\frac{\theta s^2}{(s+2)(n+ns+2s)}},
\end{aligned}
\end{equation}
for every $\theta\in \left( 0,2/n \right)$ and $p$ is chosen such that $\theta/2 + (1-\theta)/p=1/n$. Notice that since $\varphi^\prime=(1-\chi)(\psi-\psi^*)$ we deduce that
\[
\left \| \nabla \varphi^\prime \right \|_{L^n(B)} \leq C_7 \left \| \psi - \psi^* \right \|_{W^{1,n}(B\setminus\overline{B^\prime})}
\]
and again by an elementary interpolation, (\ref{rywww}), the boundedness of $B$, Theorem \ref{SMP} and since $1-\chi\equiv 0$ in $B^\prime$, we get
\begin{equation}\label{once1111}
\begin{aligned}
 \left \| \nabla \varphi^\prime \right \|_{L^n(B)} & \leq C_{12}  \left \| \psi - \psi^* \right \|_{H^1(B\setminus\overline{B^\prime})}^\theta   \left \| \psi - \psi^* \right \|_{W^{1,p}(B\setminus\overline{B^\prime})}^{1-\theta}\\
&  \leq C_{13}   \left \| d (\chi_{\Omega\cup\Omega^*}(\widetilde{A_1}-\widetilde{A_2})) \right \|_{H^{-1}(B)}^\theta \\
& \leq C_{13}  \left \| d (\chi_{\Omega\cup\Omega^*}(\widetilde{A_1}-\widetilde{A_2})) \right \|_{H^{-1}(\mathbb{R}^n)}^\theta\\
&  \leq C_{14} \left |   \log  dist \, (C_1^\Gamma, C_2^\Gamma)   \right |^{-\frac{\theta s^2}{(s+2)(n+ns+2s)}}.
\end{aligned}
\end{equation}
Hence, using estimate (\ref{diez1010}) into (\ref{uno11})-(\ref{tres33}) we get
\[
\left |  I\right |+ \left | II \right | + \left | III \right |\leq C_{15}   \left |   \log  dist \, (C_1^\Gamma, C_2^\Gamma)   \right |^{-\frac{\theta s^2}{(s+2)(n+ns+2s)}} \tau^{(s+4)/(s+2)},
\]
and using estimate (\ref{once1111}) into (\ref{cuatro44}) we obtain
\[
\left |IV  \right | \leq C_{16} (   e^{2\tau k}  dist \, (C_1^\Gamma, C_2^\Gamma)  +\left |   \log  dist \, (C_1^\Gamma, C_2^\Gamma)   \right |^{-\frac{\theta s^2}{(s+2)(n+ns+2s)}}   \tau^{(s+4)/(s+2)} ).
\]
We conclude the proof of  Claim $1$ by combining the two above estimates into (\ref{ine11}) and then the estimate (\ref{claim1}) follows.\\ 

The second step will be to prove the following claim.
\subsection*{Claim $2$} There exists $C_{17}>0$ such that
\begin{equation}\label{claim2}
\begin{aligned}
& \left | \int_B e^{i\varphi}    \chi_{\Omega\cup\Omega^*}(\widetilde{q_1}-\widetilde{q_2}) e^{i\xi\cdot x} e^{\Phi_1^\sharp + \overline{\Phi_2^\sharp}}   + \int_B e^{i\varphi}   \chi_{\Omega\cup\Omega^*} (\widetilde{q_1}-\widetilde{q_2}) e^{i\xi^*\cdot x} e^{{\Phi_1^\sharp}^* + \overline{{\Phi_2^\sharp}^*}}  \right |\\
& \leq C_{17} \left(  \left | \int_B e^{i\varphi} \chi_{\Omega\cup\Omega^*}(\widetilde{q_1}-\widetilde{q_2}) U_1\overline{U_2} \right | \right. \\
&\qquad   \left. + \tau^{-s/(s+2)} +    \left |   \log  dist \, (C_1^\Gamma, C_2^\Gamma)   \right |^{-\frac{ s^2}{(s+2)(n+ns+2s)}} + e^{-4 \pi \epsilon^2 \tau^2 \frac{\left |\xi^\prime  \right |^2}{\left |  \xi \right |^2} }  + \epsilon^s  \right),
\end{aligned}
\end{equation}
for every $\xi \in \bigcap_{l=1}^{n-1}E_l$. In fact, the idea will be to isolate the function $\chi_{\Omega\cup\Omega^*}(\widetilde{q_1}-\widetilde{q_2})$ from the left-hand side of (\ref{claim1}). For convenience, we set
\begin{equation*}
V=   \int_B e^{i\varphi}  \chi_{\Omega\cup\Omega^*}(\widetilde{q_1}-\widetilde{q_2}) e^{i\xi^*\cdot x} \left(  e^{{\Phi_1^\sharp}^*} \overline{r_2^*} + e^{\overline{{\Phi_2^\sharp}^*}} r_1^* + r_1^* \overline{r_2^*} \right),
\end{equation*}
\begin{equation*}
VI  =   \int_B e^{i\varphi}  \chi_{\Omega\cup\Omega^*}(\widetilde{q_1}-\widetilde{q_2}) e^{i\xi\cdot x} \left(  e^{{\Phi_1^\sharp}} \overline{r_2} + e^{\overline{{\Phi_2^\sharp}}} r_1 + r_1 \overline{r_2} \right),
\end{equation*}
\begin{align*}
VII &=  \int_B e^{i\varphi}  \chi_{\Omega\cup\Omega^*}(\widetilde{q_1}-\widetilde{q_2})  e^{i\left( \xi^\prime, 2\tau  \sqrt{1-\tau^{-2} \frac{\left | \xi \right |^2}{4}} \frac{ \left |  \xi^\prime \right |}{ \left |  \xi \right |} \right)} \\
& \qquad  \qquad  \qquad  \qquad  \qquad  \qquad  \qquad  \qquad    \times \left(  e^{{\Phi_1^\sharp}} \overline{r_2^*} + e^{\overline{{\Phi_2^\sharp}^*}} r_1 + r_1 \overline{r_2^*} \right),
\end{align*}
\begin{align*}
VIII &=  \int_B e^{i\varphi} \chi_{\Omega\cup\Omega^*}(\widetilde{q_1}-\widetilde{q_2})  e^{i\left( \xi^\prime, -2\tau  \sqrt{1-\tau^{-2} \frac{\left | \xi \right |^2}{4}} \frac{ \left |  \xi^\prime \right |}{ \left |  \xi \right |} \right)}  \\
&  \qquad  \qquad  \qquad  \qquad  \qquad   \qquad  \qquad  \qquad  \times  \left(  e^{{\Phi_1^\sharp}^*} \overline{r_2} + e^{\overline{{\Phi_2^\sharp}^*}} r_1^* + r_1^* \overline{r_2} \right),
\end{align*}
\begin{equation*}
IX =  \int_B e^{i\varphi} \chi_{\Omega\cup\Omega^*}(\widetilde{q_1}-\widetilde{q_2})  e^{i\left( \xi^\prime, 2\tau  \sqrt{1-\tau^{-2} \frac{\left | \xi \right |^2}{4}} \frac{ \left |  \xi^\prime \right |}{ \left |  \xi \right |} \right)} e^{\Phi_1^\sharp + \overline{\Phi_2^\sharp}},
\end{equation*}
\begin{equation*}
X =  \int_B e^{i\varphi}  \chi_{\Omega\cup\Omega^*}(\widetilde{q_1}-\widetilde{q_2})  e^{i\left( \xi^\prime, -2\tau  \sqrt{1-\tau^{-2} \frac{\left | \xi \right |^2}{4}} \frac{ \left |  \xi^\prime \right |}{ \left |  \xi \right |} \right)} e^{{\Phi_1^\sharp}^* + \overline{{\Phi_2^\sharp}^*}}.
\end{equation*}
Thus, adding and subtracting terms we obtain
\begin{equation}\label{jijo}
\begin{aligned}
& \left | \int_B e^{i\varphi}     \chi_{\Omega\cup\Omega^*}(\widetilde{q_1}-\widetilde{q_2})  e^{i\xi\cdot x} e^{\Phi_1^\sharp + \overline{\Phi_2^\sharp}}   + \int_B e^{i\varphi}     \chi_{\Omega\cup\Omega^*}(\widetilde{q_1}-\widetilde{q_2})  e^{i\xi^*\cdot x} e^{{\Phi_1^\sharp}^* + \overline{{\Phi_2^\sharp}^*}}  \right |\\
& \leq  \left |  \int_B e^{i\varphi}  \chi_{\Omega\cup\Omega^*}(\widetilde{q_1}-\widetilde{q_2}) U_1\overline{U_2}  \right | + \left |  V+VI +VII +VIII +IX +X \right |.
\end{aligned}
\end{equation}
Now the task is to estimate each term of the right-hand side of the above inequality. From (\ref{prdoszss}) with $\Omega\cup\Omega^*$ replaced by $B$, (\ref{phiiiii111}), the boundedness of $B$ and H\"older's inequality in $L^2(B)$,  we get
\begin{equation}\label{jeji}
\left | V + VI + VII + VIII \right |  \leq C_{18}\, \tau^{-s/(s+2)}.
\end{equation}
To estimate $IX$ and $X$ require a more delicate analysis. Adding and subtracting terms it is easy to check that 
\begin{equation}\label{000}
\begin{aligned}
\left | IX \right |& \leq  \left |  \int_B  \chi_{\Omega\cup\Omega^*}(\widetilde{q_1}-\widetilde{q_2})  e^{i\left( \xi^\prime, 2\tau  \sqrt{1-\tau^{-2} \frac{\left | \xi \right |^2}{4}} \frac{ \left |  \xi^\prime \right |}{ \left |  \xi \right |} \right)} \right | \\
& \quad + \left |    \int_B  \chi_{\Omega\cup\Omega^*}(\widetilde{q_1}-\widetilde{q_2})  e^{i\left( \xi^\prime, 2\tau  \sqrt{1-\tau^{-2} \frac{\left | \xi \right |^2}{4}} \frac{ \left |  \xi^\prime \right |}{ \left |  \xi \right |} \right)} ( 1- e^{\Phi_1 + \overline{\Phi_2}+i\varphi} ) \right |\\
&\quad  + \left |    \int_B e^{i\varphi}  \chi_{\Omega\cup\Omega^*}(\widetilde{q_1}-\widetilde{q_2})  e^{i\left( \xi^\prime, 2\tau  \sqrt{1-\tau^{-2} \frac{\left | \xi \right |^2}{4}} \frac{ \left |  \xi^\prime \right |}{ \left |  \xi \right |} \right)} ( e^{\Phi_1 + \overline{\Phi_2}} - e^{\Phi_1^\sharp + \overline{\Phi_2^\sharp}}   ) \right |.
\end{aligned}
\end{equation}
On one hand, by Lemma \ref{heyw} we have $ \chi_{\Omega\cup\Omega^*}(\widetilde{q_1}-\widetilde{q_2})\in B^{2,\infty}_s(\mathbb{R}^n)$. Thus, by Lemma \ref{RLe} applied to $f= \chi_{\Omega\cup\Omega^*}(\widetilde{q_1}-\widetilde{q_2}) $ , $C_0=\left \|  \chi_{\Omega\cup\Omega^*}(\widetilde{q_1}-\widetilde{q_2}) \right \|_{B^{2,\infty}_s}$ and $\sigma=1$, we obtain
\begin{equation}\label{z100tr}
 \left |  \int_B  \chi_{\Omega\cup\Omega^*}(\widetilde{q_1}-\widetilde{q_2})  e^{i\left( \xi^\prime, 2\tau  \sqrt{1-\tau^{-2} \frac{\left | \xi \right |^2}{4}} \frac{ \left |  \xi^\prime \right |}{ \left |  \xi \right |} \right)} \right | \leq C_{19}\, \left( e^{-4 \pi \epsilon^2 \tau^2 \frac{\left |\xi^\prime  \right |^2}{\left |  \xi \right |^2} }  + \epsilon^s \right).
\end{equation}
On the other hand, from (\ref{xvtyu}) we deduce that
\begin{equation}\label{mukw2}
(i\mu_1+\mu_2)\cdot \nabla (\Phi_1+ \overline{\Phi_2} + i\varphi)= (\mu_1-i\mu_2)\cdot (\chi_{\Omega\cup\Omega^*}(\widetilde{A_1}-\widetilde{A_2})-\nabla \varphi).
\end{equation}
Thus, by the boundeness of  $((i\mu_1+\mu_2)\cdot \nabla)^{-1}$ in weighted $L^2$ spaces, the estimates (\ref{mreyww}) and (\ref{rywww}), we obtain
\begin{equation}\label{jyt091}
\begin{aligned}
& \left \|\Phi_1+ \overline{\Phi_2} + i\varphi  \right \|_{L^2(B)} \leq C_{19} \left \|  \chi_{\Omega\cup\Omega^*}(\widetilde{A_1}-\widetilde{A_2})-\nabla \varphi  \right \|_{L^2(B)}  \\
& \leq C_{20}\left( \left \|  \chi_{\Omega\cup\Omega^*}(\widetilde{A_1}-\widetilde{A_2})-\nabla \psi\right \|_{L^2(B)}  +  \left \| \psi-\psi^*  \right \|_{H^1(B\setminus \overline{B^\prime})}  \right)\\
& \leq C_{21} \left |   \log  dist \, (C_1^\Gamma, C_2^\Gamma)   \right |^{-\frac{ s^2}{(s+2)(n+ns+2s)}}.
\end{aligned}
\end{equation}
From (\ref{ineq}), the boundedness of $B$ and the above inequality we have
\begin{equation}\label{200}
\begin{aligned}
&  \left |    \int_B \chi_{\Omega\cup\Omega^*}(\widetilde{q_1}-\widetilde{q_2})  e^{i\left( \xi^\prime, 2\tau  \sqrt{1-\tau^{-2} \frac{\left | \xi \right |^2}{4}} \frac{ \left |  \xi^\prime \right |}{ \left |  \xi \right |} \right)}  \left( 1- e^{\Phi_1 + \overline{\Phi_2}+i\varphi} \right)  \right |\\  
& \leq C_{22} \left \|  \chi_{\Omega\cup\Omega^*}(\widetilde{q_1}-\widetilde{q_2}) \right \|_{L^2(B)} \left \|\Phi_1+ \overline{\Phi_2} + i\varphi  \right \|_{L^2(B)}\\
& \leq C_{23}   \left |   \log  dist \, (C_1^\Gamma, C_2^\Gamma)   \right |^{-\frac{ s^2}{(s+2)(n+ns+2s)}}.
 \end{aligned}
\end{equation}
From (\ref{zuno1}), (\ref{ineq}), (\ref{phiiiii111}) and following similar arguments to estimate (\ref{two2}), we get
\begin{equation}\label{300}
\begin{aligned}
& \left |    \int_B e^{i\varphi} \chi_{\Omega\cup\Omega^*}(\widetilde{q_1}-\widetilde{q_2}) e^{i(\xi^\prime, 2\tau \left | \xi^\prime \right |  )} ( e^{\Phi_1 + \overline{\Phi_2}} - e^{\Phi_1^\sharp + \overline{\Phi_2^\sharp}}   ) \right |  \\
&\qquad  \leq  \left \| e^{i\varphi} \chi_{\Omega\cup\Omega^*}(\widetilde{q_1}-\widetilde{q_2}) \right \|_{L^2(B)}  \left \| \Phi_1 - \Phi_1^\sharp+ \overline{\Phi_2} - \Phi_2^\sharp\right \|_{L^2(B)}\\
& \qquad \leq C_{24}\,  \tau^{-s/(s+2)}.
\end{aligned}
\end{equation}
Hence, by replacing (\ref{z100tr})-(\ref{300}) into (\ref{000}), we obtain
\begin{equation}\label{––p09}
\left | IX \right |\leq C_{25}( \tau^{-s/(s+2)} +    \left |   \log  dist \, (C_1^\Gamma, C_2^\Gamma)   \right |^{-\frac{ s^2}{(s+2)(n+ns+2s)}} + e^{-4 \pi \epsilon^2 \tau^2 \frac{\left |\xi^\prime  \right |^2}{\left |  \xi \right |^2} }  + \epsilon^s).
\end{equation}
Now we estimate the term $X$. From (\ref{xvtyu}) and (\ref{mukw2}) we deduce 
\begin{equation}
(i\mu_1^*+\mu_2^*)\cdot \nabla (\Phi_1^*+ \overline{\Phi_2^*} + i\varphi)= (\mu_1^*-i\mu_2^*)\cdot (\chi_{\Omega\cup\Omega^*}(\widetilde{A_1}-\widetilde{A_2})-\nabla \varphi),
\end{equation}
which imply that
\begin{equation}\label{jyt0911}
\begin{aligned}
& \left \|\Phi_1^*+ \overline{\Phi_2^*} + i\varphi  \right \|_{L^2(B)} \leq C_{19} \left \|  \chi_{\Omega\cup\Omega^*}(\widetilde{A_1}-\widetilde{A_2})-\nabla \varphi  \right \|_{L^2(B)}  \\
& \leq C_{20}\left( \left \|  \chi_{\Omega\cup\Omega^*}(\widetilde{A_1}-\widetilde{A_2})-\nabla \psi\right \|_{L^2(B)}  +  \left \| \psi-\psi^*  \right \|_{H^1(B\setminus \overline{B^\prime})}  \right)\\
& \leq C_{21} \left |   \log  dist \, (C_1^\Gamma, C_2^\Gamma)   \right |^{-\frac{ s^2}{(s+2)(n+ns+2s)}}.
\end{aligned}
\end{equation}
Thus, by similar reasoning applied to estimate $IX$, we can deduce that
\begin{equation}\label{––p091}
\left | X \right |\leq C_{26}( \tau^{-s/(s+2)} +    \left |   \log  dist \, (C_1^\Gamma, C_2^\Gamma)   \right |^{-\frac{ s^2}{(s+2)(n+ns+2s)}} + e^{-4 \pi \epsilon^2 \tau^2 \frac{\left |\xi^\prime  \right |^2}{\left |  \xi \right |^2} }  + \epsilon^s).
\end{equation}
We conclude the proof of Claim $2$ by combining the estimates (\ref{jeji}) and (\ref{––p09})-(\ref{––p091}) into (\ref{jijo}).\\

The last step will be to use the estimates from the claims $1$ and $2$ in order to deduce the assertion of this proposition. By a straightforward computation (adding and subtracting terms), for every $\xi\in \bigcap_{l=1}^{n-1}E_l$, we can deduce that
\begin{equation}\label{mmmmmzqa123}
\begin{aligned}
&\left |  \mathcal{F}\left[ \chi_{\Omega\cup\Omega^*} \widetilde{q_1}\right](\xi) -\mathcal{F}\left[ \chi_{\Omega\cup\Omega^*} \widetilde{q_2}\right](\xi)  \right |\\
& \quad= \left |   \int_B \chi_{\Omega\cup\Omega^*}  (\widetilde{q_1}-\widetilde{q_2})e^{i\xi\cdot x} + \int_B  \chi_{\Omega\cup\Omega^*}  (\widetilde{q_1}-\widetilde{q_2}) e^{i\xi^*\cdot x}  \right | \\
& \quad :=\left |   M_1+ M_2 +M_3+ M_4 + M_5   \right |,
\end{aligned}
\end{equation}
where
\[
M_1=  \int_B e^{i\varphi}      \chi_{\Omega\cup\Omega^*}  (\widetilde{q_1}-\widetilde{q_2}) e^{i\xi\cdot x} e^{\Phi_1^\sharp + \overline{\Phi_2^\sharp}}   + \int_B e^{i\varphi}     \chi_{\Omega\cup\Omega^*}  (\widetilde{q_1}-\widetilde{q_2}) e^{i\xi^*\cdot x} e^{{\Phi_1^\sharp}^* + \overline{{\Phi_2^\sharp}^*}},
\]
\[
M_2=  \int_B  \chi_{\Omega\cup\Omega^*}  (\widetilde{q_1}-\widetilde{q_2}) e^{i\xi\cdot x} (1- e^{\Phi_1+ \overline{\Phi_2}+i\varphi}),
\]
\[
M_3= \int_B \chi_{\Omega\cup\Omega^*}  (\widetilde{q_1}-\widetilde{q_2}) e^{i\xi^*\cdot x} (1- e^{\Phi_1^*+ \overline{\Phi_2^*}+i\varphi}), 
\]
\[
M_4=  \int_B e^{i\varphi} \chi_{\Omega\cup\Omega^*}  (\widetilde{q_1}-\widetilde{q_2}) e^{i\xi\cdot x} (e^{\Phi_1+ \overline{\Phi_2}} - e^{\Phi_1^\sharp + \overline{\Phi_2^\sharp}}  )  
\]
and
\[
M_5=  \int_B e^{i\varphi} \chi_{\Omega\cup\Omega^*}  (\widetilde{q_1}-\widetilde{q_2}) e^{i\xi^*\cdot x} (e^{\Phi_1^*+ \overline{\Phi_2^*}} - e^{{\Phi_1^\sharp}^* + \overline{{\Phi_2^\sharp}^*}}  ). 
\]
The task is now to estimate each term of the above identities. From the claims $1$ and $2$, we obtain
\begin{equation}\label{mmmmmzqa1}
\begin{aligned}
& \left | M_1 \right | \leq C_{26}( e^{2\tau \kappa}  dist \, (C_1^\Gamma, C_2^\Gamma)  +\left |   \log  dist \, (C_1^\Gamma, C_2^\Gamma)   \right |^{-\frac{\theta s^2}{(s+2)(n+ns+2s)}}  \tau^{(s+4)/(s+2)}   \\
&\quad \quad   +   \tau^{-s/(s+2)} +    \left |   \log  dist \, (C_1^\Gamma, C_2^\Gamma)   \right |^{-\frac{ s^2}{(s+2)(n+ns+2s)}} + e^{-4 \pi \epsilon^2 \tau^2 \frac{\left |\xi^\prime  \right |^2}{\left |  \xi \right |^2} }  + \epsilon^s ).
\end{aligned}
\end{equation}
From the boundedness of $B$, (\ref{ineq}) and (\ref{jyt091}), we get 
\begin{equation}
\begin{aligned}
\left |M_2  \right |&\leq C_{27} \left \|  \chi_{\Omega\cup\Omega^*}  (\widetilde{q_1}-\widetilde{q_2}) \right \|_{L^2(B)}   \left \| \Phi_1 + \overline{\Phi_2}+i\varphi \right \|_{L^2(B)}\\
& \leq C_{28}  \left |   \log  dist \, (C_1^\Gamma, C_2^\Gamma)   \right |^{-\frac{ s^2}{(s+2)(n+ns+2s)}}.
\end{aligned}
\end{equation}
By similar arguments and from (\ref{jyt0911}), we have
\begin{equation}
\left |M_3  \right |\leq C_{29}\left |   \log  dist \, (C_1^\Gamma, C_2^\Gamma)   \right |^{-\frac{ s^2}{(s+2)(n+ns+2s)}}.
\end{equation}
From (\ref{zuno1}), (\ref{ineq}), (\ref{phiiiii111}) and as was estimated (\ref{two2}) and (\ref{300}), we get
\begin{equation}\label{mmmmmzqa12}
\left |M_4 + M_5  \right |\leq C_{30}\,\tau^{-s/(s+2)}.
\end{equation}
We conclude the proof of this Proposition by replacing (\ref{mmmmmzqa1})-(\ref{mmmmmzqa12}) into (\ref{mmmmmzqa123}).
\end{proof}

\subsection{Proof of Theorem \ref{SEP}}

The proof of Theorem \ref{SEP} is standard and similar to the proof of Theorem \ref{SMP}.
By Proposition \ref{ggggg5}, taking into account that the constant $C>0$ in the estimate (\ref{final2}) is independent of $\xi\in \bigcap_{l=1}^{n-1}E_{l}$ and since the set $\bigcap_{l=1}^{n-1}E_{l}$ is dense in $\mathbb{R}^n$, it follows that the following estimate
\begin{equation}\label{final21zswq}
\begin{aligned}
& \left |  \mathcal{F}\left[ \chi_{\Omega\cup\Omega^*}\widetilde{q_1}  \right] (\xi) -  \mathcal{F}\left[ \chi_{\Omega\cup\Omega^*}\widetilde{q_2}  \right] (\xi)     \right | \\
&  \leq C\left( e^{2\tau \kappa}  dist \, (C_1^\Gamma, C_2^\Gamma)  +\left |   \log  dist \, (C_1^\Gamma, C_2^\Gamma)   \right |^{-\frac{\theta s^2}{(s+2)(n+ns+2s)}}  \tau^{(s+4)/(s+2)}   \right.\\
&\qquad \qquad \qquad  \left. +  \tau^{-s/(s+2)} + e^{-4 \pi \epsilon^2 \tau^2 \frac{\left |\xi^\prime  \right |^2}{\left |  \xi \right |^2} }  + \epsilon^s\right),
\end{aligned}
\end{equation}
holds true for all $\xi\in \mathbb{R}^n$. Now consider $R\geq 1$ (which will be fixed later) and denote by $B_R(0)$ the open ball in $\mathbb{R}^n$ centered at $0$ of radius $R$. For convenience we denote $\widetilde{q}:=\chi_{\Omega\cup\Omega^*}(\widetilde{q_1}-\widetilde{q_2})$. By Plancherel's theorem we have
\begin{equation}\label{w3zsd}
\left \| \widetilde{q} \right \|_{L^{2}(\mathbb{R}^n)}^2 =  \int_{B_R(0)\setminus \left \{0  \right \}}\left | \mathcal{F}\left[ \widetilde{q} \right]  (\xi) \right |^2 d\xi    +  \int_{\mathbb{R}^n\setminus B_R(0)}\left | \mathcal{F}\left[ \widetilde{q} \right]  (\xi) \right |^2 d\xi.
\end{equation}
From (\ref{final21zswq}), we get 
\begin{align*}
 &   \int_{B_R(0)\setminus \left \{0  \right \}}\left | \mathcal{F}\left[ \widetilde{q} \right]  (\xi) \right |^2 d\xi\\
 & \leq C_1 R^n \left(   \tau^{-2s/(s+2)}+\epsilon^{2s}  +  e^{4\tau \kappa  }  dist \, (C_1^\Gamma, C_2^\Gamma)^{2}    \right.\\
 &\qquad \qquad  \qquad \left.  +\left |   \log  dist \, (C_1^\Gamma, C_2^\Gamma)   \right |^{-\frac{2\theta s^2}{(s+2)(n+ns+2s)}}  \tau^{2(s+4)/(s+2)}    \right)    \\
 & \quad + C_1\, \int_{B_R(0)\setminus \left \{0  \right \}} e^{-8 \pi \epsilon^2 \tau^2 \frac{\left |\xi^\prime  \right |^2}{\left |  \xi \right |^2} }  \\
 &  \quad \leq C_2\, R^n\left(  \tau^{-2s/(s+2)}  + e^{4\tau \kappa  }  dist \, (C_1^\Gamma, C_2^\Gamma)^{2} + \epsilon^{2s}+\epsilon^{-2} \tau^{-2}\right. \\
 & \qquad \qquad \qquad  \left.  + \left |   \log  dist \, (C_1^\Gamma, C_2^\Gamma)   \right |^{-\frac{2\theta s^2}{(s+2)(n+ns+2s)}}  \tau^{2(s+4)/(s+2)}      \right).
 \end{align*}
%\end{equation}
Thus, by equating $\epsilon^{2s}$ and $\epsilon^{-2}\tau^{-2}$, that is $\epsilon=\tau^{-1/(s+1)}$, we obtain
\begin{equation}\label{w1}
\begin{aligned}
&  \int_{B_R(0)\setminus \left \{0  \right \}}\left | \mathcal{F}\left[ \widetilde{q} \right]  (\xi) \right |^2 d\xi\\
&\leq C_3\, R^n\left(\tau^{-2s/(s+2)} +  e^{4\tau \kappa  }  dist \, (C_1^\Gamma, C_2^\Gamma)^{2} \right.\\
&\qquad \qquad \qquad\left.  + \left |   \log  dist \, (C_1^\Gamma, C_2^\Gamma)   \right |^{-\frac{2\theta s^2}{(s+2)(n+ns+2s)}}  \tau^{2(s+4)/(s+2)}      \right).
\end{aligned}
\end{equation}

We now turn to estimate the integral term on $\mathbb{R}^n\setminus B_R(0)$ from the right-hand side of (\ref{w3zsd}). By hypothesis, the functions $\chi_\Omega q_1$ and $\chi_\Omega q_2$ belong to the class of admissible electric potentials $\mathcal{Q}(\Omega, M, s)$. Hence, from Lemma \ref{heyw}, it follows that $\widetilde{q}\in \mathcal{Q}(\Omega\cup\Omega^*, 2M, s)$. In particular, $\widetilde{q}\in B^{2,\infty}_s(\mathbb{R}^n)$ and $\left \| \widetilde{q} \right \|_{B^{2,\infty}_s}\leq 2M$. By combining Proposition $10$ and Theorem $5$, both in \cite{St1}, we obtain the following chain of embeddings:   
\[
B^{2,\infty}_s(\mathbb{R}^n) \subset B^{2, 2}_{s/2}(\mathbb{R}^n) \subset H^{s/2}(\mathbb{R}^n).
\]
Hence, we deduce that $\widetilde{q}\in H^{s/2}(\mathbb{R}^n)$ and its norm in $H^{s/2}(\mathbb{R}^n)$ only depends on a priori bounds for the magnetic and electric potentials. Then, we have
\begin{equation}\label{w22}
\begin{aligned}
 & \int_{\mathbb{R}^n\setminus B_R(0)}\left | \mathcal{F}\left[ \widetilde{q} \right]  (\xi)\right |^2 d\xi\\
 &= \int_{\mathbb{R}^n\setminus B_R(0)}(1+\left | \xi \right |^2   )^{-s/2} (1+\left | \xi \right |^2   )^{s/2} \left | \mathcal{F}\left[ \widetilde{q} \right]  (\xi)\right |^2 d\xi \\
 & \leq  R^{-s}  \int_{\mathbb{R}^n\setminus B_R(0)}(1+\left | \xi \right |^2   )^{s/2} \left |  \mathcal{F}\left[ \widetilde{q} \right]  (\xi) \right |^2 d\xi \\
& \leq    R^{-s}  \left \|\widetilde{q}  \right \|^2_{H^{s/2}(\mathbb{R}^n)} \leq C_4\, R^{-s}.
\end{aligned}
\end{equation}
Thus, replacing (\ref{w1}) and (\ref{w22}) into (\ref{w3zsd}) we have that there exist two positive constants $C_5$ and $\tau_1$ such that the estimate
\begin{equation}\label{leypoma}
\begin{aligned}
&\left \| \widetilde{q} \right \|_{L^{2}(\mathbb{R}^n)}^2 \leq C_5\, \left( R^n \tau^{-2s/(s+2)} + R^{-s}+  R^n e^{4\tau \kappa  }  dist \, (C_1^\Gamma, C_2^\Gamma)^{2} \right.\\
&\qquad \qquad \qquad\left.  + R^n \left |   \log  dist \, (C_1^\Gamma, C_2^\Gamma)   \right |^{-\frac{2\theta s^2}{(s+2)(n+ns+2s)}}  \tau^{2(s+4)/(s+2)}      \right).
\end{aligned}
\end{equation}
holds true for all $\tau\geq \tau_1$. 
By equating the two first terms on the left-hand side of the above inequality, we take $R=\tau^{2s/((n+s)(s+2))}$. Moreover, there exist two positive constants $C_6$ and $\tau_2$ such that
\[
R^n=\tau^{2ns/((n+s)(s+2)) }  \leq C_6 e^{\tau \kappa}, \quad \tau\geq \tau_2.
\]
Hence, from (\ref{w1}) and (\ref{leypoma}) into (\ref{w3zsd}), we have
\begin{equation}\label{x3}
\begin{aligned}
& \left \| \widetilde{q} \right \|_{L^{2}(\mathbb{R}^n)}^2\leq C_7\left( \tau^{-4s/((n+s)(s+2))}+ e^{5\tau \kappa  }  dist \, (C_1^\Gamma, C_2^\Gamma)^{2}  \right.\\
&\qquad \qquad  \left. +  \left |   \log  dist \, (C_1^\Gamma, C_2^\Gamma)   \right |^{-\frac{2\theta s^2}{(s+2)(n+ns+2s)}}  \tau^{2(2ns+4n+s^2+4s)/((n+s)(s+2))}   \right).
\end{aligned}
\end{equation}
Now we consider a large enough $\tau>0$ such that
\[
  \left |   \log  dist \, (C_1^\Gamma, C_2^\Gamma)   \right |^{-\frac{2\theta s^2}{(s+2)(n+ns+2s)}}    \tau^{2(2ns+4n+s^2+4s)/((n+s)(s+2))}  
\] 
will be comparable with 
\[
\left |   \log  dist \, (C_1^\Gamma, C_2^\Gamma)   \right |^{-\frac{\theta s^2}{(s+2)(n+ns+2s)}} . 
\]
Hence, taking $\tau_0\geq \max \left \{\tau_1,\tau_2  \right \}$ such that $5\kappa\tau_0\geq 1$, it is easy to check that
\begin{equation}\label{3xs}
\tau:= \dfrac{1}{5\kappa}\left |   \log  dist \, (C_1^\Gamma, C_2^\Gamma)   \right |^{\frac{\theta s^2(n+s)}{2(n+ns+2s)(2ns+4n+s^2+4s)}}\geq \tau_0,
\end{equation}
whenever
\begin{equation}\label{lmjbkdsk}
 dist \, (C_1^\Gamma, C_2^\Gamma)  \leq e^{-(5\kappa\tau_0)^{\frac{2(n+ns+2s)(2ns+4n+s^2+4s)}{\theta s^2 (n+s)}}}.
\end{equation}
Thus, from (\ref{3xs}) it follows that 
\begin{equation}\label{zxp091}
 \tau^{-4s/((n+s)(s+2))} \leq C_8 \left |   \log  dist \, (C_1^\Gamma, C_2^\Gamma)   \right |^{-\frac{2\theta s^3}{(s+2)(n+ns+2s)(2ns+4n+s^2+4s)}}.
\end{equation}
From (\ref{3xs}), we have
\[
e^{5\tau \kappa  }  dist \, (C_1^\Gamma, C_2^\Gamma)^{2}=e^{5\tau\kappa - 2(5\tau\kappa)^{\frac{2(n+ns+2s)(2ns+4n+s^2+4s)}{\theta s^2 (n+s)}}} \leq  e^{-5\tau \kappa  }, 
\]
where we have used that $dist \, (C_1^\Gamma, C_2^\Gamma) \leq e^{-1}$. This fact can be easily  deduced from (\ref{lmjbkdsk}) and $5\kappa \tau\geq 1$. Then
\begin{equation}\label{zxp0912}
e^{5\tau \kappa  }  dist \, (C_1^\Gamma, C_2^\Gamma)^{2}\leq \frac{1}{5\tau \kappa}= \left |   \log  dist \, (C_1^\Gamma, C_2^\Gamma)   \right |^{\frac{-\theta s^2(n+s)}{2(n+ns+2s)(2ns+4n+s^2+4s)}}.
\end{equation}
By construction, the last term on the right-hand side of (\ref{x3}) satisfies 
\begin{equation}\label{zxp0913}
\begin{aligned}
& \left |   \log  dist \, (C_1^\Gamma, C_2^\Gamma)   \right |^{-\frac{2\theta s^2}{(s+2)(n+ns+2s)}}  \tau^{4(ns+2n+s+4)/((n+2)(s+2))} \\
&\qquad  \leq C_{10}\left |   \log  dist \, (C_1^\Gamma, C_2^\Gamma)   \right |^{-\frac{\theta s^2}{(s+2)(n+ns+2s)}}.
\end{aligned}
\end{equation}
By replacing (\ref{zxp091})-(\ref{zxp0913}) into (\ref{x3}), we obtain
\begin{equation}\label{msdbfdhjsfaen}
\left \| \widetilde{q} \right \|_{L^{2}(\mathbb{R}^n)}\leq C_{11}   \left |   \log  dist \, (C_1^\Gamma, C_2^\Gamma)   \right |^{-\frac{\theta s^3}{(s+2)(n+ns+2s)(2ns+4n+s^2+4s)}}.
\end{equation}
Now since $n\geq 3$, we have
\[
 n + \frac{n}{2}+1\leq 2n, \qquad 5n + \frac{9}{4} \leq 6n,
\]
an since  $s\in \left( 0, 1/2 \right)$, we get 
\begin{align*}
&\frac{\theta s^3}{(s+2)(n+ns+2s)(2ns+4n+s^2+4s)} \\
&\qquad \qquad  \geq  \dfrac{2\theta s^3}{5(n + \frac{n}{2}+1 )( 5n+ \frac{9}{4})}\geq \frac{\theta s^3}{30 n^2}.
\end{align*}
By replacing this inequality into (\ref{msdbfdhjsfaen}), we get
\[
\left \| \widetilde{q} \right \|_{L^{2}(\mathbb{R}^n)}\leq C_{11}   \left |   \log  dist \, (C_1^\Gamma, C_2^\Gamma)   \right |^{- \frac{\theta s^3}{30 n^2}}.
\]
Thus, we conclude the proof by taking $\lambda=\theta/{30}$, $C$ as follows
\[
C=\max\left\{(5\kappa\tau_0)^{\frac{2(n+ns+2s)(2ns+4n+s^2+4s)}{\theta s^2 (n+s)}}, C_{11} \right\}
\]
and finally taking into account that
\[
\left \| q_1-q_2 \right \|_{L^{2}(\Omega)}\leq \left \| \widetilde{q} \right \|_{L^{2}(\mathbb{R}^n)}.
\]

\section{Identifiability for the magnetic field and the electric potential}
\subsection{Proof of Theorem \ref{identibiafgteru}} We only give the main ideas to prove the identifiability for the magnetic field and electric potential since it is just the qualitative version of what we have proved in the previous sections. We consider $\rho_1$ and $\rho_2$ given by (\ref{jkw}). Now let $U_1, U_2\in H^1(\Omega)$ satisfying $\mathcal{L}_{A_1,q_1}U_1=0$ with $U_1|_{\Gamma_0}=0$ and $\mathcal{L}_{\overline{A_2},\overline{q_2}}U_2=0$ with $U_2|_{\Gamma_0}=0$. The existence of such functions are given by Proposition \ref{construcsolutions}, except that we replaced the estimates (\ref{prunoz})-(\ref{prcuatroz}) by the estimates from Remark \ref{remesbaxzq}, (\ref{prunozxa})-(\ref{prcuatrozxa}). Thus, since $C_{A_1, q_1}^\Gamma= C_{A_2, q_2}^\Gamma$, Corollary \ref{coAlid} ensures that
\[
 \int_{\Omega}  (A_1-A_2)\cdot \left( DU_1\overline{U_2} +  U_1 \overline{DU_2} \right) + \left( A_1^2-A_2^2+q_1-q_2 \right) U_1\overline{U_2} =0.
\]

From this integral identity we can prove the identifiability for the magnetic potentials, following the proof of the Proposition \ref{ggggg5}, applying the Riemann--Lebesgue lemma to the function $\chi_\Omega(A_1-A_2)e^{\Phi_1^\sharp+ \overline{\Phi_2^{\sharp^*}} }$ to estimate (\ref{lilax9})-(\ref{lilax10}) and taking into account Lemma \ref{KyUh} in order to remove the term $e^{\Phi_1+\overline{\Phi_2}}$ in the left-hand side of (\ref{removing}). At this point, since the Fourier transform is analytic, Proposition \ref{ggggg5} and (\ref{final2zq}) give us the following equality in the sense of the distributions in $\mathbb{R}^n$:
\[
d(\chi_{\Omega\cup\Omega^*}\widetilde{A_1})= d(\chi_{\Omega\cup\Omega^*}\widetilde{A_2}),
\]
which imply that $dA_1=dA_2$ in $\Omega$.\\

The proof of the identifiability for the electric potential is as follows. We consider the Hogde decomposition for $\chi_{\Omega\cup\Omega^*}(\widetilde{A_1}- \widetilde{A_2})$ in a ball $B$ satisfying $\Omega\cup\Omega^*\subset\subset B$. We also take into account the estimates from Remark \ref{remesbaxzq} and the Riemann--Lebesgue lemma applied to the function $\chi_{\Omega\cup\Omega^*}(\widetilde{q_1}-\widetilde{q_2})$. Finally, since the Fourier transform is analytic, Proposition \ref{electyc} and (\ref{final21zswq}) imply that
\[
\chi_{\Omega\cup\Omega^*}\widetilde{q_1}=\chi_{\Omega\cup\Omega^*}\widetilde{q_2},
\]
so we have $q_1=q_2$ in $\Omega$.

\section*{Acknowledgements} The author would like to thank Alberto Ruiz  for the very nice  discussions. For his support and encouragement during the preparation of this paper. I would also want to thank Pedro Caro for fruitful conversations, suggestions and comments which allowed us to improve our original results. I would also want to thank Mikko Salo for his hospitality during my research stay in Jyv\"askyl\"a and also for several nice conversations. This paper is part of my PhD. dissertation and it is supporting by the Project MTM$2011-28198$ of Ministerio de Econom\'ia y Competividad de Espa\~na.

\end{document}